\documentclass[12pt,a4paper]{amsart}
\usepackage{a4wide}
\usepackage{amsmath, amssymb, amsfonts,enumerate}
\usepackage[all]{xy}
\usepackage{amscd}
\usepackage{comment}
\usepackage{mathtools}
\usepackage{tikz} 
\usepackage{hyperref}
\usepackage{tikz}
\usetikzlibrary{arrows.meta, automata, positioning, shadows}

\newtheorem{theorem}{Theorem}[section]
\newtheorem{lemma}[theorem]{Lemma}
\newtheorem{corollary}[theorem]{Corollary}
\newtheorem{proposition}[theorem]{Proposition}

 \theoremstyle{definition}
 \newtheorem{definition}[theorem]{Definition}
 \newtheorem{remark}[theorem]{Remark}

 \newtheorem{example}[theorem]{Example}

\newtheorem{question}{Question}

\newtheorem*{problem}{Problem}

\numberwithin{equation}{section}
\newcommand {\N}{\mathbb{N}} 
\newcommand {\Z}{\mathbb{Z}} 

\newcommand{\CC}{\mathcal{C}}

\newcommand{\KK}{\mathcal{K}}

\newcommand{\OO}{\mathcal{O}}

\DeclareMathOperator{\lcm}{lcm}
\DeclareMathOperator{\Per}{Per}

\DeclareMathOperator{\NW}{NW}

\DeclareMathOperator{\End}{End}

\DeclareMathOperator{\Id}{Id}

\DeclareMathOperator{\pref}{pref}
\DeclareMathOperator{\suff}{suff}
\DeclareMathOperator{\CBrank}{rk_{\scriptstyle  \text{\rm CB}}}

\begin{document}
\title{On surjunctive and injunctive subshifts of finite type}
\author{Tullio Ceccherini-Silberstein}
\address{Dipartimento di Ingegneria, Universit\`a del Sannio, I-82100 Benevento, Italy}
\address{Istituto Nazionale di Alta Matematica ``Francesco Severi'', I-00185 Rome, Italy}
\email{tullio.cs@sbai.uniroma1.it}
\author{Michel Coornaert}
\address{Universit\'e de Strasbourg, CNRS, IRMA UMR 7501, F-67000 Strasbourg, France}
\email{michel.coornaert@math.unistra.fr}
\author{Ville Salo}
\address{Department of Mathematics and Statistics, University of Turku, 20014 Turku, Finland}
\email{vosalo@utu.fi}
\dedicatory{To Doug Lind on his 80th birthday}
\begin{abstract}
A dynamical system is said to be surjunctive  if every injective  endomorphism of the system is surjective and
it is said to be injunctive if every surjective endomorphism is injective. 
An endomorphism of a dynamical system is called pre-injective if its restriction to every homoclinicity class of the phase space is injective. 
One says that a dynamical system has the Moore property if every surjective endomorphism of the system is pre-injective
and that it has the Myhill property if every pre-injective endomorphism is surjective. 
We give characterisations of surjunctivity and injunctivity for  $\Z$-subshifts of finite type in terms of their irreducible components
and their Cantor-Bendixson decomposition.
We also prove that a $\Z$-subshift of finite type is surjunctive if and only if it has the Moore property
and that every injunctive $\Z$-subshift of finite type is surjunctive.  
This implies in particular that a $\Z$-subshift of finite type has the Moore property whenever it  has the Myhill property.
\end{abstract}
\date{\today}
\maketitle

\tableofcontents

\section{Introduction}

A self-map of a finite set is injective if and only if it is surjective. 
In a concrete category, one says that an object is \emph{surjunctive} if every injective endomorphism is surjective~\cite{csc-cat}, and \emph{injunctive} if every surjective endomorphism is injective. 
These are natural finiteness properties of objects, and besides the category of finite sets, they hold in several other settings.
For instance, in the category of vector spaces over a field, with linear maps as morphisms, every finite-dimensional object is surjunctive and injunctive.
The Ax-Grothendieck theorem~\cite{ax-elementary} implies surjunctivity for affine varieties over an algebraically closed field, with polynomials as 
morphisms.
Surjunctivity plays an interesting role in the theory of cellular automata. 
It is known that for a large class of groups $G$, called \emph{sofic groups}, for any finite alphabet $A$, all injective cellular automata $f \colon A^G \to A^G$ (continuous maps commuting with the shift action of $G$) are 
surjective~\cite{gromov-esav}, \cite{weiss-sgds}.  
This can be stated by saying that full shifts over sofic groups are surjunctive in the category of subshifts.
This result has been generalized in several ways. In particular, it is known that  non-wandering Smale systems are 
surjunctive~\cite{csc-goe-smale}.
\par
In the present paper, we fully characterize surjunctivity and injunctivity in the category of subshifts of finite type over the group $\Z$
in terms of the poset structure of their irreducible components and  their
Cantor-Bendixson decomposition.
It turns out that the surjunctive objects in this category are exactly those which satisfy the Moore property (which is a certain weakening of the injunctivity property).
 In addition,   the Myhill property (which is a certain strengthening of the surjunctivity property)
 implies the Moore property. 
\par
We now proceed to more precise definitions and statements.
Let $A$ be a finite set. 
Consider the set $A^{\Z}$ consisting of all sequences $x = (x_i)_{i \in \Z}$ such that  $x_i \in A$ for all $i \in \Z$.
Equip $A^{\Z}$ with its \emph{prodiscrete topology}, i.e., the product topology obtained by taking the discrete topology on every factor $A$ of $A^{\Z} = \prod_\Z A$.
The \emph{shift map} on $A^{\Z}$ is the homeomorphism $\sigma \colon A^{\Z} \to A^{\Z}$ defined by $(x_i)_{i \in \Z} \mapsto (x_{i + 1})_{i \in \Z}$.
A closed $\sigma$-invariant subset $X \subset A^{\Z}$ is called a \emph{subshift}.
A subshift $X \subset A^{\Z}$ is said to be of \emph{finite type} if there exist an integer $n \geq 0$ and a subset $S \subset A^n$ such that $X$ consists of all $x \in A^{\Z}$ 
that satisfy $(x_{i + 1}, x_{i + 2}, \dots ,x_{i + n}) \in S$ for all $i \in \Z$.
\par
Let $X \subset A^{\Z}$ be a subshift.
An \emph{endomorphism} of  $X$
is a continuous  map $\tau \colon X \to X$ which commutes with the shift map $\sigma$.
These morphisms are also called \emph{sliding-block codes}~\cite{lind-marcus-second} or
\emph{one-dimensional cellular automata}~\cite{csc-cag2, csc-ecag}.
\par
One says that the  subshift $X$ is \emph{surjunctive}~\cite{gottschalk} if every injective endomorphism of $X$ is surjective
and that it is  \emph{injunctive}, or \emph{coalescent} (cf.~\cite{auslander-endomorphisms}, \cite{downarowicz}), if every surjective endomorphism of $X$ is injective.
Two sequences  $x, y \in X$ are called \emph{homoclinic} (or \emph{almost equal}) 
if  the set of $i \in \Z$ such that $x_i \not= y_i$ is finite. 
Homoclinicity is an equivalence relation on $X$.
One says that an endomorphism of $X$ is \emph{pre-injective}~\cite{gromov-esav} if its restriction to every homoclinicity class of $X$ is injective.
Injectivity trivially implies pre-injectivity but the converse implication does not hold in general.
One says that $X$  has the \emph{Moore property}, or that $X$ is \emph{Moore}~\cite{moore}, if every surjective endomorphism of $X$   is pre-injective. One says that
$X$  has the \emph{Myhill property}, or that $X$ is \emph{Myhill}~\cite{myhill},  if every pre-injective endomorphism of $X$   is surjective.
Finally, one says that $X$  satisfies the \emph{Garden of Eden theorem} if $X$ is both Moore and Myhill.
\par
Subshifts over $\Z$ form a concrete category where the morphisms are the continuous shift-commuting maps. 
We note that in this category, surjectivity and injectivity correspond exactly to the abstract properties of epicness and monicness (see \cite[Table~1]{salo-torma_2015}, 
where $K2$ is the category of $\Z$-subshifts of finite type), so surjunctivity (resp., injunctivity) can equivalently be stated as the property that monic endomorphisms are epic (resp., epic endomorphisms are monic).
\par
In the present paper, we show that   surjunctivity is equivalent to the Moore property for subshifts of finite type.
We prove it by showing that each of these properties is equivalent to the same condition   
on  the poset structure of the irreducible components of the subshift (see Section~\ref{sec:irred-compo}  for precise definitions).
Actually, we shall establish the following.

\begin{theorem}
\label{t:char-surj-sft}
Let $A$ be a finite set and let $X \subset A^{\Z}$ be a subshift of finite type.
Then the following conditions are equivalent:
\begin{enumerate}[\rm (a)]
\item 
$X$ is surjunctive;
\item
$X$ is Moore;
\item 
every irreducible component of $X$ is extremal and
every infinite irreducible component of $X$ is isolated;
\item
$X$   has Cantor-Bendixson rank at most $2$
and the perfect kernel  of $X$ is the union of the isolated infinite irreducible components of $X$;
\item
$X$ is the disjoint union of a non-wandering subshift of finite type and of a countable subshift of finite type having Cantor-Bendixson rank at most $2$.
\end{enumerate}
\end{theorem}

 Since injectivity implies pre-injectivity,
every subshift having the Myhill property is surjunctive.
Thus, as an immediate consequence of Theorem~\ref{t:char-surj-sft}, we get the following.

\begin{corollary}
\label{c:main-sft}
Let $A$ be a finite set and let $X \subset A^{\Z}$ be a subshift of finite type.
If $X$ has the Myhill property then it has the Moore property.
In other words, if $X$ has the Myhill property then it satisfies the Garden of Eden theorem.
\end{corollary}

The hypothesis that the subshift $X$ is of finite type cannot be removed neither from Theorem~\ref {t:char-surj-sft} nor from Corollary~\ref{c:main-sft}.
Indeed, Fiorenzi proved in ~\cite{fiorenzi-sofic} that the even subshift 
(i.e., the subshift  consisting of all sequences of $0$s and $1$s satisfying the condition that there is always  an even number of $0$s between two $1$s)
is Myhill but not Moore.
Note that the even subshift is a sofic subshift (i.e., a factor of a subshift of finite type) and topologically mixing.
 \par
 We shall also establish the following result.

\begin{theorem}
\label{t:finite-extremal-components}
Let $A$ be a finite set and let $X \subset A^{\Z}$ be a subshift of finite type.
Then the following conditions are equivalent:
\begin{enumerate}[\rm (a)]
\item
 $X$ is injunctive;
 \item
every irreducible component of $X$ is finite and extremal;
\item
$X$ is countable with Cantor-Bendixson rank at most $2$.
\end{enumerate}
\end{theorem}

From Theorem~\ref{t:char-surj-sft} and Theorem~\ref{t:finite-extremal-components}, we deduce that a subshift of finite type is surjunctive if and only if it is the disjoint union of a non-wandering subshift of finite type and of an injunctive subshift of finite type.
\par
Condition (c) in Theorem~\ref{t:finite-extremal-components}
is purely topological.
We deduce that being injunctive
is a topological invariant for subshifts of finite type.
Thus, we have the following.

\begin{corollary}
\label{c:surjunctive-top-inv}
Let $A$ and $B$ be  finite sets and let $X \subset A^{\Z}$ and $Y \subset B^{\Z}$ be subshifts of finite type.
Suppose that the topological spaces underlying $X$ and $Y$ are homeomorphic.
Then $X$ is injunctive if and only if $Y$ is injunctive.
\end{corollary}

It follows from our results that  subshifts of finite type satisfy the following:
  \[
\begin{matrix}
 &  &\text{injunctive} & & \\
  &  & \Downarrow &  & \\
\text{Myhill}   &  \implies & \text{surjunctive} & \iff & \text{Moore} 
  \end{matrix}
  \]
We shall give examples showing that each of the two one-sided implications appearing in the diagram above is strict (see~Example~\ref{ex:type-2} and Example~\ref{ex:type-3}).
There exist even surjunctive (and hence Moore) subshifts of finite type that are neither Myhill nor injunctive (see Example~\ref{ex:type-4}).
\par
The paper is organized as follows.
In Section~\ref{sec:background}, we introduce the necessary terminology and collect some background material on directed graphs, dynamical systems, symbolic dynamics, and Cantor-Bendixson theory. The reader already familiar with these preliminaries may safely skip this section.
Section~\ref{sec:irred-compo} describes the poset $\CC$ of irreducible components of a subshift of finite type $X$.
We prove that an irreducible component of a subshift of finite type $X$ is isolated in the poset of all irreducible components of $X$ if and only if it is open in $X$.
We show that every endomorphism $\tau$ of $X$ naturally induces a poset endomorphism $\rho_\tau$ of $\CC$ and
that  $\rho_\tau$ is a poset automorphism provided that $\tau$ is surjective or injective.
We also give a characterization of isolated configurations  in terms of their initial and terminal irreducible components.
In Section~\ref{sec:markers}, we use the technique of
markers to construct endomorphisms of subshifts of finite type with an infinite irreducible component that are surjective but not injective.
In Section~\ref{sec:construction-tau-C}, given an irreducible component $C$ of a vertex subshift $X$,
we construct an injective endomorphism  $\tau_C \colon X \to X$.
The fact that it may fail to be surjective (depending on the structure of the poset of irreducible components if $X$) is used in Section~\ref{sec:char-surj-sft} for proving the properties of surjunctive subshifts of finite type stated in Theorem~\ref{t:char-surj-sft}.
In Section~\ref{sec:construction-tauC}, given an infinite irreducible component $C$ of a vertex subshift $X$
which is minimal in the poset of all irreducible components of $X$, we construct a surjective endomorphism $\tau^C \colon X \to X$.
The fact that it may fail to be pre-injective (depending on the structure of the poset of irreducible components if $X$) is used in Section~\ref{sec:char-surj-sft} for proving the properties of Moore subshifts of finite type stated in Theorem~\ref{t:char-surj-sft}.
The proofs of Theorem~\ref{t:char-surj-sft} and Theorem~\ref{t:finite-extremal-components}
are  given in Section~\ref{sec:proofs}.
Section~\ref{sec:examples}  contains a series of examples.
Subshifts of finite type are divided into five classes
according to whether or not they satisfy surjunctivity, injunctivity, Moore, and Myhill.
We give  explicit examples for  each of these classes.
We also show that  surjunctivity (resp.~Moore, resp.~Myhill) is not  a topological invariant for subshifts of finite type and that injunctivity  is not a topological invariant for sofic subshifts.
In the final section, we have  collected some questions that remain open.\\

\noindent
{\bf Acknowledgments.} We express our gratitude to the anonymous referee for the very careful reading of the manuscript and for
pointing out to us several inaccuracies.

\section{Preliminaries}
\label{sec:background}

In this section, we fix notation and collect the necessary material regarding
posets, words, directed graphs, dynamical systems, symbolic dynamics, and Cantor-Bendixson theory.
Some proofs are given. More details can be found in \cite{brin-stuck}, \cite{csc-cag2}, \cite{csc-ecag}, \cite{kechris}, \cite{kitchens}, \cite{lind-marcus-second}, and \cite{walters-intro-book}. 
Readers already familiar with these preliminaries may safely skip this section. 

\subsection{General terminology}
We write $\Z$ for the set of integers and $\N$ for the set of non-negative integers (thus, $0 \in \N$).
\par
Given a set  $X$, we write $|X|$ for the cardinality of $X$. 
We denote by $\Id_X$ the identity map on $X$, i.e., the map
$\Id_X \colon X \to X$ given by $\Id_X(x) \coloneqq x$ for all $x \in X$.
A \emph{partition} of $X$ is a set of pairwise disjoint non-empty subsets of $X$ whose union is $X$.
\par
Given a map $f \colon X \to Y$ between sets and subsets $X' \subset X$, $Y' \subset Y$, we define the subsets $f(X') \subset Y$ and $f^{-1}(Y') \subset X$ by
$f(X') \coloneqq \{f(x) : x \in X'\}$ and $f^{-1}(Y') \coloneqq \{x \in X : f(x) \in Y'\}$.
\par
Given maps $f \colon X \to Y$ and $g \colon Y \to Z$, we denote by $g \circ f$ the \emph{composite} of $f$ and $g$, i.e., the map
$g \circ f \colon X \to Z$ defined by $(g \circ f)(x) \coloneqq g(f(x))$ for all $x \in X$.
\par
Given a map $f \colon X \to X$, we set $f^0 \coloneqq \Id_X$ and define inductively $f^n$ for $n \in \N$ by setting $f^{n + 1} \coloneqq f \circ f^n$. 
If  $f$ is bijective, we write $f^{-1}$ for the inverse of $f$ and 
$f^{-n} \coloneqq (f^{-1})^n$ for every $n \in \N$.

  \subsection{Posets}
Let $X$ be a set and let $\preceq$ be a binary relation on $X$.
One says that $\preceq$ is \emph{reflexive} if one has $x \preceq x$ for all $x \in X$.
One says that $\preceq$ is \emph{antisymmetric} if $x \preceq y$ and $y \preceq x$ implies $x = y$ for all $x,y \in X$.
One says that $\preceq$ is \emph{transitive} if $x \preceq y$ and $y \preceq z$ imply $x \preceq z$ for all $x,y,z \in X$.
One says that $\preceq$ is a \emph{partial ordering} on $X$ if $\preceq$ is reflexive, antisymmetric, and transitive.
If $\preceq$ is a partial ordering on a set $X$, one says that $(X,\preceq)$ is a \emph{partially ordered set} or a \emph{poset}.
\par
Let $(X,\preceq)$ be a poset.
We write $x \prec y$ to mean $x \preceq y$ and $x \not= y$.
One says that an element $x \in X$ is \emph{minimal} (resp.~\emph{maximal}) if 
there exists no $y \in X$ such that $y \prec x$ 
(resp.~$x \prec y$).
One says that $x \in X$ is \emph{extremal} if $x$ is minimal or maximal.
One says that $x \in X$ is \emph{isolated} if $x$ is both minimal and maximal.
One says that  $x \in X$ is an \emph{intermediate} element if there exist $y, z \in X$  such that
$y \prec x \prec z$.
Observe that $x$ is intermediate if and only if $x$ is not extremal.
One says that $x,y \in X$ are \emph{adjacent} if either $x \prec y$ and there is no $z \in X$
such that $x \prec z \prec y$, or $y \prec x$ and there is no $z \in X$ such that
$y \prec z \prec x$.
One says that a subset $\Gamma \subset X$ is a \emph{chain} if the restriction of $\preceq$ to $\Gamma$ is a total ordering, i.e., for all $x,y \in \Gamma$, one has $x \preceq y$ or $y \preceq x$.
\par 
If $(X,\preceq)$ and $(X',\preceq')$ are posets, one says that a map $\varphi \colon X \to X'$ is \emph{order-preserving} if one has $\varphi(x) \preceq' \varphi(y)$ for all $x,y\in X$ such that  $x \preceq y$.
The \emph{category of posets} is the category whose objects are posets and whose morphisms are order-preserving maps between them. 
We shall denote by $\End(X,\preceq)$ the endomorphism monoid of a poset $(X,\preceq)$.

\subsection{Words}
Let $A$ be a set, called the \emph{alphabet}, whose elements are called the \emph{letters}. 
Denote by $A^*$ the free monoid based on  $A$.
Every element  $w \in A^*$ can be uniquely written as a word  $w = a_1a_2\cdots a_n$, 
where $n  \geq 0$  and $a_i \in A$ for all $i \in \{1\dots,n\}$.
We then write $|w| \coloneqq n$.
Denoting by $A^n$ the set of words  $w \in A^*$ such that $|w| = n$, we thus have $A^* = \bigcup_{n \geq 0} A^n$. 
The monoid operation on $A^*$ is the concatenation of words.
Thus, if $w \coloneqq a_1 \cdots a_n \in A^n$ and $w' \coloneqq a_1'\cdots a_m' \in A^{m}$,
then $w w' = a_1 \cdots a_n a_1' \cdots a_m' \in A^{n + m}$.
The identity element of $A^*$ is the \emph{empty word}, that is, the unique element of $A^0$.
Given a word $w = a_1\cdots a_n \in A^n$,  $n \geq 1$, and  an integer $k \in \{1,\dots,n\}$,
the \emph{$k$-prefix} (resp.~\emph{$k$-suffix}) of  $w$ is the word
$\pref_k(w) \coloneqq a_1  \cdots a_k \in A^k$ (resp.~$\suff_k(w) \coloneqq a_{n - k + 1} \cdots a_n \in A^k$).
The element $\pref_1(w) = a_1 \in A$ (resp.~$\suff_1(w) = a_n \in A$) is the \emph{initial} (resp.~\emph{terminal}) letter of $w$.
 \par
A \emph{right-infinite} (resp.~\emph{left-infinite}, resp.~\emph{bi-infinite}) word over $A$ is an element of $A^{\N}$ (resp.~$A^{-\N}$, resp.~$A^{\Z}$), that is,
a sequence $w = (a_n)_{n \in \N}$ (resp.~$w = (a_n)_{n \in -\N}$, resp.~$w = (a_n)_{n \in \Z}$), where $a_n \in A$ for all $n \in \N$ (resp.~$ n \in - \N$, resp.~$ n \in \Z$).
We shall also use the notation $w = a_0 a_1 a_2 \cdots$ (resp.~$w = \cdots a_{-2} a_{-1} a_0$, resp.~$w= \cdots a_{-1} .a_0 a_1 \cdots$) to denote $w$.

\subsection{Directed graphs}
A \emph{directed graph} is a pair $G = (V,E)$, where $V$ is a set and $E$ is a subset of $V^2$.
The set $V$ is called the set of \emph{vertices} of $G$ and $E$ is called its set of \emph{edges}.
A directed graph is said to be \emph{finite} if its set of vertices is finite.
\par
Let $G = (V,E)$ be a directed graph.
\par
Given a subset $W \subset V$, the pair $G_W = (W,F)$, where $F \coloneqq E \cap W^2$, is also a directed graph.
One says that $G_W$ is the \emph{subgraph induced by}  $W$.
\par 
Given vertices $v,v' \in V$ and an integer $n \geq 1$, a \emph{path} of \emph{length} $n$ from $v$ to $v'$ in $G$ is a word
$\pi = v_0v_1\cdots v_n \in V^{n + 1}$ such that $v_0 = v$, $v_n = v'$, and $v_i v_{i + 1} \in E$ for all $i \in \{0,1,\dots,n-1\}$.
Note that the length of $\pi$ is equal to $|\pi| - 1$ and that, by definition, the length of a path is always positive.
\par
Let $\pi = v_0v_1 \cdots v_m \in V^{m+1}$ be a path in $G$ of length $m$ from a vertex $v_0$ to a vertex $v_m$ and let 
$\pi' = v_mv_{m+1} \cdots v_{m+n} \in V^{n+1}$ be a path in $G$ of length $n$ from $v_m$ to a vertex $v_{m + n}$.
The \emph{composite}  of $\pi$ and $\pi'$
is the path in $G$ of length $m + n$ from $v_0$ to $v_{m + n}$ defined by   $\pi \# \pi' \coloneqq v_0v_1 \cdots v_m v_{m + 1} \cdots v_{m+n} \in V^{m+n+1}$. 
We define a relation $\preceq$ on $V$ by writing $v \preceq v'$ if there exists a path from $v$ to $v'$ in $G$.
Observe that  $\preceq$ is transitive on $V$.
Indeed, if $\pi$ is a path from $v$ to $v'$ and $\pi'$ is a path from $v'$ to $v''$, then
$\pi \# \pi'$ is a path from $v$ to $v''$.  
A path from a vertex $v$ to itself is called a \emph{cycle} based at $v$.
A cycle of length $1$ is called a \emph{self-loop}.
Note that if $\pi = v_0 v_1 \cdots v_{n-1} v_0$ is a cycle based at $v_0$ of length $n$ and $k \geq 1$ is an integer, then
 $\pi^{k\#} = (v_0 v_1 \cdots v_{n - 1})^k v_0$ is a cycle based at $v_0$ of length $kn$.
One says that a cycle $v_0v_1\cdots v_{n-1} v_0$ based at a vertex $v_0$ is \emph{simple} if the vertices
$v_0,v_1,\dots,v_{n - 1}$ are all distinct.
\par
A vertex $v \in V$ is called \emph{recurrent} if there exists a cycle (of positive length) based at $v$.
A vertex which is not recurrent is said to be \emph{transient}.
We denote by $R$ (resp.~$ T$) the set of all recurrent (resp.~transient) vertices of $G$.
Note that the restriction of $\preceq$ to $R$ is reflexive.
We introduce a relation $\sim$ on $R$ by writing $v \sim v'$ if there exist a path from $v$ to $v'$ and a path from $v'$ to $v$
(i.e., if $v \preceq v'$ and $v' \preceq v$).
It is clear that $\sim$ is an equivalence relation on $R$.
The elements of the quotient set $\KK \coloneqq R/\sim$, i.e., the equivalence classes of $\sim$ in $R$, are called the \emph{path-components} of $G$.
The relation $\preceq$ naturally induces a relation, also denoted by abuse $\preceq$, on $\KK$. 
Given $K,K' \in \KK$, we have $K \preceq K'$ if and only if $v \preceq v'$ for some (or, equivalently, all) $v \in K$ and $v' \in K'$.
Clearly, $\preceq$ is reflexive, antisymmetric, and transitive on $\KK$.
In other words, $(\KK,\preceq)$ is a poset.
One says that the graph $G$ is \emph{strongly connected} if $V \not= \varnothing$ and, for all $v,v' \in V$, there exists a path in $G$ from $v$ to $v'$.
Thus, $G$ is strongly connected if and only if $\KK = \{V\}$.

\subsection{Dynamical systems}
A \emph{dynamical system} is a pair $(X,f)$, where
$X$ is a compact metrizable space  and  
 $f \colon X \to X$ is  a homeomorphism.
 The space $X$ is called the \emph{phase space} of the dynamical system.
 When the homeomorphism $f$ is clear from the context, we shall sometimes simply write $X$ to designate the dynamical system  $(X,f)$.
\par 
 Let $(X,f)$ be a dynamical system.
 The \emph{orbit} of a point $x \in X$ is the subset $\OO(x) \subset X$ given by
 $\OO(x) \coloneqq \{f^n(x) : n \in \Z\}$.
 A subset $Y \subset X$ is said to be $f$-\emph{invariant}, or simply \emph{invariant},   if $f(Y) = Y$.
If $Y \subset X$ is an invariant subset, we denote by $f\vert_Y \colon Y \to Y$ the restriction of $f$ to 
$Y$.
When $Y \subset X$ is a closed invariant subset,  the pair $(Y,f|_Y)$ is also a dynamical system
and one says that $(Y,f|_Y)$ is a \emph{subsystem} of $(X,f)$.
We shall write such a subsystem  $(Y,f)$, or even simply $Y$,  if there is no risk of confusion.
\par
A point $x \in X$ is said to be \emph{periodic} if there exists an integer $n \geq 1$ such that $f^n(x) = x$. Such an integer $n$ is then called a \emph{period} of $x$.
The point  $x$ is periodic if and only if  $\OO(x)$ is finite.
If $x \in X$ is a periodic point, the least integer $n \geq 1$ such that $f^n(x) = x$ is called the \emph{minimal period} of $x$.
Observe that if $x \in X$ is periodic with minimal period $n$ then $\OO(x) = \{x,f(x) \dots,f^{n-1}(x)\}$ is finite with cardinality $n$.
We shall denote by $\Per(X)$  the set of all periodic points of $(X,f)$.
\par
A point $x \in X$ is said to be \emph{non-wandering} if for every neighborhood $U$ of $x$, there exists an integer $n \geq 1$ such that $U \cap f^n(U) \not= \varnothing$.
The set $\NW(X)$ of all non-wandering points of $f$ is a closed invariant subset of $X$.
Observe that $\Per(X) \subset \NW(X)$.
\par
One says that the dynamical system $(X,f)$ is \emph{non-wandering} if every point of $X$ is non-wandering, i.e., $\NW(X) = X$.
One says that $(X,f)$ is \emph{topologically mixing} if $X$ is non-empty and, for all non-empty open subsets $U,V \subset   X$, there exists an integer  $N \geq 1$ such that $U \cap f^n(V) \not= \varnothing$ for all  $n \geq N$.
One says that $(X,f)$ is \emph{irreducible}, or \emph{one-sided topologically transitive},  if $X$ is non-empty and, for all non-empty open subsets $U, V \subset X$, there exists $n \geq 1$ such that $U \cap f^n(V) \not= \varnothing$.
The dynamical system $(X,f)$ is irreducible if and only if there exists a point $x \in X$ whose forward orbit
$\OO^+(x) \coloneqq \{f^n(x) : n \geq 1\}$ is dense in $X$.
One says that $(X,f)$ is \emph{topologically transitive}   if $X$ is non-empty and, for all non-empty open subsets $U, V \subset X$, there exists $n \in \Z$ such that $U \cap f^n(V) \not= \varnothing$.
The dynamical system $(X,f)$ is topologically transitive if and only if there exists a point $x \in X$ whose  orbit
$\OO(x)$ is dense in $X$.
We have the following   implications:
  \[
\begin{matrix}
  \text{topologically mixing } &\implies &
  \text{irreducible } & \implies &
  \text{non-wandering} \\
  & & \Downarrow & & \\
  & & \text{topologically transitive} & &
  \end{matrix}
  \]
Each of these implications is strict.
It can be shown  that a dynamical system is irreducible if and only if it is non-wandering and topologically transitive
\cite[Theorem~5.10]{walters-intro-book}.
Moreover, if $X$ has no isolated points, then $(X,f)$ is irreducible if and only if it is topologically transitive
\cite[Proposition~2.2.2]{brin-stuck}.
\par
Let $(X,f)$ be a dynamical system.
Given $x \in X$,
the $\alpha$-\emph{limit set} (resp.~$\omega$-\emph{limit set}) of $x$
is the set $\alpha(x) = \alpha_{f}(x)$ (resp.~$\omega(x) = \omega_{f}(x)$) consisting of all cluster points of the set
$\{f^{-n}(x) : n \in \N\}$ (resp.~$\{f^n(x) : n \in \N\}$) in $X$.
Thus, a point $y \in X$ is in $ \alpha(x)$ (resp.~$\omega(x)$)
if and only if 
there exists a strictly increasing sequence $(n_k)_{k \in \N}$ of non-negative integers such that
the sequence $(f^{-n_k}(x))_{k \in \N}$ (resp.~$(f^{n_k}(x))_{k \in \N}$) converges to $y$. 
The sets $\alpha(x)$ and $\omega(x)$ are non-empty closed invariant subsets of $\NW(X)$
and one has
\begin{equation}
\label{e:alpha-omega-inverse}
\alpha_{f}(x) = \omega_{f^{-1}}(x)
\text{ and }
\omega_{f}(x) = \alpha_{f^{-1}}(x)
\end{equation}
for all $x \in X$.

Observe that if $x \in \Per(X)$ then $\alpha(x) = \omega(x) = \OO(x)$.

\par
Let $(X,f)$ be a dynamical system.
Let $d$ be a compatible metric on $X$.
Two points $x, y \in X$ are said to be  \emph{unstably equivalent} (resp.~\emph{stably equivalent}, resp.~\emph{homoclinic})
if  $d(f^{n}(x),f^{n}(y)) \to 0$ as $n \to -\infty$ (resp.~$n \to  \infty$, resp.~$|n| \to \infty$).
Clearly, unstable equivalence (resp.~stable equivalence, resp.~homoclinicity)  is an equivalence relation on $X$.
By compactness of $X$, 
each of these equivalence relations  does not depend on the choice of the compatible metric $d$.
Given $x \in X$,
we denote by  $W_f^u(x)$ (resp.~$W_f^s(x)$, resp.~$W_f^h(x)$) 
 the unstable (resp.~stable, resp.~homoclinic) equivalence class of $x$.
 Observe that $W_f^h(x) = W_f^u(x) \cap W_f^s(x)$ and that $f^n(W_f^u(x)) = W_f^u(f^n(x))$ (resp.~$f^n(W_f^s(x)) = W_f^s(f^n(x))$, resp.~$f^n(W_f^h(x)) = W_f^h(f^n(x))$) 
 for all $x \in X$ and $n \in \Z$.
 Also, $W^u_{f^{-1}}(x) = W^s_f(x)$ and $W^s_{f^{-1}}(x) = W^u_f(x)$. 
 In particular,  $W^h_{f^{-1}}(x) = W^h_f(x)$.

\begin{proposition}
\label{p:limit-sets-homoclinic}
Let $(X,f)$ be a dynamical system and let $x,y \in X$.
Then the following hold:
\begin{enumerate}[\rm (i)]
\item
if $x$ and $y$ are unstably equivalent then $\alpha(x) = \alpha(y)$;
\item
if $x$ and $y$ are stably equivalent then $\omega(x) = \omega(y)$;
\item
if $x$ and $y$ are homoclinic then $\alpha(x) = \alpha(y)$ and $\omega(x) = \omega(y)$.
\end{enumerate}
\end{proposition}

\begin{proof}
Let $d$ be a compatible metric on $X$.
Suppose that $x$ and $y$ are unstably equivalent.
Let $z \in \alpha(x)$.
This means that
there exists a strictly increasing sequence $(n_k)_{k \in \N}$ of non-negative integers such that
the sequence $(f^{-n_k}(x))_{k \in \N}$ converges to $z$.
As $x$ and $y$ are unstably equivalent,
the sequence  $(d(f^{-n}(x),f^{-n}(y))_{n \in \N}$ converges to $0$.
Since
\[
d(z,f^{-n_k}(y)) \leq d(z,f^{-n_k}(x)) + d(f^{-n_k}(x),f^{-n_k}(y))
\]
by the triangle inequality, 
 we deduce that $d(z,f^{-n_k}(y))$ tends to $0$ as $k \to \infty$.
 Therefore $z \in \alpha(y)$.
 This shows that $\alpha(x) \subset  \alpha(y)$.
 By symmetry, we get $\alpha(y) \subset  \alpha(x)$, so that $\alpha(x) =  \alpha(y)$.
 This shows (i).
 \par
 The proof of (ii) is similar.
 \par
 Assertion (iii) follows from (i) and (ii) since homoclinicity implies both unstable and stable equivalence. 
\end{proof}

The \emph{category of dynamical systems} is the category whose objects consist of all dynamical systems $(X,f)$ and whose morphisms are defined as follows.
Given two dynamical systems $(X,f)$ and $(Y,g)$, a \emph{morphism} $\varphi \colon  (X,f) \to (Y,g)$ is a continuous map $\varphi \colon X \to Y$ such that $\varphi \circ f = g \circ \varphi$.
Thus, an endomorphism of  $(X,f)$ is a continuous map $\tau \colon X \to X$ which commutes with $f$, i.e., such that $\tau \circ f = f \circ \tau$.
We shall denote by $\End(X,f)$ the endomorphism monoid of a dynamical system $(X,f)$.
\par
Observe that if $\varphi \colon (X,f) \to (Y,g)$ is a morphism of dynamical systems, then $(\varphi(X),g)$ is a subsystem of $(Y,g)$.
A surjective morphism from $(X,f)$ to $(Y,g)$ is called a \emph{factor map}.
If there exists a factor map from $(X,f)$ to $(Y,g)$, one says that $(Y,g)$ is a \emph{factor} of $(X,f)$ and that $(X,f)$ is an \emph{extension} of $(Y,g)$.
One says that $(X,f)$ and $(Y,g)$ are \emph{topologically conjugate} if $(X,f)$ and $(Y,g)$ are isomorphic objects in the category of dynamical systems.
Thus, $(X,f)$ and $(Y,g)$ are topologically conjugate if and only if there exists a homeomorphism  $\varphi \colon X \to Y$ such that
$\varphi \circ f = g \circ \varphi$.

\begin{proposition}
\label{p:properties-endo}
Let $\varphi \colon (X,f) \to (Y,g)$ be a morphism of dynamical systems.
Then the following hold:
\begin{enumerate}[\rm (i)]
\item
one has $\varphi(\NW(X)) \subset \NW(Y)$;
\item
for every $x \in X$, one has
$\varphi(\alpha_f(x)) = \alpha_g(\varphi(x))$ and 
$\varphi(\omega_f(x)) = \omega_g(\varphi(x))$;
\item
for every $x \in X$, one has
$\varphi(W_f^u(x)) \subset W_g^u(\varphi(x))$, 
$\varphi(W_f^s(x)) \subset W_g^s(\varphi(x))$,
\text{ and }
$\varphi(W_f^h(x)) \subset W_g^h(\varphi(x))$.
\item
if $(X,f)$ is non-wandering (resp.~irreducible, resp.~topologically transitive, resp.~topologically mixing) then $(\varphi(X),g)$ is
non-wandering (resp.~irreducible, resp.~topologically transitive, resp.~topologically mixing).
\end{enumerate} 
\end{proposition}

\begin{proof}
Assertion (i) follows, e.g.,~from Proposition~2.2.(iii)  in~\cite{csc-goe-smale}.
\par 
To prove (ii),  we first observe that  $\varphi \circ f^n = g^n \circ \varphi$ for all $n \in \Z$ since $\varphi \circ f = g \circ  \varphi$.
Let now $x \in X$. 
Suppose first that $y \in \varphi(\alpha_f(x))$.
This means that there exist $z \in X$ with $y = \varphi(z)$ and a strictly increasing sequence $(n_k)_{k \in \N}$ of non-negative integers such that 
$f^{- n_k}(x)$ converges to $z$ as $k \to \infty$.
Then $g^{-n_k}(\varphi(x)) = \varphi(f^{-n_k}(x))$ converges to $\varphi(z) = y$ by continuity of $\varphi$.
Therefore $y \in \alpha_g(\varphi(x))$.
This shows that $\varphi(\alpha_f(x)) \subset  \alpha_g(\varphi(x))$.
\par
To prove the converse inclusion, suppose now $y \in \alpha_g(\varphi(x))$. 
This means that there exists a strictly increasing sequence $(n_k)_{k \in \N}$ of non-negative integers such that
the sequence $(g^{-n_k}(\varphi(x)))_{k \in \N}$ converges to $y$.
By compactness of $X$, there exists a subsequence $(m_i)_{i \in \N}$ of the sequence $(n_k)_{k \in \N}$ such that 
the sequence $(f^{-m_i}(x))_{i \in \N}$ converges to some point $z \in X$.
We then have $z \in \alpha_f(x)$.
On the other hand,
\[
y = \lim_{i \to \infty} g^{-m_i}(\varphi(x)) 
= \lim_{i \to \infty}  \varphi(f^{-m_i}(x)) = \varphi(z)
\]
by continuity of $\varphi$.
This shows that $y \in \varphi(\alpha_f(x))$ and completes the proof that
$\varphi(\alpha_f(x)) = \alpha_g(\varphi(x))$.
\par
The proof of the second equality in 
(ii) is similar 
(it can also be deduced from the first equality applied to $f^{-1}$ and the second equality in ~\eqref{e:alpha-omega-inverse}).
\par
Assertion (iii) follows, e.g.,~from Proposition~2.3 in~\cite{csc-goe-smale}.
\par
If $(X,f)$ is non-wandering, then $(\varphi(X),g)$ is itself non-wandering by (i).
On the other hand, if $U$ and $V$ are non-empty subsets of $\varphi(X)$, then 
$U' \coloneqq \varphi^{-1}(U)$ and $V' \coloneqq \varphi^{-1}(V)$ are non-empty open subsets of $X$.
Moreover, if $U' \cap f^n(V') \not= \varnothing$ for some $n \geq 1$, then
\[
\varnothing \not= \varphi(U' \cap f^n(V')) = \varphi(\varphi^{-1}(U) \cap f^n(\varphi^{-1}(V))) = \varphi(\varphi^{-1}(U) \cap \varphi^{-1}(g^n(V))) \subset U \cap g^n(V),
\]
so that $U^n \cap g^n(V) \not= \varnothing$.
We deduce that if $(X,f)$ is irreducible (resp.~topologically mixing) then $\varphi(X),g)$ is 
irreducible (resp.~topologically mixing).
This completes the proof of (iv).
    \end{proof}

In the trivial case when the phase space $X$ of a dynamical system $(X,f)$ is finite, the topology on $X$ is discrete and we have the following straightforward characterization of irreduciblity.

\begin{proposition}
\label{p:finite-irred-ds}
Let $(X,f)$ be a dynamical system
and suppose that  $X$ is finite with cardinality $|X| =  n$.
Then the following conditions are equivalent:
\begin{enumerate}[\rm (a)]
\item
$(X,f)$ is irreducible;
\item
$(X,f)$ is topologically transitive;
\item
every point of $X$ is periodic with minimal period $n$;
\item
there exists a point of $X$ which is periodic with minimal period $n$;
\item
$(X,f)$ is topologically conjugate to the dynamical system $(Y,g)$, where $Y \coloneqq \Z/n\Z$ and $g \colon Y \to Y$ is given by $g(y) \coloneqq y + 1$ for all $y \in Y$.
\end{enumerate}
\end{proposition}

One says that a dynamical system $(X,f)$  is \emph{surjunctive} if every injective endomorphism of $(X,f)$ is surjective
and that it is  \emph{injunctive} if every surjective endomorphism of $(X,f)$ is injective.
One says that an endomorphism $\tau$ of a dynamical system $(X,f)$ is \emph{pre-injective} if the restriction of $\tau$ to every homoclinicity class of $X$ is injective.
One says that a dynamical system $(X,f)$ has the \emph{Moore property}, or that it is \emph{Moore}, if every surjective endomorphism of $(X,f)$ is pre-injective.
One says that a dynamical system $(X,f)$ has the \emph{Myhill property}, or that it is \emph{Myhill}, if every pre-injective endomorphism of $(X,f)$ is surjective.
Since injectivity implies pre-injectivity, we always have the following implications:
\[
(\text{Myhill } \implies \text{ surjunctive}) \text{   and   } (\text{injunctive } \implies \text{ Moore}).  
\]
These implications are strict (see Example~\ref{ex:type-2} and Example~\ref{ex:type-4}).

Surjunctivity, injunctivity, and the Moore property are clearly finiteness conditions in the sense that they are satisfied by all dynamical systems with finite phase space.
By contrast, the Myhill property is not a finiteness condition (see Example~\ref{ex:type-2}).

\subsection{Shifts and subshifts}  
Let $A$ be a finite set, called the \emph{alphabet}. 
\par
Consider the set $A^{\Z}$  consisting of all bi-infinite sequences $x = (x_i)_{i \in \Z}$, where $x_i \in A$ for all $i \in \Z$.
An element of $A^{\Z}$ is also called a \emph{configuration} over the alphabet $A$.
For $x \in A^{\Z}$ and $i,j \in \Z$ such that $i \leq j$, we define the word $x_{[i,j]} \in A^{j - i + 1}$ by
$x_{[i,j]} \coloneqq x_i x_{i + 1} \cdots x_{j}$.
We also define the left-infinite  word  $x_{(-\infty,i]} \in A^{-\N}$ and the right-infinite word $x_{[i,\infty)} \in A^{\N}$ by
$x_{(-\infty,i]} \coloneqq \cdots  x_{i - 2}x_{i - 1}x_{i} $ and $x_{[i,\infty)} \coloneqq x_i x_{i + 1} x_{i + 2} \cdots $.
\par
We equip $A$ with its discrete topology and $A^{\Z}$    with the topology of pointwise convergence.
Thus,  regarding  $A^{\Z}$ as the product of a family of copies of $A$ indexed by $\Z$,
the topology on $A^{\Z}$ is the product topology.
A base for the topology on $A^{\Z}$ is formed by the \emph{cylinders}
\[
[w]_i \coloneqq \{x \in A^{\Z} : x_{[i,i + |w| - 1]} = w \},
\]
where $w$ runs over $A^*$ and $i$ over $\Z$.
Note that every cylinder is \emph{clopen} (i.e., both closed and open) in $A^{\Z}$.
\par
One says that a word $w \in A^*$ \emph{appears} in $x \in A^{\Z}$ if there exists $i \in \Z$ such that
$x \in [w]_i$.
\par
The space $A^{\Z}$ is   compact, metrizable, and totally disconnected.
A compatible metric $d$ on $A^{\Z}$ is obtained
by setting, for all $x,y \in A^{\Z}$,
\begin{equation}
\label{e:metric-on-shift}
d(x,y) \coloneqq
\begin{cases}
0 &\text{ if }x = y, \\ 
2^{-\min\{i \in \N : (x_{-i},x_i) \not= (y_{-i},y_i)\}} &\text{ if }x \not= y.
\end{cases}
\end{equation}
\par
We also equip $A^{\Z}$ with the homeomorphism 
$\sigma \colon A^{\Z} \to A^{\Z}$ defined   by
$(\sigma(x))_i  \coloneqq x_{i + 1}$ for all $x \in A^{\Z}$ and $i \in \Z$.
The homeomorphism $\sigma$ is called the \emph{shift map} and the dynamical system $(A^{\Z},\sigma)$ is called 
the \emph{full shift} over the alphabet $A$.  
\par
Two sequences $x,y \in A^{\Z}$ are unstably equivalent (resp.~stably equivalent, resp.~homoclinic) if and only if there exists
an integer $N \geq 0$ such that $x_i = y_i$ for all $i \in \Z$ with $i \leq -N$ (resp.~$i \geq N$, resp.~$|i| \geq N$). 
\par
A closed $\sigma$-invariant subset $X \subset A^{\Z}$ is called a \emph{subshift} over the alphabet $A$.
 \par
Let $X \subset A^{\Z}$ be a subshift.
The \emph{language} of  $X$ is the subset $L(X) \subset A^*$
consisting of all words $w \in A^*$ such that $w$ appears in some configuration $x \in X$. 
\par
A subshift  $X \subset A^{\Z}$ is non-wandering if and only if, for every $u \in L(X)$, there exists $v \in A^*$ such that $uvu \in L(X)$. 
A non-empty subshift  $X \subset A^{\Z}$ is irreducible  if and only if, for all $u,v \in L(X)$, there exists $w \in A^*$ such that $u w v \in L(X)$,
and it is topologically mixing if and only if, for all $u,v \in L(X)$, there exists an integer $N \geq 1$ such that, for any integer $n \geq N$, there exists a word $w \in A^n$  such that $u w v \in L(X)$.   
\par
A subset $X \subset A^{\Z}$ is a subshift if and only if there exists a subset $F \subset A^*$ such that $X$ consists of all $x \in A^{\Z}$ such that no element of $F$ appears in $x$.
Such a subset $F$ is then called a \emph{defining set of forbidden words} for $X$.
One says that a subshift $X \subset A^{\Z}$ is \emph{of finite type} if $X$ admits a finite defining set of forbidden words.
\par
Being of finite type is a topological conjugacy invariant for subshifts, i.e., if $A$, $B$ are finite sets and $X \subset A^{\Z}$, $Y \subset B^{\Z}$ are topologically conjugate subshifts, then $X$ is of finite type if and only if $Y$ is of finite type (see, e.g.,~\cite[Theorem~2.1.10]{lind-marcus-second}).
\par
The following result is a  form of the \emph{Curtis-Hedlund-Lyndon theorem}
(see, e.g., \cite[Proposition~1.5.8]{lind-marcus-second}).

\begin{proposition}
\label{p:char-ca}
Let $A$ be a finite set and let $X \subset A^{\Z}$ be a subshift.
Let $\tau \colon X \to X$ be a map.
Then the following conditions are equivalent:
\begin{enumerate}[\rm (a)]
\item
$\tau$ is an endomorphism of $X$;
\item
$\tau$ commutes with $\sigma$ and there exists an integer $m \geq 0$ and a map $\mu \colon A^{2m + 1} \to A$ such that
 $\tau(x)_0 = \mu(x_{[-m,m]})$ for all $x \in X$.
 \end{enumerate}
\end{proposition}

\subsection{Cantor-Bendixson decompositions and ranks}
\label{subsec:cb}
A separable topological space $X$ is said to be \emph{Polish} if $X$  admits a compatible metric $d$ such that $(X,d)$ is a complete metric space.
Every compact metrizable space is Polish (see, e.g.,~\cite[Proposition~I.4.2]{kechris}).
A subset of a Polish space is Polish (for the relative topology) if and only if it is a $G_\delta$-subset, i.e., the intersection of a countable family of open subsets
\cite[Theorem~I.3.11]{kechris}.
In particular, every open (resp.~closed) subset of a Polish space is itself Polish.
\par
Let $X$ be a Polish space.
By the \emph{Cantor-Bendixson decomposition theorem} \cite[Theorem~I.6.4]{kechris}, 
there is a unique perfect closed subset $P$ of $X$ such that $X \setminus P$ is countable.
The subset $P$ is called the \emph{perfect kernel} of $X$.
Observe that  $X$ is perfect (resp.~countable) if and only if $P = X$ (resp.~$P = \varnothing$).
\par
Let $X$ be a topological space.
Recall that one says that a point $x \in X$ is a \emph{limit point} of a subset $Y \subset X$ if every neighborhood of $x$ contains a point of $Y$ other than $x$.
The set of all limit points of $Y$ is called the \emph{derived set} of $Y$ and denoted by $Y'$.
The set  $Y$ is closed in $X$ if and only if $Y' \subset Y$. 
Note that $X \setminus X'$ is the set of isolated points of $X$.
Given an ordinal $\alpha$,  the \emph{derived set of order} $\alpha$
of $X$ is  the closed subset $X^{(\alpha)} \subset X$ defined, using transfinite induction,   
by setting  $X^{(0)} \coloneqq  X$, $X^{(\alpha + 1)} \coloneqq (X^{(\alpha)})'$ for every ordinal $\alpha$, and
$X^{(\lambda)} \coloneqq \bigcap_{\alpha < \lambda} X^{(\alpha)}$ for every limit ordinal $\lambda$.
There exists an  ordinal $\alpha_0$ such that $X^{(\alpha)} = X^{(\alpha_0)}$ for every ordinal $\alpha \geq \alpha_0$.
The smallest such $\alpha_0$ is called the \emph{Cantor-Bendixson rank} of $X$ and denoted by $\CBrank(X)$.
\par
In the case when $X$ is Polish, $\CBrank(X)$ is a countable ordinal and 
$X^{(\CBrank(X))}$ is the perfect kernel of $X$
\cite[Theorem~I.6.11]{kechris}.
A Polish space $X$ is countable if and only if $X^{(\CBrank(X))} = \varnothing$.
\par
If $X$ is a non-empty countable compact metrizable space, then $\CBrank(X)$ is a \emph{successor ordinal}, i.e., there exists an ordinal $\alpha$ such that $\CBrank(X) = \alpha + 1$.
Moreover,  $\alpha$ is the smallest ordinal such that
$X^{(\alpha)}$ is finite.
The \emph{Mazurkiewicz-Sierpi\'nski characteristic} of a countable compact metrizable space $X$
is the pair $(\alpha,n)$, where $\alpha$ is the smallest ordinal such that $X^{(\alpha)}$ is finite  and $n \coloneqq |X^{(\alpha)}|$. 
The Mazurkiewicz-Sierpi\'nski characteristic is  a complete topological invariant for countable compact metrizable spaces, that is,
two countable compact metrizable spaces are homeomorphic if and only if they have the same Mazurkiewicz-Sierpi\'nski 
characteristic~\cite{mazurkiewicz-sierpinski}.

\section{The poset of irreducible components}
\label{sec:irred-compo}

In this section, we investigate the poset of irreducible components of a subshift of finite type.
This is a finite poset which is a kind of skeleton for the subshift.
We first describe this poset for vertex subshifts.

\subsection{Vertex subshifts}
Let $G = (V,E)$ be a finite directed graph.
The set $X \subset V^{\Z}$, consisting of all sequences $x = (x_i)_{i \in \Z}$ such that $x_i \in V$ and $x_i x_{i+1} \in E$ for all $i \in \Z$,
is a subshift of finite type over the alphabet $V$
admitting $F \coloneqq V^2 \setminus E$ as a defining set of forbidden words.
One says that $X$ is the \emph{vertex subshift} associated with $G$.
\par
If $\pi = v_0 v_1 \cdots v_{n - 1} v_0$ is a cycle of length $n$ in $G$, then the sequence  $x \in V^{\Z}$, defined by $x_i \coloneqq v_j$ for all $i \in \Z$ and $j \in \{0,1,\dots,n-1\}$ such that
$i \equiv j \mod n$, is a periodic configuration in $X$ with period $n$.
One says that $x$ is the periodic configuration \emph{associated with the cycle} $\pi$.
\par
For $W \subset V$, the vertex subshift $Y \subset W^{\Z}$ associated with the subgraph $G_W$ of $G$ induced by $W$
is called  the vertex subshift associated with $W$.
Note that $Y$ is also a subshift of finite type of $V^{\Z}$
and that $Y$ is irreducible whenever $W$ is strongly connected.
 \par 
If $A$ is a finite set and $X \subset A^{\Z}$ is a subshift of
finite type, then there exists a finite directed graph $G$ such that $X$ is topologically conjugate to the vertex subshift associated with $G$ 
(see, e.g., \cite[Theorem~2.3.2]{lind-marcus-second}).
\par
The following observation will be useful in the study of the topological properties of subshifts of finite type.

\begin{proposition}
\label{p:normal-form-irred}
Let $G = (V,E)$ be a finite directed graph and let $X \subset V^{\Z}$ denote the vertex subshift associated with $G$.
Let $\KK$ denote the poset of path-components of $G$ and let $T \subset V$ denote the set of its transient vertices.
Fix $x \in X$.
Let $\Gamma = \Gamma(x)$ denote the subset of $\KK$ consisting of all $K \in \KK$ such that there exists $i \in \Z$ for which $x_i \in K$.
Then $\Gamma$ is a non-empty chain of $\KK$.
Moreover, writing $\Gamma = \{K_0,K_1,K_2,  \dots, K_{n-1}, K_n\}$ with $K_0 \prec K_1 \prec K_2 \prec \dots \prec K_{n - 1} \prec K_n$ and $n = |\Gamma| - 1 \geq 0$,
the following hold:
\begin{enumerate}[\rm (i)]
\item
for each $t \in T$, there is at most one $i \in \Z$ such that $x_i = t$;
\item
the set of $i \in \Z$ such that $x_i \in T$ is finite; 
\item
for each $K \in \KK$, the set of $i \in \Z$ such that $x_i \in K$ is an interval of $\Z$;
\item
if $n = 0$, then  $x_i \in K_0$ for all $i \in \Z$;
\item
if $n \geq 1$, then
there exist  unique integers
\[
b_0 < a_1 \leq b_1 < a_2 \leq b_2 < \dots <  a_{n -1} \leq b_{n - 1} < a_n
\] in $\Z$ 
such that one has, for every $i \in \Z$,
\begin{equation}
\label{e:normal-form-irred}
x_i \in 
\begin{cases}
K_0 & \text{ if } i \in (-\infty,b_0], \\
K_j & \text{ if } i \in [a_j,b_j] \text{ for some } j \in \{1,2,\dots,n-1\}, \\
K_n & \text{ if } i \in [a_n,\infty), \\
T & \text{ otherwise.}
\end{cases}
\end{equation}
\end{enumerate}
\end{proposition}

\begin{proof}
Let $t \in T$ and suppose by contradiction that $x_i = x_j = t$ for some  $i,j \in \Z$ with $i < j$.
Then the word $x_i x_{i+1} \cdots  x_j$ is a cycle based at $t$, contradicting the fact that the vertex $t$ is transient.
This shows (i).
\par
Assertion (ii) follows from (i) since the set $T$ is finite.
\par
If $i \leq j$ are integers such that $x_i,x_j \in K$ for some $K \in \KK$, then $x_k \in K$ for all integers $k \in [i,j]$.
Indeed, if $i < k < j$, then $x_i x_{i+1} \cdots x_k$ is a path from $x_i$ to $x_k$  and $x_k x_{k+1}  \cdots x_j$ is a path from $x_k$ to $x_j$.
This shows (iii).
\par
We know that $\KK$ is a partition of $R = V \setminus T$.
Thus, given $i \in \Z$, either $x_i \in T$ or $x_i \in K$ for some unique $K \in \KK$.
On the other hand, if $K,K' \in \KK$ are such that $x_i \in K$ and $x_j \in K'$ for some integers $i < j$,
then $K \preceq K'$, since $x_i x_{i + 1}\cdots x_j$ is a path in $G$.
As the poset $\KK$ is finite, combining these observations with (ii) and (iii), 
we deduce that $\Gamma$ is a non-empty chain and we get (iv) and (v). 
\end{proof}

Let $X \subset V^{\Z}$ denote the vertex subshift associated with a finite directed graph $G = (V,E)$ and let $x \in X$.
With the notation introduced in Proposition~\ref{p:normal-form-irred}, 
one says that $\Gamma$ is the \emph{chain of path-components} of   $x$
and that  $K_0$ (resp.~$K_n$) is  the
\emph{initial} (resp.~\emph{terminal}) path-component of $x$.

\begin{proposition}
\label{p:limi-set-vertex-ss}
Let $G = (V,E)$ be a finite directed graph and let $X \subset V^{\Z}$ denote the vertex subshift associated with $G$.
Let $x \in X$ and denote by $K$ (resp.~$K'$) the initial (resp.~terminal) path-component of $x$.
Let $C$ (resp.~$C'$) denote the  irreducible subshift of $V^{\Z}$ associated with $K$ (resp.~$K'$).
Then one has
\begin{equation}
\label{e:alpha-C-omega-C-prime}
\alpha(x) \subset C
\quad \text{and} \quad
\omega(x) \subset C'.
\end{equation}
Moreover, the following conditions are equivalent:
\begin{enumerate}[\rm (a)]
\item
  $C = C'$;
  \item
  $K = K'$;
  \item
$x \in C$;
\item
$x \in C'$;
\item
$x \in \NW(X)$.
\end{enumerate}
\end{proposition}

\begin{proof}
Let $y \in \alpha(x)$ and 
let $U$ be a neighborhood of $y$ in $X$.
Choose  $k \in \N$ so that
\[   
U' \coloneqq \{z \in X :z_{[-k,k]}  = y_{[-k,k]} \} \subset U.
\]
By Assertions (iv) and (v) in Proposition~\ref{p:normal-form-irred},
there exists  $b_0 = b_0(x) \in \Z$ such that $x_i \in K$ for all  $i \leq b_0$.
Since $y \in \alpha(x)$, there exists  $m \in \N$ such that 
$k - m \leq b_0$ and 
$\sigma^{-m}(x) \in U'$.
This implies $y_i = x_{i-m} \in K$ for all $-k \leq  i \leq k$.
Since $y_{-k}y_{-k+1} \cdots y_k$ is a path in $K$ and $K$ is strongly connected, 
we can find $z \in C$ such that $z_{[-k,k]} = y_{[-k,k]}$.
We then have $z \in U' \subset U$.
Thus, every neighborhood of $y$ in $X$ meets $C$.
As $C$ is closed in $X$, we deduce that $y \in C$.
This shows $\alpha(x) \subset C$.
\par
The proof of $\omega(x) \subset C'$ is similar.
\par
This shows~\eqref{e:alpha-C-omega-C-prime}.
\par
Let us prove now the equivalence of Conditions (a)---(e).
\par
The elements of $\KK$ form a partition of the set of recurrent vertices.
We deduce that (a) and (b) are equivalent.
 \par
If $x \in C$, then $x_i \in K$ for all $i \in \Z$.
This implies that $K$ is both the initial and the terminal path-component of $x$.
Therefore $K = K'$.
Conversely, suppose now  $K = K'$. 
Then $x_i \in K$ for all $i \in \Z$
by Proposition~\ref{p:normal-form-irred}.(iv).
Hence $x \in C$.
This shows that (c) and (a) are equivalent.
\par
The proof of the equivalence of (d) and (a) is similar.
\par
Since the subgraph $K$ is strongly connected, 
the subshift $C$ is irreducible and hence  non-wandering. It follows that if $x \in C$ then $x \in \NW(X)$.
Therefore (c) implies (e).
\par
Suppose now $K \not= K'$.
With the notation of Proposition~\ref{p:normal-form-irred}.(iv), we have
$K_0 = K \prec K' = K_n$.
Consider the open neighborhood  $U$ of $x$ in $X$ defined  by  
\[
U \coloneqq \{y \in X : y_{[b_0,a_n]} = x_{[b_0,a_n]} \}.
\]
We deduce from Proposition~\ref{p:normal-form-irred}.(iv)  that the sets
$\sigma^m(U)$, $m \geq 0$, are all disjoint.
Therefore $x \notin \NW(X)$.
This shows that (e) implies (b) and
completes the proof of the equivalence of (a)---(e). 
\end{proof}

\begin{proposition}
\label{p:nw-vertex-ss}
Let $G = (V,E)$ be a finite directed graph and let $X \subset V^{\Z}$ denote the subshift of finite type associated with $G$. 
Let $\KK$ denote the poset of path-components of $G$.
For each $K \in \KK$, denote by $C_K \subset V^{\Z}$ the   subshift of finite type associated with $K$.
Then the following hold:
\begin{enumerate}[\rm (i)]
\item
each $C_K$, $K \in \KK$, is an irreducible subshift of finite type;
\item
the subshifts $C_K$, $K \in \KK$, are pairwise disjoint;
\item
$\NW(X) = \bigcup_{K \in \KK} C_K$; 
\item
$\NW(X)$ is a subshift of finite type;
\item
given $K,K' \in \KK$, one has $K \preceq K'$ if and only if there exists $x \in X$ such that
$\alpha(x) \subset C_K$ and $\omega(x) \subset C_{K'}$.
\end{enumerate}
\end{proposition}

\begin{proof}
Assertion (i) follows from the fact that each $K \in \KK$ is strongly connected (this has been already observed in the proof of Proposition~\ref{p:limi-set-vertex-ss}).
\par
For every $x \in C_K$, $K \in \KK$, we have $\{x_i : i \in \Z\} \subset K$.
As the elements of $\KK$ are pairwise disjoint, Assertion (ii) follows.
\par 
Assertion (iii)  follows from the equivalence between Conditions (c) and (e) in Proposition~\ref{p:limi-set-vertex-ss}.
\par
By (iii), $\NW(X)$ is a finite union of subshifts of finite type $C_K$, $K\in\KK$, defined over the pairwise disjoint alphabets $K$.
Hence $\NW(X)$ is itself a subshift of finite type (cf.~\cite[Exercise~1.59]{csc-ecag}).
This shows (iv).
\par
Let $K,K' \in \KK$.
Choose vertices $v \in K$ and $v' \in K'$.
Suppose first $K \preceq K'$.
This means that there exists a path $v_0\cdots v_n$ in $G$ of length $n \geq 1$  from $v$ to $v'$.
Since the vertices $v$ and $v'$ are recurrent,
we can find a cycle $c_0\cdots c_k$ based at $v$ and a cycle $c_0' \cdots  c_{k'}'$ based at $v'$.
Consider the sequence $x \in V^{\Z}$ defined by
\[
x_i \coloneqq
\begin{cases}
c_j & \text{ if } i \leq 0 \text{ and } i \equiv j \mod k, \\
v_i & \text{ if } 0 \leq i \leq n \\
c_j' & \text{ if } n \leq i \text{ and } i - n \equiv j \mod k'. 
\end{cases}
\]
Then $K$ (resp.~$K'$) is the initial (resp.~terminal) path-component of $x$.
Therefore $\alpha(x) \subset C_K$ and $\omega(x) \subset C_{K'}$ by \eqref{e:alpha-C-omega-C-prime}.
\par
Conversely, suppose that there exists $x \in X$ such that $\alpha(x) \subset C_K$ and $\omega(x) \subset C_{K'}$.
Then $K$ (resp.~$K'$) is the initial (resp.~terminal) path-component of $x$ by \eqref{e:alpha-C-omega-C-prime}. 
It follows that there exist integers $m,n$ with $m < n$ such that
$x_i \in K$ for all $i \leq m$ and $x_i \in K'$ for all $n \leq i$.
Then $x_m \cdots x_n$ is a path in $G$ from $x_m \in K$ to $x_n \in K'$.
Therefore  $K \preceq K'$. This completes the proof of (v).
\end{proof}

\subsection{Irreducible components of  subshifts of finite type}

\begin{proposition}
\label{p:irred-comp}
Let $A$ be a finite set and let $X \subset A^{\Z}$ be a subshift of finite type.
Then the following hold:
\begin{enumerate}[\rm (i)]
\item
the subshift $\NW(X)$ is of finite type;
\item
if $\CC$ is a finite partition of $\NW(X)$ into irreducible subshifts and $Z \subset X$ is an irreducible subshift,  then there exists  a unique $C \in \CC$ such that  $Z \subset C$;
\item
there exists a unique finite partition  $\CC$  of $\NW(X)$ into irreducible subshifts;
\item
each subshift $C \in \CC$ is of finite type;
\item
given $x \in X$,  there exist a unique $C_\alpha(x) \in \CC$ and a unique $C_\omega(x) \in \CC$ such that
$\alpha(x) \subset C_\alpha(x)$ and $\omega(x) \subset C_\omega(x)$;
\item
given $x \in X$ and $C \in \CC$,
one has $C_\alpha(x) = C_\omega(x) = C$ if and only if $x \in C$;
\item
given $x \in X$, one has $x \in \NW(X)$ if and only if $C_\alpha(x) = C_\omega(x)$;
\item 
the relation $\preceq$ on $\CC$, defined by writing $C \preceq C'$ for $C,C' \in \CC$ if  there exists $x \in X$ such that
$C_\alpha(x) = C$ and $C_\omega(x) = C'$, is a partial ordering on $\CC$.
\end{enumerate}
\end{proposition}

\begin{proof}
Since every subshift of finite type is topologically conjugate to a vertex subshift, we may assume that $X$ is the vertex subshift associated with a finite directed graph $G = (V,E)$.
\par
Assertion~(i) is then   Assertion~(iv) in Proposition~\ref{p:nw-vertex-ss}.
\par
To prove (ii),
suppose that $\CC$ is a finite partition of $\NW(X)$ into irreducible subshifts  and
let $Z \subset X$ be an irreducible subshift.
Then $Z \subset \NW(X)$ since every irreducible dynamical system  is non-wandering.
As $\CC$ is a finite partition of $\NW(X)$ into closed subsets,
each element of $\CC$ is open in $\NW(X)$.
We deduce that the sets $Z \cap C$, $C \in \CC$, are open $\sigma$-invariant subsets of $Z$.
Since $Z$ is irreducible, it follows that there exists a unique $C \in \CC$ such that $Z \subset C$.
This shows (ii).
\par
To prove (iii), consider the set  $\CC \coloneqq  \{C_K : K \in \KK\}$, where $\KK$ denotes the set of path-components of $G$ and $C_K$ denotes the vertex subshift associated with $K \in \KK$.
 We know that $\CC$ is a finite partition of $\NW(X)$ into irreducible subshifts   (cf.~Assertions~(i), (ii), and (iii)   in Proposition~\ref{p:nw-vertex-ss}).
This shows existence in (iii).
\par
To prove uniqueness, suppose that $\CC'$ is another  finite partition of $\NW(X)$ into irreducible subshifts.
By (ii), there exist a  map $f \colon \CC \to \CC'$ such that $C \subset f(C)$ for all $C \in \CC$ and a  map
$g \colon \CC' \to \CC$ such that $C' \subset g(C')$ for all $C' \in \CC'$.
We then have $C \subset g(f(C))$ and hence $C = g(f(C))$ for all $C \in \CC$. 
This implies that $f$ is bijective and $\CC = \CC'$ follows. 
This completes the proof of  (iii).
\par
Assertion (iv) follows from Proposition~\ref{p:nw-vertex-ss}.(i).
\par
Assertions~(v), (vi), and  (vii)   follow from  Proposition~\ref{p:limi-set-vertex-ss}.
\par
Assertion~(viii) follows from Proposition~\ref{p:nw-vertex-ss}.(vi).
\end{proof}

With the notation of Proposition~\ref{p:irred-comp},
the subshifts  $C$, $C \in \CC$, are called the \emph{irreducible components} of $X$.
As $\CC$ is a poset, it makes sense to speak of an isolated (resp.~intermediate, resp.~extremal, resp.~maximal, resp.~minimal) irreducible component of $X$.
\par
From the results of Proposition~\ref{p:irred-comp}, we get the following characterization of non-wandering subshifts of finite type.
 
\begin{corollary}
\label{c:char-nw-sft}
Let $A$ be a finite set and let $X \subset A^{\Z}$ be a subshift of finite type.
Then $X$ is non-wandering if and only if $X$ is the union of its isolated irreducible components.
\end{corollary}

\begin{remark}
\label{r:irred-vertex-sft}
Suppose that $X \subset V^{\Z}$ is the vertex subshift associated with a finite directed graph $G = (V,E)$.
Let $\KK$ denote the set of path-components of $G$.
For each $K \in \KK$, denote by $C_K \subset V^{\Z}$ the   subshift of finite type associated with $K$.
It follows from Proposition~\ref{p:nw-vertex-ss} that the set
of irreducible components of $X$ is $\CC = \{C_K : K \in \KK\}$ and that the map
$\KK \to \CC$, $K \mapsto C_K$, is a poset isomorphism.
One says that $K$ (resp.~$C_K$) is the path-component (resp.~irreducible component) \emph{associated with} $C_K$ (resp.~$K$).
Note also that, for every $x \in X$,
if the initial (resp.~terminal) path-component of $x$ is $K$ (resp.~$ K'$), then one has $C_\alpha(x) = C_K$ (resp.~$C_\omega(x) = C_{K'}$). 
\end{remark}

\begin{proposition}
\label{p:finite-component-1-1}
Let $G = (V,E)$ be a finite directed graph and let $X \subset V^{\Z}$ denote the associated vertex subshift.
Let $K$ be a path-component of $G$ and let $C$ denote the irreducible component of $X$ associated with $K$.
Then the following conditions are equivalent:
\begin{enumerate}[{\rm (a)}]
\item 
$C$ is finite with cardinality $|C| = n$;
\item 
there exists a simple cycle $v_0 v_1  \cdots  v_n$ in $G$ such that $K = \{v_0, v_1, \ldots, v_{n-1}\}$
and $E \cap K^2  = \{v_rv_{r + 1} : 0 \leq r \leq n - 1\}$.
\end{enumerate}
\end{proposition}

\begin{proof}
Suppose (a). 
Let $v_0 v_1  \cdots v_m$ be a cycle in $K$ of minimal length $m \geq 1$.
Note that such a  cycle  is  simple by minimality.
Consider the associated periodic configuration $x \in C$, i.e., the configuration  defined by
$x_i \coloneqq v_j$ if $i \in \Z$ and $i \equiv j \mod m$.
Then $x$ is periodic with minimal period $m$.
By applying Proposition~\ref{p:finite-irred-ds},
we deduce that  $\OO(x) = C$ and  $m = n$.  
Let $v \in K$. 
As $v$ is recurrent, there exists a cycle $v'_0  v'_1  \cdots  v'_{m'}$ based at $v$. Consider the  associated periodic configuration $x'$ defined by setting $x'_i \coloneqq v'_j$ if $i \in \Z$ and $i \equiv j \mod m'$.
As $x' \in C = \OO(x)$, there exists $k \in \{0,1, \ldots, n-1\}$ such that $x' = \sigma^k(x)$.
We deduce that $v = v'_0 = x'_0 = (\sigma^k(x))_0 = x_k = v_k$.
This shows that $K = \{v_0,v_1, \ldots, v_{n-1}\}$.
\par
We have $v_r v_{r + 1} \in E$ for all $0 \leq r \leq n - 1$ since $v_0 v_1 \cdots  v_n$ is a cycle in $G$. 
Conversely, suppose that  $v,w \in K$ are such that  $v w \in E$.
As $K$ is a path-component,
there exists a path $w_1 \cdots w_p$ in $K$ going from $w$ to $v$.
Then, setting $w_0 \coloneqq v$, the word  $w_0 w_1 \cdots w_p$ is a cycle based at $v$ in $K$.
Consider the associated periodic configuration $z$ defined by $z_i \coloneqq w_j$ if $i \in \Z$ and $i \equiv j \mod p$.
We have $z \in C = \OO(x)$. Consequently, there exists $r \in \{0,1,\dots,n-1\}$ such that $z = \sigma^r(x)$.
Then $v = w_0 = z_0 = x_r$ and $w = w_1 = z_1 = x_{r + 1}$.
This shows that $E \cap K^2  = \{v_r v_{r+1} : 0 \leq r \leq n - 1\}$.
The implication (a) $\implies$ (b) follows.
\par
Conversely, suppose (b). 
Consider the configuration $x \in C$ defined by $x_i \coloneqq  v_j$ if $i \in \Z$ and $i \equiv j \mod n$.
Note that $x$ is periodic with minimal period $n$ since $v_0 v_1 \cdots v_n$ is a simple cycle.
Let now $y \in C$.
By (b), for every $i \in \Z$, there exists $r \in \{0,1,\dots,n - 1\}$ such that $y_i y_{i + 1} = v_r v_{r + 1}$.
Denoting by $k$ the unique element in $\{0,1,\dots,n-1\}$ such that $y_0 = v_k$,
we deduce, by induction on $|i|$, that
$y_i = v_j$ for every $i \in \Z$, where $j \in \{0,1,\dots,n-1\}$ satisfies
$k + i \equiv j \mod n$.  
Thus, $y = \sigma^k(x) \in \OO(x)$.
It follows that  $C = \OO(x)$, so that  $|C| = |\OO(x)| = n$.
This shows that (b) implies (a).
\end{proof}

\begin{corollary}
\label{c:int-x-finite-irred-compo}
Let $G = (V,E)$ be a finite directed graph and let $X \subset V^{\Z}$ denote the associated vertex subshift.
Let $K$ be a path-component of $G$ and let $C$ denote the irreducible component of $X$ associated with $K$.
Let $x \in X$ and consider the interval $I \subset \Z$ consisting of all $i \in \Z$ such that $x_i \in K$
(cf.~Proposition~\ref{p:normal-form-irred}.(iii)).
Suppose that $C$ is finite and let $i_0 \in I$. 
Then,  the restriction of $x$ to $I$ is entirely determined by $x_{i_0}$.
\end{corollary}

\begin{proof}
By virtue of Proposition~\ref{p:finite-component-1-1},
there exists a simple cycle $v_0 v_1  \cdots v_n$ in $G$ such that $K = \{v_0, v_1, \ldots, v_{n - 1}\}$
and $E \cap K^2  = \{v_rv_{r + 1} : 0 \leq r \leq n - 1\}$. 
We deduce that there exists a unique element
$ r \in \{0,1, \dots, n - 1\}$ such that $x_{i_0} = v_r$. 
Moreover, for every $i \in I$, we then have $x_i = v_j$, where $j \in \{0,1,\dots,n-1\}$ is given by  $j \equiv r + i  - i_0 \mod n$.
\end{proof}

\begin{proposition}
\label{p:finite-component-2}
Let $A$ be a finite set and let $X \subset A^{\Z}$ be a subshift of finite type.
Then every  infinite irreducible component of $X$
contains infinitely many periodic configurations.
\end{proposition}

\begin{proof}
Let $C$ be an irreducible component of $X$.
Since $C$ is an irreducible subshift of finite type, its periodic configurations are dense in $C$ (see, e.g.,~\cite[Exercise~6.1.12]{lind-marcus-second}). 
We deduce that $C$ is finite if it contains only  finitely many periodic configurations. 
\end{proof}

Every irreducible component $C$ of a subshift of finite type $X$ is closed in $X$.
By the following result,  $C$ is open (and hence clopen) in $X$ if and only if $C$ is isolated in the poset of irreducible components of $X$.

\begin{proposition}
\label{p:isolated-component}
Let $A$ be a finite set and let $X \subset A^{\Z}$ be a subshift of finite type.
Let $\CC$ denote the poset of irreducible components of $X$ and let $C \in \CC$.
Then the following conditions are equivalent:
\begin{enumerate}[\rm (a)]
\item
$C$ is isolated in the poset $\CC$;
\item
 $C$ is an open subset of $X$.
 \end{enumerate}
\end{proposition}

\begin{proof}
Suppose that $C$ is not isolated in $\CC$.
This means that  there  exists $C' \in \CC$ such that either $C \prec C'$ or $C' \prec C$.
Suppose that $C \prec C'$ (the case when $C' \prec C$ is analogous).
It follows that there exists a configuration $x \in X$ such that $C_\alpha(x) = C$ and $C_\omega(x) = C'$.
Then $x \notin C$ and there exists a strictly increasing sequence $(n_k)_{k \in \N}$ of non-negative integers such that
the sequence $(\sigma^{-n_k}(x))_{k \in \N}$ converges to some point in $\alpha(x) \subset C_\alpha(x) = C$.
As $\sigma^n(x) \notin C$ for all $n \in \Z$ since $C$ is $\sigma$-invariant, this implies that $C$ is not open in $X$.
This shows that (b) implies (a).
\par
Conversely, suppose that $C$ is isolated in $\CC$ and 
let us show that $C$ is open in $X$.
We may assume that $X$ is the vertex subshift associated with a finite directed graph $G = (V,E)$.
Denote by  $K$ the path-component of $G$ associated with $C$.
Let $x \in C$.
Then $x_i \in K$ for all $i \in \Z$.
In particular, $x_0 \in K$.
Consider the open neighborhood  $U$ of $x$ in $X$ consisting of all configurations in $X$ which take the same value as $x$ at $0$.
If $y \in U$, then $y_0 \in K$.
As $K$ is isolated in the poset of path-components of $G$, 
we then deduce from Proposition~\ref{p:normal-form-irred} that $y_i \in K$ for all $i \in \Z$.
It follows that $y \in C$. Therefore $U \subset C$.
This shows that $C$ is open in $X$.
\end{proof}

\subsection{Characterization of isolated configurations}

\begin{proposition}
\label{p:isolated-configuration}
Let $A$ be a finite set, let $X \subset A^{\Z}$ be a subshift of finite type,
and let $x \in X$.
Denote by $\CC$ the poset of irreducible components of $X$. 
Then the following conditions are equivalent:
\begin{enumerate}[{\rm (a)}]
\item 
the configuration $x$ is isolated in $X$;
\item 
the irreducible components $C_\alpha(x)$ and $C_\omega(x)$ are both finite,
with $C_\alpha(x)$ minimal and  $C_\omega(x)$  maximal in $\CC$.
\end{enumerate}
\end{proposition}

\begin{proof}
We may assume that $X$ is the vertex shift associated with a finite directed graph $G = (V,E)$.
To simplify, let us set $C \coloneqq C_\alpha(x)$ and $C' \coloneqq C_\omega(x)$.
Let $K$ (resp.~$K'$) denote the path-component of $G$ associated with $C$ (resp.~$C'$).
Since $K$ (resp.~$K'$) is the initial (resp.~terminal) path-component of $x$, there exist  $a,b \in \Z$ 
with $a \leq b$ such that $x_i \in K$ (resp.~$x_i \in K'$) for all  $i \leq a$ (resp.~$i \geq b$).
\par
Let us assume that Condition~(b) fails to hold. Let us show that $x$ is not isolated in $X$.
Suppose first that $C$ is infinite (the case where $C'$ is infinite is similar).
Let $M$ be a positive integer and set $i_0 \coloneqq \min(a,-M)$. 
Given $y \in \Per(C)$, we construct a configuration $z = z(y) \in V^{\Z}$ as follows.
Since $y_0, x_{i_0} \in K$ and $K$ is a path-component of $G$, there exist an integer $k \geq 1$ and a path $v_0 v_1 \cdots v_k$ of length $k$ in $K$ 
going from $v_0 = y_0$ to $v_k = x_{i_0}$. 
Denoting by $p(y)$ the minimal period of $y$, we then set, for every $i \in \Z$,
\begin{equation}
\label{e:x-z-isolated}
z_i \coloneqq 
\begin{cases} 
y_j & \mbox{ if } i \leq i_0 - k \mbox{ and } i \equiv i_0-k+j \mod p(y)\\
v_{i-i_0+k} & \mbox{ if } i_0-k \leq i \leq i_0 \\
x_i & \mbox{ if } i \geq i_0.
\end{cases}
\end{equation}
We clearly have  $z \in X$ and $z_{[-M,M]} = x_{[-M,M]}$.
Moreover, the $\alpha$-limit set $\alpha(z)$ of $z$ equals $\alpha(y) = \OO(y)$ and,
in particular, $\vert \alpha(z) \vert = \vert \alpha(y) \vert = \vert \OO(y) \vert = p(y)$.
\par
Since $C$ is infinite, it contains infinitely many periodic configurations by Proposition~\ref{p:finite-component-2}.
Thus, there exist  $y,y' \in \Per(C)$ whose minimal periods $p(y)$ and $p(y')$ are distinct.
Denote by $z = z(y)$ and $z' = z(y')$ the configurations in $X$ associated with $y$ and $y'$ via~\eqref{e:x-z-isolated}.
As $\vert \alpha(z) \vert = p(y) \neq p(y') = \vert \alpha(z') \vert$, we have $z \neq z'$.
Therefore, we have $x \not= z$ or $x \not= z'$.
Since the positive integer $M$ was arbitrary and the three configurations $x,z$ and $z'$ satisfy $x_{[-M,M]} = z_{[-M,M]} = z'_{[-M,M]}$,
we deduce that $x$ is not isolated in $X$.
\par
Suppose now that the irreducible component $C$ is not  minimal  (the case where $C'$ is not maximal is similar).
Let $M$ be a positive integer and, as before, set $i_0 \coloneqq \min(a,-M)$. 
As $C$ is not minimal, there exists and irreducible component $C''$ of $X$  
such that  $C'' \prec C$. 
Let $y \in C''$.
Denoting by $K''$ the path-conponent of $G$ associated with $C''$, we have $K'' \prec K$.
Consequently, there exists an integer $k \geq 1$ and a path $v_0 v_1 \cdots v_k$ of length $k$ in $G$
going from $v_0 = y_0$ to $v_k = x_{i_0}$. 
Consider the configuration $z \in V^{\Z}$ defined, for all $i \in \Z$,  by 
\begin{equation*}
z_i \coloneqq 
\begin{cases} 
y_{i - i_0 + k} & \mbox{ if } i \leq i_0 - k \\
v_{i-i_0+k} & \mbox{ if } i_0-k \leq i \leq i_0 \\
x_i & \mbox{ if } i \geq i_0.
\end{cases}
\end{equation*}
Note that $z \in X$.
We have  $z \neq x$ since $C_\alpha(z) = C'' \neq C = C_\alpha(x)$.
As $z_{[-M,M]} = x_{[-M,M]}$ 
and the positive integer $M$ was arbitrary, this shows that $x$ is not isolated.
The implication (a) $\implies$ (b) follows.
\par
Conversely, assume (b). Let us show that $x$ is isolated in $X$.
Let $y \in X$ such that $y_{[a,b]} = x_{[a,b]}$. 
Using the fact that  $K$ is minimal (resp.~$K'$ is maximal), 
we get $x_i, y_i \in K$ (resp.~$x_i,y_i \in K'$) for all $i \leq a$ (resp.~$i \geq b$).
As $C$ (resp.~$C'$) is finite and $x_a = y_a$ (resp.~$x_b = y_b$)
we  deduce from Corollary~\ref{c:int-x-finite-irred-compo} 
that $x_i = y_i$ for all $i \leq a$ (resp.~$i \geq b$).
Since $x_{[a,b]} = y_{[a,b]}$, we have that $x = y$.
 As the set of configurations in $X$ which coincide with $x$ on $[a,b]$ is a neighborhood of $x$,
we conclude that  $x$ is isolated in $X$.
This proves (b) $\implies$ (a).
\end{proof}

\begin{corollary}
\label{c:inf-component-uncountable}
Let $A$ be a finite set, let $X \subset A^{\Z}$ be a subshift of finite type,
and let $C$ be an infinite irreducible component of $X$.
Then $C$ is a perfect space.
Consequently, $C$ is homeomorphic to the Cantor set and hence uncountable.
\end{corollary}

\begin{proposition}
\label{p:countable-sft}
Let $A$ be a finite set and let $X \subset A^{\Z}$ be a subshift of finite type.
Then $X$ is countable if and only if every irreducible component of $X$ is finite.
\end{proposition}

\begin{proof}
If $X$ contains an infinite component then $X$ is uncountable by Corollary~\ref{c:inf-component-uncountable}.
Thus, the condition is necessary.
\par
Conversely, suppose  that every irreducible component of $X$ is finite and let us show that $X$ is countable.
We may assume that $X$ is the vertex subshift associated with a finite directed graph $G = (V,E)$.
Let $x \in X$. Set $C \coloneqq C_\alpha(x)$ and $C' \coloneqq C_\omega(x)$.
If $C = C'$ then $x \in C$.
Otherwise, we have $C \prec C'$ and if the integers $b_0 < a_n$ are as in Proposition~\ref{p:normal-form-irred},
it follows from Corollary~\ref{c:int-x-finite-irred-compo} that
$x_{(-\infty,b_0]}$ (resp.~$x_{[a_n,\infty)}$) is entirely determined by the value of $x_{b_0}$ (resp.~$x_{a_n}$).
We deduce that there are only finitely many $x \in X$ for a fixed pair $(b_0,a_n)$.
 This implies that $X$ is countable.
\end{proof}

\subsection{Action of endomorphisms on irreducible components}

\begin{proposition}
\label{p:ca-sft-irred-compo}
Let $A$ be a finite set, let $X \subset A^{\Z}$ be a subshift of finite type,
and let $\CC$ denote the poset of irreducible components of $X$.
Suppose that  $\tau \colon X \to X$ is an endomorphism of $X$.
Then the following hold:
\begin{enumerate}[\rm (i)]
\item
there exists a unique map $\rho = \rho_\tau \colon \CC \to \CC$ such that,
for every $C \in \CC$, one has
$\tau(C) \subset \rho(C)$;
\item
for every $x \in X$, one has 
$\rho(C_\alpha(x)) = C_\alpha(\tau(x))$ and 
$\rho(C_\omega(x)) = C_\omega(\tau(x))$;
\item
the map $\rho \colon \CC \to \CC$ is a poset endomorphism of $\CC$.
\end{enumerate}
Moreover, the map $\End(X,\sigma) \to \End(\CC,\preceq)$, $\tau \mapsto \rho_\tau$, is a monoid morphism.
\end{proposition}

\begin{proof}
Let $C \in \CC$.
As $\tau$ is an endomorphism of $(X,\sigma)$, we deduce from  
Proposition~\ref{p:properties-endo}.(iv)  that   $\tau(C) \subset X$ is an irreducible subshift.
It then follows from Proposition~\ref{p:irred-comp}.(iii)   that there exists a unique $C' \in \CC$ such that $\tau(C) \subset C'$.
This proves (i).
\par
Let $x \in X$.
We have  $\tau(\alpha(x)) = \alpha(\tau(x))$ and $\tau(\omega(x)) = \omega(\tau(x))$
by Proposition~\ref{p:properties-endo}.(ii).
We deduce that
\[
\tau(\alpha(x)) = \alpha(\tau(x)) \subset C_\alpha(\tau(x))
\]
and
\[
\tau(\omega(x)) = \omega(\tau(x)) \subset C_\omega(\tau(x)).
\]
On the other hand, we have
$\varnothing \not= \alpha(\tau(x)) \subset C_\alpha(\tau(x))$ and $\varnothing \not= \omega(\tau(x)) \subset C_\omega(\tau(x))$.
Using uniqueness in (i),
we deduce that
$\tau(C_\alpha(x)) \subset C_\alpha(\tau(x))$ and 
$\tau(C_\omega(x)) \subset C_\omega(\tau(x))$. 
This shows (ii).
  \par
To prove (iii), suppose that  $C,C' \in \CC$ are such that $C \preceq C'$. 
This means that there exists $x \in X$ such that
$\alpha(x) \subset C$ and $\omega(x) \subset C'$.
Using (ii), we get
\[
\alpha(\tau(x)) = \tau(\alpha(x)) \subset \tau(C) \subset \rho(C) 
\]
and
\[
\omega(\tau(x)) = \tau(\omega(x)) \subset \tau(C') \subset \rho(C'). 
\]
We deduce that  $\rho(C) \preceq \rho(C')$.
This shows  (iii).
\par
We clearly have $\rho_{\Id_X} = \Id_{\CC}$.
On the other hand, let $\tau_1,\tau_2 \in \End(X,\sigma)$.
For every $C \in \CC$, we have
\[
(\tau_1\tau_2)(C) = \tau_1(\tau_2(C)) \subset \rho_{\tau_1}(\tau_2(C)) \subset \rho_{\tau_1}(\rho_{\tau_2}(C)) = (\rho_{\tau_1} \rho_{\tau_2})(C).
\]
Therefore $\rho_{\tau_1\tau_2} = \rho_{\tau_1} \rho_{\tau_2}$.
This shows that $\tau \mapsto \rho_\tau$ is a monoid morphism from $\End(X,\sigma)$ to $\End(\CC,\preceq)$.
\end{proof}

\begin{proposition}
\label{p:surj-tau-implies-surj-rho}
Let $A$ be a finite set, let $X \subset A^{\Z}$ be a subshift of finite type,
and let $\CC$ denote the poset of irreducible components of $X$.
Suppose that  $\tau \colon X \to X$ is a surjective endomorphism of $X$
and let $\rho = \rho_\tau \colon \CC \to \CC$ denote the poset endomorphism of $\CC$ associated with $\tau$.
Then the following hold:
\begin{enumerate}[\rm (i)]
\item
the map  $\rho \colon \CC \to \CC$  is a poset automorphism  of $\CC$;
\item
one has  $\tau^{-1}(C) = \rho^{-1}(C)$ for each $C \in \CC$;
\item
one has $\tau(C) = \rho(C)$ for each $C \in \CC$;
\item
one has $\tau(\NW(X)) = \NW(X)$ and $\tau^{-1}(\NW(X)) = \NW(X)$;
\item
one has $\tau(X \setminus \NW(X)) = X \setminus \NW(X)$ and $\tau^{-1}(X \setminus \NW(X)) = X \setminus \NW(X)$.
\end{enumerate}
\end{proposition}

\begin{proof}
Let $C \in \CC$ and let $y \in C$. 
We have  $C_\alpha(y)  = C$ by Proposition~\ref{p:irred-comp}.(vi).
Since $\tau \colon X \to X$ is surjective, there exists  $x \in X$ such that $\tau(x) = y$.
Let  $C' \coloneqq C_\alpha(x) \in \CC$.
Using Proposition~\ref{p:ca-sft-irred-compo}.(ii), we get  
\[
\tau(C') = \tau(C_\alpha(x)) \subset C_\alpha(\tau(x)) = C_\alpha(y) = C, 
\]
which implies  $\rho(C') = C$.
It follows that   $\rho$ is surjective. 
As $\CC$ is finite, $\rho$ is bijective.
Using Proposition~\ref{p:ca-sft-irred-compo}.(iii), we deduce that $\rho$ is a poset automorphism of $\CC$.
This shows (i).
\par
To prove (ii), let $C \in \CC$ and let  $C' \coloneqq  \rho^{-1}(C)$.
Since $\rho(C') = C$, we have $\tau(C') \subset C$
and hence
$C' \subset \tau^{-1}(C)$.
\par
To prove the converse inclusion,
 suppose  that $x \in \tau^{-1}(C)$.
Then $y \coloneqq \tau(x) \in C$,
so that $C_\alpha(y) = C_\omega(y) = C$ by Proposition~\ref{p:irred-comp}.(vi).
It follows from Proposition~\ref{p:ca-sft-irred-compo}.(ii) that
\[
\tau(C_\alpha(x)) \subset C_\alpha(\tau(x)) = C_\alpha(y) = C,
\]
and
\[
\tau(C_\omega(x)) \subset C_\omega(\tau(x)) = C_\omega(y) = C. 
\]
Therefore, we have $\rho(C_\alpha(x)) = \rho(C_\omega(x)) = C$.
As $\rho$ is bijective by (i),
this implies  $C_\alpha(x) = C_\omega(x) = \rho^{-1}(C) = C'$.
We deduce that $x \in C'$ by applying again Proposition~\ref{p:irred-comp}.(vi).
This shows that $\tau^{-1}(C) \subset C'$ and completes the proof of (ii). 
\par
To prove (iii), let $C \in \CC$ and let $C' \coloneqq \rho(C)$.
We then have $\tau(C) \subset C'$ by definition of $\rho$.
As $\tau^{-1}(C') = \rho^{-1}(C') = C$ by (ii),
we conclude that $\tau(C) = C'$. This shows (iii).
\par
We have
\begin{align*}
\tau^{-1}(\NW(X))
&= \tau^{-1}\left(\bigcup_{C \in \CC} C\right) && \text{(by Proposition~\ref{p:irred-comp}.(ii))} \\
&= \bigcup_{C \in \CC} \tau^{-1}(C) \\  
&= \bigcup_{C \in \CC} \rho^{-1}(C) && \text{(by (ii))} \\
&= \bigcup_{C \in \CC} C && \text{(since $\rho$ permutes $\CC$)} \\
&= \NW(X) && \text{(by Proposition~\ref{p:irred-comp}.(ii)).}
\end{align*}
As $\tau$ is surjective, we deduce that $\tau(\NW(X)) = \NW(X)$. This shows (iv).
\par
Assertion (v)  follows from (iv) and the surjectivity of $\tau$.
\end{proof}

\begin{proposition}
\label{p:inj-tau-implies-inj-rho}
Let $A$ be a finite set, let $X \subset A^{\Z}$ be a subshift of finite type,
and let $\CC$ denote the poset of irreducible components of $X$.
Suppose that $\tau \colon X \to X$ is an injective endomorphism of $X$
and let $\rho = \rho_\tau \colon \CC \to \CC$ denote the poset endomorphism of $\CC$ associated with $\tau$.
Then the following hold:
\begin{enumerate}[\rm (i)]
\item
one has $\tau(\NW(X)) = \NW(X)$
and $\tau$ induces by restriction a topological conjugacy of $\NW(X)$ onto itself;
\item
the map $\rho \colon \CC \to \CC$ is a poset automorphism of $\CC$;
\item
for each $C \in \CC$,
one has $\tau(C) = \rho(C)$ and $\tau$ induces by restriction a topological conjugacy of $C$ onto $\rho(C)$;
\item
one has $\tau(X \setminus \NW(X)) \subset X \setminus \NW(X)$ and
$\tau^{-1}(X \setminus \NW(X)) = X \setminus \NW(X)$;
\item
$\tau$ is surjective if and only if
$\tau(X \setminus \NW(X)) = X \setminus \NW(X)$.
\end{enumerate}
\end{proposition}

\begin{proof}
It follows from Proposition~\ref{p:properties-endo}.(i) that 
$\tau(\NW(X)) \subset \NW(X)$.
Thus, $\tau$ induces, by restriction, an injective endomorphism
$\varphi \coloneqq \tau \vert_{\NW(X)} \colon \NW(X) \to \NW(X)$.
The subshift $\NW(X)$ is of finite type by Proposition~\ref{p:irred-comp}.(i)
and non-wandering by definition.
Therefore $\NW(X)$  is surjunctive by Corollary~1.4 in~\cite{csc-goe-smale}. 
It follows that $\varphi$ is surjective.
We deduce that $\varphi \colon \NW(X) \to \NW(X)$ is 
a topological conjugacy from $\NW(X)$ onto itself. This shows (i).
\par
Observe that the set of irreducible components of $\NW(X)$ is the same as that of $X$, namely $\CC$. 
We deduce from (i) that $\tau$ permutes the elements of $\CC$. 
It then follows from  Assertion (iii)  in Proposition~\ref{p:surj-tau-implies-surj-rho}
that $\rho$ is a poset automorphism of $\CC$.
Moreover,  $\tau$ induces by restriction
a topological conjugacy of $C$ onto $\rho(C)$ for every $C \in \CC$. 
This shows (ii) and (iii).
\par
Finally, (iv) and (v) follow from (i) and injectivity of $\tau$.
\end{proof}

\section{Markers and injunctivity}
\label{sec:markers}

The introduction of markers in symbolic dynamics goes back to Hedlund~\cite{hedlund} who used them in his investigation of  the automorphism group of full shifts.
The technique of markers  was extended by Boyle, Lind, and Rudolph~\cite{boyle-lind-rudolph}
and by Kim and Roush~\cite{kim-roush}
to study the automorphism group of topologically mixing subshifts of finite type (see also~\cite[Chapter~3]{kitchens}).
In this section, we use markers to construct endomorphisms of  subshifts of finite type having an infinite irreducible component
in order to show that  these subshifts are not injunctive.
For our purpose, we define a marker in the following way.
 
\begin{definition}
\label{def:marker}
Let $A$ be a set.
One says that a pair of words $(w_0,w_1) \in A^* \times A^*$ is a \emph{marker} if the following conditions are satisfied:
\begin{enumerate}[\rm (M1)]
\item
one has $w_0,w_1 \in A^n$ for some $n  \geq 3$;
\item
one has $w_0 \not= w_1$;
\item
one has $\pref_1(w_0) = \suff_1(w_0) = \pref_1(w_1) = \suff_1(w_1)$;
\item
for all $k \in \{2,\dots,n - 1\}$ and  $r,s \in \{0,1\}$,
one has $\suff_k(w_r) \not= \pref_k(w_s)$. 
\end{enumerate}
\end{definition}

Condition (M4) says that, for all $r,s \in \{0,1\}$,  the words $w_r$ and $w_s$ have only trivial overlaps.

\begin{example}
Let $A \coloneqq \{0,1\}$. Then the pair $(0100,0110)$ is a marker
while the pair $(0100,0010)$ is not a marker.
\end{example}

\begin{lemma}
\label{l:marking-exist}
Let $G = (V,E)$  be a finite directed graph and let $X \subset V^{\Z}$ denote the subshift of finite type associated with $G$.
Suppose that the subshift $X$ contains an infinite irreducible component $C$ and let $K$ denote the path-connected component of $G$ associated with $C$.
Then there exists a marker $(w_0,w_1) \in V^* \times V^*$ such that $w_0$ and $w_1$ are both cycles in $K$. 
\end{lemma}

\begin{proof} (cf.~\cite[Proof of Lemma~2.2]{boyle-lind-rudolph})
Since $C$ is infinite, there exist distinct vertices $v_0,v_1 \in K$ such that $v_0v_1 \in E$
(otherwise, being path-connected, $K$ would be reduced to a single vertex having a self-loop,
and $C$ would be reduced to a single constant configuration).
Let then $v_1\cdots v_h v_{h+1}$ be a path in $K$ going from $v_1$ to $v_0 = v_{h+1}$ of minimal length $h \geq 1$.
It follows that $\pi \coloneqq v_0v_1 \cdots v_hv_0$ is a simple cycle in $K$ of length $\ell_\pi = h+1$.
\par
From Proposition~\ref{p:finite-component-1-1}, we deduce that there exist $u,v \in K$ such that
$uv \in (E \cap K^2) \setminus \{v_iv_{i+1}: i=0,1, \ldots, h\}$.
We claim that there exist $i \in \{0,1,\ldots,h\}$ and $w \in K \setminus \{v_{i+1}\}$ such that $v_iw \in E$.
Indeed, if $u = v_i$ for some $i \in \{0,1, \ldots, h\}$, then  we can take $w \coloneqq v$.
Otherwise,  let $v'_0v'_1 \cdots v'_t$ be a path in $K$ of minimal length $t \geq 1$ connecting
$\pi$ to $u$. 
Thus, there exists $i \in \{0,1,\ldots,h\}$ such that $v'_0 = v_i$.
By minimality, we have $v'_1 \neq v_{i+1}$. One may then take $w \coloneqq v'_1$.
This proves the claim.
Up to relabeling the indices, we may assume that $i = 0$.
\par
Suppose first that $w \neq v_0$.
Let $u_1u_2 \cdots u_tv_{t'}$, $0 \leq t' \leq h$, be a path in $K$ of minimal length $t \geq 0$ connecting $w = u_1$ to $\pi$
(note that $t = 0$ corresponds to the case $u_1 = v_{t'}$ for some $t' \in \{2,\dots,h\}$).
Then $\zeta \coloneqq v_0u_1u_2 \cdots u_tv_{t'}v_{t'+1} \cdots v_hv_0$ is a simple cycle based at $v_0$ of length
$\ell_\zeta = t + h -t'+2$.
We then set
\[
w_0 \coloneqq \pi\# \pi\# \zeta \# \pi \#\zeta \#\zeta \ \mbox{ and } \ w_1 \coloneqq \pi\# \pi\# \zeta \# \zeta \# \pi \# \zeta.
\]
Observe that the words $w_0,w_1 \in V^*$ are such that  
\[
|w_0| = |w_1| = n \coloneqq  3(\ell_\pi + \ell_\zeta) + 1 = 6h+3(t-t')+ 10 \geq 3,
\] 
$w_0 \not= w_1$, and $\pref_1(w_0) = \suff_1(w_0) = \pref_1(w_1) = \suff_1(w_1) = v_0$.
It follows that $(w_0,w_1)$ satisfies Conditions (M1), (M2), and (M3) in Definition~\ref{def:marker}. 
\par
Suppose now that $w = v_0$. This means that $\theta \coloneqq v_0w = v_0v_0$ is a self-loop at $v_0$ in $K$,
 of length $\ell_\theta = 1$. 
 We then set
\[
w_0 \coloneqq \pi\# \pi\# \theta \# \pi \#\theta \#\theta \ \mbox{ and } \ w_1 \coloneqq \pi\# \pi\# \theta \# \theta \# \pi \# \theta.
\]
Observe that the words $w_0, w_1 \in V^*$ are such that
\[  
|w_0| = |w_1| =  n \coloneqq  3(\ell_\pi + \ell_\theta) + 1 = 3h + 7 \geq 3,
\] 
$w_0 \not= w_1$,
and $\pref_1(w_0) = \suff_1(w_0) = \pref_1(w_1) = \suff_1(w_1) = v_0$.
It follows that $(w_0,w_1)$ satisfies Conditions (M1), (M2), and (M3) in Definition~\ref{def:marker}. 
\par
We claim that, in either case,  $(w_0,w_1)$ also satisfies condition (M4) in Definition~\ref{def:marker}.
First note that, by minimality, the vertex $v_0$ only occurs at the beginning and at the end of the word $\pi$
(resp.~$\zeta$, resp.~$\theta$). 
Consequently, $v_0$ occurs in $w_0$ and $w_1$ exactly $7$ times,
say at positions $1 = i_1 < i_2 < \cdots < i_7 = n$ and
$1 = j_1 < j_2 < \cdots < j_7 = n$, respectively.
Thus, when comparing $\suff_k(w_r)$ and $\pref_k(w_s)$ for $k \in \{2,\dots,n-1\}$ and $r,s \in \{0,1\}$, 
we may limit ourselves to the  cases corresponding to $k = n - i_m$ (resp. $k = n - j_m$) for $m= 2,3, \ldots, 6$.
\par
In order to prove the claim, it thus suffices to look at the following tables, where in the first line we have factorized
the word $w_r$ into blocks (made up of $\pi$s, $\zeta$s, or $\theta$s) and in the subsequent five lines we have factorized
the blocks (also made up of $\pi$s, $\zeta$s, or $\theta$s) for the five meaningful prefixes of $w_s$.
A comparison of the columns shows that there is always a mismatch for some position of the corresponding blocks.

\begin{table}[h!]
\centering
\renewcommand{\arraystretch}{1.2}
\begin{tabular}{|c|c|c|c|c|c|}
\hline
$\pi$ & $\pi$ & $\zeta$ & $\pi$ & $\zeta$ & $\zeta$\\
\hline
& $\pi$ & $\pi$ & $\zeta$ & $\pi$ & $\zeta$\\
& \ & $\pi$ & $\pi$ & $\zeta$ & $\pi$\\
& \ & \ & $\pi$ & $\pi$ & $\zeta$\\
& \ & \ & \ & $\pi$ & $\pi$\\
& \ & \ & \ & \ & $\pi$\\
\hline
\end{tabular}
\vspace{0.5cm}
\caption{$w_0$ has no nontrivial overlaps with itself: $\suff_k(w_0) \neq \pref_k(w_0)$ for $k \neq 1, n$.}
\end{table}

\begin{table}[h!]
\centering
\renewcommand{\arraystretch}{1.2}
\begin{tabular}{|c|c|c|c|c|c|}
\hline
$\pi$ & $\pi$ & $\zeta$ & $\pi$ & $\zeta$ & $\zeta$\\
\hline
& $\pi$ & $\pi$ & $\zeta$ & $\zeta$ & $\pi$\\
& \ & $\pi$ & $\pi$ & $\zeta$ & $\zeta$\\
& \ & \ & $\pi$ & $\pi$ & $\zeta$\\
& \ & \ & \ & $\pi$ & $\pi$\\
& \ & \ & \ & \ & $\pi$\\
\hline
\end{tabular}
\vspace{0.5cm}
\caption{$w_0$ has no nontrivial overlaps with $w_1$: $\suff_k(w_0) \neq \pref_k(w_1)$ for $k \neq 1, n$. This covers the case $w \neq v_0$. By replacing $\zeta$ by $\theta$, the case $w = v_0$ is covered as well.}
\end{table}

\begin{table}[h!]
\centering
\renewcommand{\arraystretch}{1.2}
\begin{tabular}{|c|c|c|c|c|c|}
\hline
$\pi$ & $\pi$ & $\zeta$ & $\zeta$ & $\pi$ & $\zeta$\\
\hline
& $\pi$ & $\pi$ & $\zeta$ & $\zeta$ & $\pi$\\
& \ & $\pi$ & $\pi$ & $\zeta$ & $\zeta$\\
& \ & \ & $\pi$ & $\pi$ & $\zeta$\\
& \ & \ & \ & $\pi$ & $\pi$\\
& \ & \ & \ & \ & $\pi$\\
\hline
\end{tabular}
\vspace{0.5cm}
\caption{$w_1$ has no nontrivial overlaps with itself: $\suff_k(w_1) \neq \pref_k(w_1)$ for $k \neq 1, n$. This covers the case $w \neq v_0$. By replacing $\zeta$ by $\theta$, the case $w = v_0$ is covered as well.}
\end{table}

This proves the claim.
\par
We deduce that $(w_0,w_1)$ is a marker.
\end{proof}

\begin{lemma}
\label{l:not-injunctive}
Let $G = (V,E)$  be a finite directed graph and let $X \subset V^{\Z}$ denote the subshift of finite type associated with $G$.
Suppose that  there exists a marker $(w_0,w_1) \in V^* \times V^*$ such that $w_0$ and $w_1$ are both cycles in $G$.
Then the subshift $X$ is not injunctive.
\end{lemma}

\begin{proof}
Denote by  $\ell$  the common length of the cycles $w_0$ and $w_1$ (so that $\ell = |w_0| - 1 = |w_1| -1$).
\par
Given a configuration $x \in X$, define the configuration $\tau(x)$ as follows. 
For  $i \in \Z$, if there exist  $r,s \in \{0,1\}$
such that $x_{[i,i+2\ell]} = w_{r}\# w_{s}$, 
define $\tau(x)_{[i,i+ \ell]} \coloneqq w_t$, where $t \coloneqq r+s \mod 2$. 
All other entries of $x$ remain unchanged.
The words  $w_0$ and $w_1$ have only trivial overlaps by (M4).
Moreover, by (M3), there is a vertex $v \in V$ such that $w_0$ and $w_1$ are both cycles based at $v$.
It follows that the configuration  $\tau(x)$ is well defined and that one has $\tau(x) \in X$.
\par
The map $\tau \colon X \to X$ is an endomorphism of $X$ since it clearly satisfies condition (b) in Proposition~\ref{p:char-ca} for $m = 2\ell+1$.
\par
Consider now the periodic configurations  $x^{(0)} \in X$ (resp.~$x^{(1)} \in X$) associated with the cycle
$w_{0}$
(resp.~$w_1$).
We have $x^{(0)} \not= x^{(1)}$ since $w_0 \not= w_1$ by (M2).
As $\tau(x^{(0)}) = \tau(x^{(1)}) = x^{(0)}$,
it follows that $\tau$ is not injective.
\par
Let us show that $\tau$ is surjective.
Let $y \in X$.
If there are no occurrences in $y$ of any of the words $w_{r} \# w_{s}$, $r,s \in \{0,1\}$,   then
$y = \tau(y)$. 
Otherwise, consider the set $S \subset \Z$ consisting of all $i \in \Z$ such that there exists $j \in \Z$ such that $j \leq i \leq j + 2\ell$ and $x_{[j,j+2\ell]} = w_{r}\# w_{s}$ for some $r,s \in \{0,1\}$.
We construct $x \in X$ such that $\tau(x) = y$ as follows.
If $i \in \Z \setminus S$, we set $x_i \coloneqq y_i$.
Suppose now $i \in S$.
Let $I$ be the maximal interval of $\Z$  such that $i \in I \subset S$.
The interval $I$ is of one of the following forms:
\begin{enumerate}
\item
$I = [a,a+k \ell ]$ for some $a,k \in \Z$ with $k \geq 2$;
\item
$I = [a,\infty)$ for some $a \in \Z$;
\item
$I = (-\infty,a]$ for some $a \in \Z$;
\item
$I = \Z$.   
\end{enumerate}
When $I$ is of the first form, 
we have $y_I = w_{r_1} \# \cdots \# w_{r_k}$, where $r_j \in \{0,1\}$ for all $j \in \{1,\dots,k\}$.
We then define $x_I$ by replacing $y_I$ by
$w_{s_1} \# \cdots \# w_{s_k}$, where
$s_k \coloneqq r_k$ and $s_j \coloneqq r_j + r_{j + 1} \mod 2$ for all $j \in \{1, \dots , k - 1\}$.
\par
When  $I$ is of the second form, we have 
$y_I = w_{r_0} \# w_{r_1} \# w_{r_2} \# \cdots $, where $r_j \in \{0,1\}$ for all $j \in \N$.
We then define $x_I$ by replacing $y_I$ by
$w_{s_0} \# w_{s_1} \# w_{s_2} \# \cdots $, where
$s_0 \coloneqq r_0$ and $s_j \coloneqq s_{j - 1} + r_j \mod 2$ for all $j \in \{1,2,\dots\}$.
\par
When $I$ is of the third form, we have
$y_I = \cdots \# w_{r_{-2}} \# w_{r_{-1}} \# w_{r_0}$, where $r_{-j} \in \{0,1\}$ for all $j \in \N$. 
We then define $x_I$ by replacing $y_I$ by
$\cdots \# w_{s_{-2}} \# w_{s_{-1}} \# w_{s_0}$, where
$s_0 \coloneqq r_0$ and $s_j \coloneqq s_{j + 1} + r_j \mod 2$ for all $j \in \{-1,-2,\dots\}$.
\par
Finally, when $y$ is of the fourth form, we have $S = \Z$ and 
$y = \cdots \# w_{r_{-1}} \# w_{r_0} \# w_{r_1} \#\cdots$ (the precise position of $y_0$ is irrelevant)
and in this case we take
$x \coloneqq  \cdots \# w_{s_{-1}} \# w_{s_0} \# w_{s_1} \#\cdots$, where
$s_0 \coloneqq r_0$,
$s_j \coloneqq s_{j - 1} + r_j \mod 2$ for all $j \in \{1,2,\dots\}$,
and $s_j \coloneqq s_{j + 1} + r_j \mod 2$ for all $j \in \{-1,-2,\dots\}$. 
   \par
We then clearly have  $x \in X$ and $\tau(x) = y$. 
This shows that $\tau$ is surjective.
As $\tau$ is not injective, we conclude that $X$ is not injunctive.
\end{proof}

\begin{proposition}
\label{p:not-injunctive-if-inf-irred}
Let $A$ be a finite set and let $X \subset A^{\Z}$ be a subshift of finite type.
Suppose that $X$ contains an infinite irreducible component.
Then the subshift $X$ is not injunctive.
\end{proposition}

\begin{proof}
As we can assume that $X$ is the vertex subshift associated with a finite directed graph $G = (V,E)$,
this immediately follows from Lemma~\ref{l:marking-exist} and Lemma~\ref{l:not-injunctive}.
\end{proof}

\section{Construction of $\tau_C$}
\label{sec:construction-tau-C}

Let $G = (V,E)$ be a finite directed graph and let $X \subset V^{\Z}$ denote the subshift of finite type associated with $G$.
Fix some irreducible component $C$ of $X$.
In this section, we construct an injective endomorphism $\tau_C \colon X \to X$.
The fact that it may fail to be surjective (depending on the structure of the poset of irreducible components if $X$)
will be used in Section~\ref{sec:char-surj-sft} for proving  properties of surjunctive subshifts of finite type stated in Theorem~\ref{t:char-surj-sft}.
The construction of $\tau_C$ and the proof that it is an injective endomorphism are divided into several steps. \\  

\noindent
{\bf Step 1.}
Denote by $R$ the set of recurrent vertices of $G$.
Then there exists an integer $\ell = \ell(G) \geq 1$ such that, for each $v \in R$,
there is a cycle of length $\ell$ based at $v$ in $G$.
Indeed, by definition of recurrence, for each $v \in R$, we can find a cycle  based at $v$.
Denoting by $\ell_v$ the length of such a cycle,
we can take $\ell \coloneqq \lcm \{\ell_v : v \in R\}$.
In the sequel, we fix for each $v \in R$ a \emph{preferred} cycle $\pi_v$ of length $\ell$ based at $v$ in $G$.
Note that if $K$ is the unique path-component of $G$ containing $v$,
then all the vertices of $\pi_v$ lie in $K$. \\

\noindent
{\bf Step 2.}
Fix some  irreducible component $C$ of $X$.
Denote by $\KK$ the poset of path-components of $G$. 
Let $K \in \KK$ denote the path-component of $G$ associated with $C$.
Set  $\KK_1 \coloneqq \{K' \in \KK: K' \prec K\}$ and $\KK_2 \coloneqq \KK \setminus \KK_1$.
Note that if $K' \in \KK_1$ (resp.~$K' \in \KK_2$) and $K'' \in \KK$ satisfy $K'' \preceq K'$ (resp.~$K' \preceq K''$) then we have
$K'' \in \KK_1$ (resp.~$K'' \in \KK_2$). \\
Denote by $V_1 \subset V$  the set of vertices $v$ such that $v \prec u$ for some (equivalently, all) $u \in K$ and
set $V_2 \coloneqq V \setminus V_1$. 
Also set $R_1 \coloneqq R \cap V_1$ and $R_2 \coloneqq R \cap V_2$. \\
Observe that $R_1 = \bigcup_{K’ \in \KK_1} K’$ and
$R_2 = \bigcup_{K’ \in \KK_2} K’$.\\
\par

\noindent
{\bf Step 3.}
We are now in a position to define the map $\tau = \tau_C \colon X \to X$ we alluded to above.
Let $x \in X$ and let $\Gamma = \{K_0,K_1, \ldots, K_n\}$, where $n \geq 0$ and $K_0 \prec K_1 \prec \dots \prec K_n$, denote the chain of path-components of $x$.
Then we are in one and only one of the following three cases:

\begin{enumerate}[{\rm {Case} 1:}]
\item  The terminal path-component $K_n$ of $x$ is in $\KK_1$.
Note that under this assumption, all path-components $K_j$s of $x$ lie in $\KK_1$ as well.
We then set $\tau(x) \coloneqq \sigma^\ell(x)$, where the integer $\ell = \ell(G) \geq 1$ is given by Step 1.

\item  The initial path-component $K_0$ of $x$ is in $\KK_2$.
Note that under this assumption, all path-components $K_j$s of $x$ lie in $\KK_2$ as well.
We then set $\tau(x) \coloneqq \sigma^{-\ell}(x)$.

\item  The initial path-component $K_0$ of $x$ is in $\KK_1$ and its terminal path-component $K_n$ is in $\KK_2$.
Then, with the notation in \eqref{e:normal-form-irred}, we have $n = n(x) \geq 1$
and there exists a unique integer $h = h(x) \in \{0,1,\dots,n-1\}$  such that
$K_j \in \KK_1$ for all $0 \leq j \leq h$ and $K_j \in \KK_2$ for all $h+1 \leq j \leq n$.
\par
Let now $u \coloneqq x_{b_h} \in R_1$ (resp.\ $v \coloneqq x_{a_{h+1}} \in R_2$) and let 
$\pi_u = c_0 \cdots c_\ell$
(resp.\ $\pi_v = c'_0\cdots c'_\ell$) denote the preferred cycle of length $\ell$ based at $u$ (resp.\ $v$).
As remarked above, we have $c_0,\ldots,c_\ell \in K_h$ (resp.\ $c'_0,\ldots, c'_\ell \in K_{h+1}$).
We then set, for every $i \in \Z$,
\begin{equation}
\label{e:tau-x}
\tau(x)_i \coloneqq
\begin{cases}
x_{i+\ell} & \text{ if } i \in (-\infty, b_h - \ell],\\
c_{i - b_h+\ell} & \text{ if } i \in [b_h-\ell, b_h - 1],\\
x_i & \text{ if } i \in [b_h , a_{h+1}],\\
c'_{i - a_{h+1}} & \text{ if } i \in [a_{h+1}, a_{h+1}+\ell - 1],\\
x_{i-\ell} & \text{ if } i \in [a_{h+1}+\ell, +\infty).
\end{cases}
\end{equation}
\end{enumerate}

\begin{remark}
\label{r:same-chain}
In either case, the chain of path-components of $\tau(x)$ is the same as that of $x$.
Observe that in Case~1 (resp.~Case~2), we have $x_i \in V_1$ (resp.~$ x_i \in V_2$) for all $i \in \Z$. 
In Case~3, 
we have   $h(\tau(x)) = h(x)$,   $b_h(\tau(x)) = b_h(x)$, and $a_{h+1}(\tau(x)) = a_{h+1}(x)$.
In Case~3, all vertices $x_i$ which are in $V_1$ (resp.\ in $V_2 \setminus [b_h+1, a_{h+1}]$) are shifted to the
left (resp.\ to the right) by $\ell$ positions in $\tau(x)$.
\end{remark}

It is clear that $\tau(x) \in X$ for all $x \in X$ (in Cases 1 and 2, the configuration $\tau(x)$ is obtained from $x$ by shifting it while, in Case 3, the configuration $\tau(x)$ is obtained from $x$ by inserting the two cycles $\pi_u$ and $\pi_v$ based at vertices $u$ and $v$, respectively). \\

\noindent
{\bf Step 4.}
Let us show that the map $\tau \colon X \to X$ is injective.
Let $x,y \in X$ such that $y = \tau(x)$.
As observed above, the configurations $x$ and $y$ have the same initial path-component $K_0 \coloneqq K_0(x) = K_0(y)$
and the same terminal path-component $K_n \coloneqq K_n(x) = K_n(y)$, where $n \coloneqq n(x) = n(y)$.
We distinguish the following mutually exclusive cases.

\begin{enumerate}[{\rm {Case} 1:}]
\item $K_n \in \KK_1$. Then, we have $y = \sigma^\ell(x)$, so that $x = \sigma^{-\ell}(y)$.
\item $K_0 \in \KK_2$. Then, we have $y = \sigma^{-\ell}(x)$, so that $x = \sigma^{\ell}(y)$.
\item $K_0 \in \KK_1$ and $K_n \in\KK_2$. Then, we deduce from \eqref{e:tau-x}  that
\begin{equation*}
x_i =
\begin{cases}
y_{i-\ell} & \text{ if } i \in (-\infty, b_h],\\
y_i & \text{ if } i \in [b_h+1, a_{h+1}],\\
y_{i+\ell} & \text{ if } i \in [a_{h+1}+1, +\infty).
\end{cases}
\end{equation*}
\end{enumerate}
In each case, the configuration $x$ is uniquely determined by the configuration $y$.
It follows that $\tau$ is injective.\\

\noindent
{\bf Step 5.}
Let us show that $\tau$ commutes with the shift map $\sigma$.
Let indeed $x \in X$. Observe that the chain of path-components of $\sigma(x)$ is the same as that of $x$.
In particular, $x$ and $\sigma(x)$ have the same initial path-component $K_0 \coloneqq K_0(x) = K_0(\sigma(x))$
and the same terminal path-component $K_n \coloneqq K_n(x) = K_n(\sigma(x))$, where $n \coloneqq n(x) = n(\sigma(x))$.
\par
We distinguish the same mutually exclusive cases as above.

\begin{enumerate}[{\rm {Case} 1:}]
\item $K_n \in \KK_1$.  
Then $\tau(x) = \sigma^\ell(x)$ and $\tau(\sigma(x)) = \sigma^\ell(\sigma(x)) = \sigma^{\ell + 1}(x)$.
We thus have $(\sigma \circ \tau)(x) = \sigma(\sigma^\ell(x)) = \sigma^{\ell + 1}(x) =  \tau(\sigma(x)) = (\tau \circ \sigma)(x)$.

\item $K_0 \in \KK_2$. 
Then $\tau(x) = \sigma^{-\ell}(x)$ and $\tau(\sigma(x)) = \sigma^{-\ell}(\sigma(x)) = \sigma^{-\ell + 1}(x)$.
We thus have $(\sigma \circ \tau)(x) = \sigma(\sigma^{-\ell}(x)) = \sigma^{-\ell + 1}(x) =  \tau(\sigma(x)) = (\tau \circ \sigma)(x)$.
\item $K_0 \in \KK_1$ and $K_n \in \KK_2$. 
We then have $h(\sigma(x)) = h(x)$,
 $a_j(\sigma(x)) = a_j(x)-1$ for all $j=1,2,\ldots,n$,
and $b_j(\sigma(x)) = b_j(x)-1$ for all $j=0,1,\ldots,n-1$.
From \eqref{e:tau-x}, we deduce that $(\tau \circ \sigma)(x) = \tau(\sigma(x)) = \sigma(\tau(x)) = (\sigma \circ \tau)(x)$.
\end{enumerate}
This shows that $(\tau \circ \sigma)(x) = (\sigma \circ \tau)(x)$ for all $x \in X$.
Therefore, we have $\tau \circ \sigma = \sigma \circ \tau$.\\

\noindent
{\bf Step 6.} 
Let $x \in X$ and $y \coloneqq \tau(x)$. 
We claim that $y_0$ only depends on the word $x_{[-\ell,\ell]} \in V^{2 \ell + 1}$.
We distinguish the following three cases.
\begin{enumerate}[{\rm (1)}]
\item $x_\ell  \in V_1$. 
Then, we have (cf.\ Remark~\ref{r:same-chain})
\[
y_0 = \sigma^\ell(x)_0 = x_{\ell}.
\]
\item $x_{-\ell} \in V_2$. Then, we have (cf.\ Remark~\ref{r:same-chain}):
\[
y_0 = \sigma^{-\ell}(x)_0 = x_{-\ell}.
\]
\item $x_{-\ell} \in V_1$ and $x_\ell \in V_2$.
We are then in Case 3.
Let $\Gamma(x) = \{K_0,K_1,  \ldots, K_n\}$, with $K_0 \prec K_1 \prec \dots \prec K_n$, denote the chain of path-components
of $x$ and recall that $ h \in \{0,1,\dots,n-1\}$  is the unique index such that $K_h \in \KK_1$ and $K_{h+1} \in \KK_2$.
Recall also that $b_h$ (resp.\ $a_{h+1}$) is the largest (resp.\ least) integer such that $x_{b_h} \in K_h$ (resp.\
$x_{a_{h+1}} \in K_{h+1}$).
Thus, we have $b_h < \ell$ and $a_{h + 1} > - \ell$ in this case.

\par
We distinguish the following subcases:
\begin{enumerate}[{\rm (3.1):}]
\item {$0 \leq b_h$.}
We have
\[
y_{[-\ell, b_h]} = x_0 \cdots x_{b_h} c_1c_2 \cdots c_\ell.
\]
Observe that the values $c_1,c_2, \dots, c_\ell$ are uniquely determined by $x_{b_h}$.
As $-\ell \leq 0 \leq b_h < \ell$, the claim follows in this case. \\

\item {$b_h \leq 0$ and $a_{h+1} \leq 0$.}
We have
\[
y_{[a_{h+1}, 2\ell]} = x_{a_{h+1}}c'_1c'_2 \cdots c'_\ell x_{a_{h+1}+1} \cdots x_{\ell}.
\]
Observe that the values $c'_1,c'_2, \dots, c'_\ell$ are uniquely determined by $x_{a_{h+1}}$.
As $-\ell < a_{h+1} \leq 0 \leq 2\ell$, the claim follows in this case. \\

\item {$b_h \leq 0$ and $0 \leq a_{h+1} \leq \ell$.}
We have
\[
y_{[b_h, a_{h+1}]} = x_{b_h} \cdots x_{a_{h+1}}.
\]
As $-\ell < b_h \leq 0 \leq a_{h+1} \leq \ell$, the claim follows in this case.

\item {$b_h \leq 0$ and $\ell < a_{h+1}$.}
We have
\[
y_{[b_h, \ell]} = x_{b_h} \cdots x_\ell.
\]
As $- \ell < b_h \leq 0 \leq \ell$, the claim follows in this case as well. \\
\end{enumerate}
\end{enumerate}
This completes the proof that $\tau(x)_0$ only depends on the word $x_{[-\ell,\ell]}$. \\

\noindent
{\bf Step 7.}
It follows from the previous step that there is a map $\mu \colon V^{2\ell + 1}  \to V$ such that
$\tau(x)_0 = \mu(x_{[-\ell,\ell]})$ for all $x \in X$.
As $\tau$ commutes with the shift (cf.\ Step 5), 
we deduce from Proposition~\ref{p:char-ca} that $\tau \colon X \to X$ is an endomorphism of $X$.

\section{Construction of $\tau^C$}
\label{sec:construction-tauC}

Let $G = (V,E)$ be a finite directed graph and let $X \subset V^{\Z}$ denote the subshift of finite type associated with $G$.
Let $C$ be an infinite irreducible component of $X$ which is minimal in the poset of all irreducible components of $X$.
In this section, we construct a surjective endomorphism $\tau^C \colon X \to X$.
The fact that it may fail to be pre-injective (depending on the structure of the poset of irreducible components if $X$) will be used in Section~\ref{sec:char-surj-sft}
for proving the properties of Moore subshifts of finite type stated in Theorem~\ref{t:char-surj-sft}.
\par
We first need some auxiliary lemmas of a combinatorial nature.
\par
\begin{lemma}
\label{l:2-cycles-de-long-n}
Let $K=(V,E)$ be a finite strongly connected directed graph.
Suppose that the associated subshift of finite type $X \subset V^{\Z}$ is infinite.
Then there exists an integer $n = n(K) \geq 1$ such that for each $v \in V$
there are two distinct cycles of length $n$ based at $v$ in $K$.
\end{lemma}
\begin{proof}
Fix $v_0 \in V$. Then there exists $v_1 \in K \setminus \{v_0\}$ such that $v_0v_1 \in E$
(otherwise, being path-connected, $K$ would be reduced to the single vertex $v_0$ with a self-loop, and $X$ would consist only of the configuration
with constant value $v_0$, contradicting the hypothesis that $X$ is infinite).
Let then $v_1\cdots v_h v_0$ be a path in $K$ going from $v_1$ to $v_0$ of minimal length $h \geq 1$.
As in the proof of Lemma~\ref{l:marking-exist}, we can find $i \in \{0,1,\ldots,h\}$ and
$w \in K \setminus \{v_{i+1}\}$ such that $v_iw \in E$. As $K$ is path-connected, we can find a path
$w_0 \cdots w_k$ in $K$ going from $w = w_0$ to $v_0 = w_k$ of minimal length $k \geq 1$.
Then, $v_0v_1 \cdots v_hv_0$ and $v_0v_1 \cdots v_iw_0w_1 \cdots w_{k-1}v_0$ are distinct cycles in $K$ and
$(v_0v_1 \cdots v_hv_0)^{s\sharp} \neq (v_0 \cdots v_iw_0w_1 \cdots w_{k-1}v_0)^{t\sharp}$ for all integers $s,t \geq 1$.
In particular, taking $s \coloneqq i+k+1$ and $t \coloneqq h+1$ yields two distinct cycles based at $v_0$ of the same
length $n_{v_0} = (h+1)(i+k+1)$.
\par
One may then take $n \coloneqq \lcm(n_v: v \in V)$.
\end{proof}

\begin{remark}
We may also deduce Lemma~\ref{l:2-cycles-de-long-n} by using results from the Perron-Frobenius theory of non-negative matrices 
(cf.~\cite[Chapter~4]{lind-marcus-second}, \cite[Section~1.3]{kitchens}) in the following way.
\par
The \emph{adjacency matrix} of $K = (V,E)$ is the matrix $A = (A_{uv})_{u,v \in V}$ where $A_{uv} = 1$ if $uv \in E$ and $A_{uv} = 0$ otherwise.
For every integer $m \geq 1$ and any $u,v \in V$, the $uv$-entry $(A^m)_{uv}$ of the matrix $A^m$ is  the number of paths of length $m$ in $K$
going from $u$ to $v$. 
Since $K$ is strongly connected, the matrix $A$ is irreducible
(i.e., for all $u,v \in V$, there exists an integer
$m = m(u,v) \geq 1$ such that $(A^m)_{uv} \geq 1$). 
The entropy of the vertex subshift $X$ associated with $K$ is given by $h(X) = \log \lambda$,
where $\lambda$ is the Perron-Frobenius eigenvalue  of $A$ 
(cf.~\cite[Theorem~4.3.1 and Theorem~4.3.3]{lind-marcus-second}).
Since $X$ is infinite, we have $\lambda > 1$
(cf.~\cite[Exercise~4.3.1 and Corollary~4.4.9]{lind-marcus-second}). 
The \emph{period} of $A$ is the integer $p \geq 1$ defined by $p \coloneqq \gcd\{n \geq 1: (A^n)_{vv} > 0\}$, where $v \in V$ 
(irreducibility of $A$
implies that $p$  is finite and does not depend on the choice of  $v$).
If $p=1$, the matrix  $A$ is primitive
(i.e., there exists an integer $m \geq 1$ such that all entries of $A^m$ are positive).
\par
There exists a partition $V = V_1 \sqcup \cdots \sqcup V_p$ and primitive irreducible matrices
$B_k = \left((B_k)_{uv}\right)_{u,v \in V_k}$, $k=1,\ldots,p$, each having Perron-Frobenius eigenvalue  $\lambda^p > 1$,
such that, up to  reindexing, the matrix $A^p$ admits the diagonal block  decomposition
\[
A^p =
\begin{pmatrix}
B_1 & 0 & \cdots & 0 \\
0 & B_2 & \cdots & 0 \\
\vdots & \vdots & \ddots & \vdots \\
0 & 0 & \cdots & B_p
\end{pmatrix}.
\]
It follows from the Perron-Frobenius theory for primitive matrices that,
for all $u,v \in V_k$ and $k=1,\ldots,p$,
 the sequence $(B_k^m)_{uv}/\lambda^m$ converges  to a finite positive limit 
as $m \to \infty$ 
(cf.~\cite[Theorem~4.5.12]{lind-marcus-second}, \cite[Theorem~1.3.2(f)]{kitchens}). 
Since $\lambda > 1$, we deduce that there exists an integer $m \geq 1$ such that
$(B_k^m)_{uv} \geq 2$
for all $u,v \in V_k$ and $k=1,\ldots,p$.
 Now, for every $v\in V$, the entry $(A^{pm})_{vv}$ belongs to one of the diagonal blocks $B_k^m$.
Hence $(A^{pm})_{vv} \geq 2$ for all $v \in V$, so that we can take $n \coloneqq pm$. 
\end{remark}

\begin{lemma}
\label{l:K-n-N-ell}
Let $K=(V,E)$ be a finite directed graph and let $n$ be a positive integer.
Then there exist positive integers $\ell = \ell(K,n)$ and $N = N(K,n)$ such that,
for every path $v_0v_1\cdots v_N$ of length $N$ in $K$, there exist integers $t,d\geq 1$
and $0\leq i_1<\cdots<i_t\leq N$ such that
\begin{itemize}
\item $i_j+dn \leq i_{j+1}$ for $j=1,\dots,t-1$;
\item $v_{i_j}=v_{i_j+dn}$ for $j=1,\dots,t$;
\item $tdn=\ell$.
\end{itemize}
\end{lemma}
\begin{proof}
Let $m \coloneqq |V|$ and set $M \coloneqq mn$ and $\ell \coloneqq \lcm(n,2n,\dots,Mn)$.
Note that $\ell$ is a multiple of $n$ and that $\ell/(dn)$ is a positive integer for all $d=1,2, \ldots, M$.
Define $B \coloneqq 1 +\sum_{d=1}^{M} (\ell/(dn) -1)$ and set $N \coloneqq BM$.
\par
Let $v_0v_1\cdots v_N$ be a path of length $N$ in $K$.
Consider the $B$ consecutive subpaths $\pi_r \coloneqq v_{rM}v_{rM+1}\cdots v_{rM+M}$, $r=0,\ldots,B-1$, of length $M$.
For each $r=0,\ldots,B-1$ consider the $M+1$ pairs $(v_{rM+a}, rM+a \bmod n) \in V \times \Z/n\Z$, $a=0,\ldots,M$.
Since $|V \times \Z/n\Z| = mn = M$, it follows from the pigeonhole principle that there exist indices
$rM \leq k_r < k'_r \leq rM+M$ such that $v_{k_r}=v_{k'_r}$ and $k_r \equiv k'_r \pmod n$.
Therefore $\ell_r \coloneqq k'_r - k_r$ is a positive multiple of $n$. Write $\ell_r = d_r n$, where $1 \leq d_r \leq M$.
\par
We claim that there exists some $d\in\{1,\dots,M\}$ such that $d_r=d$ for at least $t \coloneqq \ell/(dn)$ values of $r$.
Indeed, assume, for contradiction, that for every $d \in \{1,\dots,M\}$ the value $d$ occurs strictly fewer than $t$
times among the $d_r$'s. Since the number of occurrences is an integer, this means that each $d$ occurs at most
$t-1$ times. Summing over all possible values of $d$, we obtain the bound $B \leq \sum_{d=1}^{M}(\ell/(dn) -1)$,
a contradiction with the definition of $B$. This proves the claim.
\par
Let then $r_1 < r_2 < \cdots < r_t$ be values of $r$ such that $d_{r_j}  = d$,
and set $i_j \coloneqq k_{r_j}$ for $j=1,\dots,t$.
Then $0\leq i_1 < \cdots < i_t\leq N$ and $tdn = \ell$.
Moreover, since the chosen repetitions lie in the distinct subpaths $\pi_{r_j}$, $j=1,\dots,t$, we have
$i_j+dn = k_{r_j}+dn = k'_{r_j} \leq r_jM+M \leq r_{j+1}M \leq k_{r_{j+1}} = i_{j+1}$
for $j=1,\dots,t-1$.
Finally, by construction, $v_{i_j}=v_{k_{r_j}} = v_{k_{r_j}+dn} = v_{i_j+dn}$ for all $j=1,\dots,t$.
\end{proof}

\begin{remark}
\label{r:unique-admissible-pair}
Let $K=(V,E)$ be a finite directed graph and let $n$ be a positive integer.
Fix $\ell=\ell(K,n)$ and $N=N(K,n)$ as in Lemma~\ref{l:K-n-N-ell}.
\par
For each path $v_0v_1\cdots v_N$ of length $N$ in $K$, let ${\mathcal A}(v_0v_1\cdots v_N)$ denote the set of all pairs
$\bigl(t,(i_j)_{j=1}^t\bigr)$ satisfying the conclusions of the lemma for some positive integer $d$.
Note that $d$ is uniquely determined by $t$, since $tdn=\ell$.
Moreover, $\mathcal A(v_0v_1\cdots v_N)$ is finite and, by the lemma, it is nonempty.
\par
We select a unique admissible pair as follows.
First, choose the smallest admissible value of $t$.
Among the pairs with this value of $t$, choose the tuple
$(i_1,\dots,i_t)$ that is maximal in the right lexicographic order,
that is, comparing first $i_t$, then $i_{t-1}$, and so on.
Since this order is total, the selection determines a unique pair,
which we denote by $T(v_0v_1 \cdots v_N) \in {\mathcal A}(v_0v_1 \cdots v_N)$.
\par
For example, suppose that $\ell=dn$ for some positive integer $d$,
and let $u_0u_1\cdots u_n$ be a cycle of length $n$ in $K$, based at $u_0=u_n$.
Suppose that $N\geq\ell$ and that the terminal subpath
$v_{N-\ell}\cdots v_N$ consists of $d$ consecutive traversals of this cycle.
Then $v_{N-\ell}=v_N$, so the pair $(1,(N-\ell))$ is admissible.
Consequently, the smallest admissible value of $t$ is $1$.
For this value, the condition $i_1+dn\leq N$ becomes $i_1\leq N-\ell$.
Thus the maximal admissible starting index is $N-\ell$, and $T(v_0v_1\cdots v_N)=(1,(N-\ell))$.
\end{remark}

We are now in a position to introduce the map $\tau^C \colon X \to X$ we alluded to above.
Its construction and the proof that it is a surjective endomorphism are divided into several steps. \\  

\noindent
{\bf Step 1.}
Let $K$ denote the path-component of $G$ associated with $C$.
\par
Let $n = n(K)$ be as in Lemma~\ref{l:2-cycles-de-long-n}, and let $N = N(K,n)$ and $\ell = \ell(K,n)$ be the positive integers
provided by Lemma~\ref{l:K-n-N-ell}.
\par
Define a map $\tau = \tau^C \colon X \to X$ as follows.
Let $x \in X$ and let $\{K_0,K_1, \ldots, K_m\}$, where $K_0 \prec K_1 \prec \dots \prec K_m$, denote its chain of path-components. 
If $K_0\neq K$, we set $\tau(x) \coloneqq \sigma^\ell(x)$.
If $K_0=K_m=K$, equivalently $x\in C$, we set $\tau(x) \coloneqq x$.
It remains to consider the case $K_0=K$ and $K_m\neq K$.
Then there exists a unique $b \in \Z$ such that $x_i \in K$ for all $i \in (-\infty, b]$ and $x_j \in V \setminus K$
for all $j \in [b+1, \infty)$. 
Let $T(x) \coloneqq T(x_{b-N}x_{b-N+1} \cdots x_b) =  \bigl(t,(i_j)_{j=1}^t\bigr) \in {\mathcal A}(x_{b-N}x_{b-N+1} \cdots x_b)$ 
be the admissible pair selected in Remark~\ref{r:unique-admissible-pair}, and let $d$ be the positive
integer determined by $tdn=\ell$. The indices $i_j$ are relative to the path $x_{b-N}x_{b-N+1}\cdots x_b$.
Accordingly, set $a_j\coloneqq b-N+i_j$ and $r_j\coloneqq a_j-(j-1)dn$, for all $j=1,\ldots,t$.
Admissibility gives $a_j+dn \leq a_{j+1}$ for $1\leq j<t$, $a_t+dn\leq b$, and $x_{a_j}=x_{a_j+dn}$ for all $1\leq j\leq t$.
In particular, we have $r_1\leq r_2\leq\cdots\leq r_t$.
\par
We then set, for every $i \in \Z$,
\begin{equation}
\label{e_def-tau-squeeze}
\tau(x)_i \coloneqq
\begin{cases}
x_i & \text{ if } i \in (-\infty, r_1],\\
x_{i+dn} & \text{ if } i \in [r_1, r_2],\\
x_{i+2dn} & \text{ if } i \in [r_2, r_3],\\
\cdots\\
x_{i+(t-1)dn} & \text{ if } i \in [r_{t-1}, r_t],\\
x_{i+\ell} & \text{ if } i \in [r_t, +\infty).
\end{cases}
\end{equation}
When $t=1$, the middle cases are absent.
Moreover, the formulas agree at every common endpoint, since $x_{r_j+(j-1)dn}=x_{a_j}=x_{a_j+dn} =x_{r_j+jdn}$, 
for all $j=1,\ldots,t$.
\par
The interpretation of \eqref{e_def-tau-squeeze} goes as follows: $\tau$ squeezes each cycle 
$x_{a_j}x_{a_j+1}\cdots x_{a_j+dn}$ to its common initial and terminal vertex $x_{a_j}=x_{a_j+dn}$, 
for all $j=1,\ldots,t$. Equivalently, it deletes the paths $x_{a_j+1}\cdots x_{a_j+dn}$ and shifts the subsequent
coordinates to the left, leaving all coordinates less than $a_1=r_1$ unchanged.
Since the endpoints of each deleted closed subpath coincide, these deletions preserve admissibility. Thus $\tau(x)\in X$.
\par
The uniqueness of the selected pair $T(x)$ (cf.\ Remark~\ref{r:unique-admissible-pair}), together with \eqref{e_def-tau-squeeze}, 
shows that $\tau$ is well defined.\\

\noindent
{\bf Step 2.}
Let us show that the map $\tau \colon X \to X$ commutes with the shift map $\sigma$.
Let indeed $x \in X$. Observe that the chain of path-components of $\sigma(x)$ is the same as that of $x$.
In particular, $x$ and $\sigma(x)$ have the same initial path-component $K_0 \coloneqq K_0(x) = K_0(\sigma(x))$
and the same terminal path-component $K_m \coloneqq K_m(x) = K_m(\sigma(x))$, where $m \coloneqq m(x) = m(\sigma(x))$.
\par
We distinguish the same mutually exclusive cases as above.

\begin{enumerate}[{\rm {Case} 1:}]

\item $K_0 \neq K$.
Then $\tau(x) = \sigma^\ell(x)$ and $\tau(\sigma(x)) = \sigma^\ell(\sigma(x)) = \sigma^{\ell + 1}(x)$.
We thus have $(\sigma \circ \tau)(x) = \sigma(\sigma^\ell(x)) = \sigma^{\ell + 1}(x) =  \tau(\sigma(x)) = (\tau \circ \sigma)(x)$.

\item $K_m = K$, that is, $x \in C$. Then $\sigma(x) \in C$ and therefore $\tau(\sigma(x)) = \sigma(x) = \sigma(\tau(x))$.

\item $K_0 = K \neq K_m$.
Let $b=b(x)\in\mathbb Z$ be the unique coordinate such that $x_b\in K$ and $x_{b+1}\notin K$.
Then $b'\coloneqq b(\sigma(x))=b-1$. Moreover, $(\sigma(x))_{b'-N}\cdots(\sigma(x))_{b'} = x_{b-N}\cdots x_b$.
Since $T$ depends only on this finite path, we have $T(\sigma(x))=T(x)=\bigl(t,(i_j)_{j=1}^{t}\bigr)$.
In particular, the values of $t$ and $d$ are unchanged.
The corresponding absolute indices and cut coordinates satisfy $a'_j=b'-N+i_j=a_j-1$ and $r'_j=a'_j-(j-1)dn=r_j-1$, for
all $j=1,\ldots,t$. Applying \eqref{e_def-tau-squeeze} to $\sigma(x)$ therefore
amounts to replacing each $r_j$ by $r_j-1$ and each $x_k$ by $x_{k+1}$.
Thus, for every $i\in\Z$,
\[
\tau(\sigma(x))_i = \tau(x)_{i+1} = \sigma(\tau(x))_i.
\]
Consequently, $\tau(\sigma(x))=\sigma(\tau(x))$.
\end{enumerate}

This shows that $(\tau \circ \sigma)(x) = (\sigma \circ \tau)(x)$ for all $x \in X$.
Therefore, we have $\tau \circ \sigma = \sigma \circ \tau$.\\

\begin{enumerate}[{\rm {Case} 1:}]
\item $x_0 \in V \setminus K$. Then $y_0 = \tau(x)_0 = \sigma^\ell(x)_0 = x_\ell$ only depends on $x_0$ and $x_\ell$.
\item $x_\ell \in K$. Then $y_0 = x_0$ only depends on $x_\ell$ and $x_0$.
\item $x_0 \in K$ and $x_\ell \in V \setminus K$. Let $0 \leq b < \ell$ such that $x_b \in K$ and $x_{b+1} \in V \setminus K$.
Recalling that the pair $T(x) = T(x_{b-N}x_{b-N+1} \cdots x_b)$ only depends on $x_{[b-N,b]}$,
the value $y_0$ equals one of $x_0, x_{dn}, x_{2dn}, \ldots, x_{\ell}$.
In particular, $y_0$ only depends on $x_{[b-N, \ell]}$.
\end{enumerate}
This completes the proof that $\tau(x)_0$ only depends on the word $x_{[-N,\ell]}$. \\

\noindent
\textbf{Step 3.}
Let $x\in X$ and $y \coloneqq \tau(x)$.
We claim that $y_0$ depends only on the word $x_{[-N,N]} \in V^{2N + 1}$.
\par
Recall that $\ell\leq N$. Indeed, the admissibility conditions give
$(t-1)dn\leq i_t-i_1$ and $i_t+dn\leq N$, so that $\ell=tdn \leq i_t+dn-i_1\leq N$.
\par
We distinguish the following three cases.

\begin{enumerate}[{\rm {Case} 1:}]
\item
$x_0\notin K$. If the initial path-component of $x$ is different from $K$,
then $\tau(x)=\sigma^\ell(x)$ by definition.
Otherwise, $x$ has already left $K$, so $b < 0$. Since $a_t+dn\leq b$, we have
$r_t=a_t-(t-1)dn \leq b-tdn = b-\ell<0$. Thus, in either situation, $y_0=x_\ell$.

\item
$x_N\in K$. As $K$ is minimal, the initial path-component of $x$ is $K$,
and $x_i\in K$ for all $i\leq N$.
If $x\in C$, then $y_0=x_0$ by definition.
Otherwise, the exit coordinate satisfies $b\geq N$, and therefore
$r_1=a_1=b-N+i_1\geq 0$. Formula~\eqref{e_def-tau-squeeze} again gives $y_0=x_0$.

\item
$x_0\in K$ and $x_N\notin K$.
The exit coordinate $b$ satisfies $0\leq b<N$ and is determined by the word $x_{[0,N]}$.
The pair $T(x)=T(x_{b-N}\cdots x_b)$ is determined by $x_{[b-N,b]}$.
Consequently, $t$, $d$, and the cut coordinates $r_j=b-N+i_j-(j-1)dn$, $j=1,\ldots,t$,
are determined by $x_{[-N,N]}$.
Formula~\eqref{e_def-tau-squeeze} then determines $y_0$ as one of the coordinates
$x_0, x_{dn},x_{2dn},\ldots,x_{tdn}=x_\ell$.
All these coordinates lie in $[0,N]$, since $\ell\leq N$.
\end{enumerate}

Thus $\tau(x)_0$ depends only on $x_{[-N,N]}$.\\

\noindent
{\bf Step 4.}
It follows from the previous step that there is a map $\mu \colon V^{N+\ell + 1} \to V$ such that
$\tau(x)_0 = \mu(x_{[-N,\ell]})$ for all $x \in X$.
As $\tau$ commutes with the shift (Step 2), we deduce from Proposition~\ref{p:char-ca} that $\tau \colon X \to X$
is an endomorphism of $X$.\\

\noindent
{\bf Step 5.}
Let us show that the map $\tau \colon X \to X$ is surjective.
Let $x \in X$. We distinguish the following three cases as above.

\begin{enumerate}[{\rm {Case} 1:}]
\item $K_0 \neq K$. Then $\tau(x) = \sigma^\ell(x)$ so that $x = \sigma^{-\ell}(\tau(x)) = \tau(\sigma^{-\ell}(x))$.
\item $K_m = K$, that is, $x \in C$. Then $x = \tau(x)$.
\item $K_0 = K \neq K_m$. As usual, let $b \in \Z$ such that $x_b \in K$ and $x_{b+1} \notin K$.
Consider the configuration $z$ defined by setting, for all $i \in \Z$,
\[
z_i \coloneqq \begin{cases} x_i & \mbox{ if } i \leq b\\
v_j & \mbox{ if } i = b + j \bmod n, \ b \leq i \leq b + \ell\\
x_{i-\ell} & \mbox{ if } i \geq b + \ell,
\end{cases}
\]
where $\pi = v_0v_1 \cdots v_n$ is a cycle of length $n$ based at $v_0 = x_b$ (cf.~Lemma~\ref{l:2-cycles-de-long-n}).
It is clear that $z \in X$, $T(z) = (1,b)$ (cf.~end of Remark~\ref{r:unique-admissible-pair})
and $\tau(z) = x$.
\end{enumerate}
This shows that $\tau$ is surjective.

\section{Proofs of the main results}
\label{sec:proofs}

In this section, we give the proofs of the theorems stated in the introduction.

\subsection{An auxiliary result}
We start by establishing a  result that will be used in the proof of each of the theorems stated in  the introduction.
\par 
Let $A$ be a finite set and let $X \subset A^{\Z}$ be a subshift of finite type.  
We denote by $\CC$ the poset of irreducible components of $X$.
Given $C,C' \in \CC$, we denote by $X_{C,C'} \subset X$ the subset consisting of all configurations $x \in X$ such that
$C_\alpha(x) = C$ and $C_\omega(x) = C'$.
Observe that $X_{C,C'}$ is a shift-invariant (but in general not closed) subset of $X$ and that one has $X_{C,C'} \not= \varnothing$ if and only if $C \preceq C'$.
We denote by $\OO_{C,C'} \coloneqq \{\OO(x) : x \in X_{C,C'} \}$ the set of shift-orbits in $X_{C,C'}$. 

\begin{lemma}
\label{l:W-C-C'}
Let $A$ be a finite set and let $X \subset A^{\Z}$ be a subshift of finite type.  
Let $C$ and $C'$ be finite irreducible components of $X$. 
Suppose that $C$ and $C'$ are  adjacent with $C \prec C'$.
Then the set $\OO_{C,C'}$ is finite.
\end{lemma}

\begin{proof}
We may assume that $X$ is the vertex subshift associated with a finite directed graph $G = (V,E)$.
Let $K$ (resp.~$K'$) denote the path-component of $G$ associated with $C$ (resp.~$C'$).
Thus, for $x \in X_{C,C'}$ the chain of path-components is $\Gamma(x) = \{K,K'\}$ and it follows from \eqref{e:normal-form-irred}
that there exist integers $b_0 < a_1$ such that $x_i \in K$ for $i \leq b_0$,
$x_i \in T$ for $b_0 < i < a_1$, and $x_i \in K'$ for $a_1 \leq i$.
Consider the word $w(x) \coloneqq u t_1\cdots t_kv \in V^{k + 2}$, where $u \coloneqq x_{b_0} \in K$,
$t_j  \coloneqq x_{b_0+j} \in T$ for $j =1,2, \ldots, k \coloneqq a_1 - b_0 - 1$, 
and $v \coloneqq x_{a_1} \in K'$.
Recall (cf.\ Proposition~\ref{p:normal-form-irred}.(i)) that the vertices $t_1,\dots,t_k$ are  all distinct.
Moreover, we clearly have $w(\sigma(x)) = w(x)$ for all $x \in X_{C,C'}$.
As the orbit of $x \in X_{C,C'}$ is entirely determined by $w(x)$ (cf.~Corollary~\ref{c:int-x-finite-irred-compo}),
it follows that $|\OO_{C,C'}| = |\{w(x): x \in X_{C,C'}\}|$. 
Since the sets $K$, $K'$, and $T$ are finite,
we deduce that the set $\OO_{C,C'}$ is itself finite. 
\end{proof}

\subsection{Proof of Theorem~\ref{t:char-surj-sft}}
\label{sec:char-surj-sft}
We shall prove the equivalence of the four conditions stated in Theorem~\ref{t:char-surj-sft} by successively showing
 (a) $\iff$ (c), (c) $\iff$ (b),  and (c) $\iff$ (d).

\begin{proof}[Proof of (a) $\iff$ (c) in Theorem~\ref{t:char-surj-sft}]
We can assume that $X$ is the vertex subshift associated with a finite directed graph $G = (V,E)$.
\par
Suppose that $X$ is surjunctive.
We first show that every irreducible component of $X$ is extremal.
Suppose by contradiction that there exists $C \in \CC$ which is not extremal.
Then there exist irreducible components  $C',C'' \in \CC$ such that $C' \prec C \prec C''$.
Since the poset $\CC$ is finite, we can take $C'$ and $C''$ adjacent to $C$.
Let  $K_0$ (resp.~$K_1$, resp.~$K_2$) denote the path-component of $G$ associated with $C'$ (resp.~$C$, resp.~$C''$).
Observe that $K_0 \prec K_1 \prec K_2$ with $K_0$ and $K_2$ both adjacent to $K_1$.
Denote by $Y$ the nonempty set consisting of all configurations $x \in X$ whose chain of path components is
$\Gamma(x) = \{K_0,K_1,K_2\}$.
Consider the endomorphism $\tau \coloneqq  \tau_C \colon X \to X$ defined in Section~\ref{sec:construction-tau-C}.
As $\tau$ is injective and the subshift $X$ is surjunctive, $\tau$ is surjective.
It follows from  Remark~\ref{r:same-chain} that $\tau^{-1}(Y) \subset  Y$.
As $\tau$ is surjective, we deduce that $\tau(Y) = Y$.
For every $x \in Y$, denote by $b_0(x) < a_1(x) \leq b_1(x) < a_2(x)$ the integers satisfying \eqref{e:normal-form-irred}
and set $m(x) \coloneqq b_1(x) - a_1(x)$.
It follows from \eqref{e:tau-x} (here $h = 0$) that $a_1(\tau(x)) = a_1(x)$ and $b_1(\tau(x)) = b_1(x) + \ell$.
Thus, we have
\begin{equation}
\label{e:m(x)}
m(\tau(x)) = b_1(\tau(x)) - a_1(\tau(x)) = b_1(x) + \ell - a_1(x) = m(x) + \ell,
\end{equation}
for all $x \in Y$. For a subset $Z \subset Y$ set $m(Z) \coloneqq \min_{x \in Z} m(x)$.
Then, keeping in mind that $Y \neq \varnothing$, from \eqref{e:m(x)} we get $m(\tau(Y)) = m(Y) + \ell > m(Y)$.
This yields $\tau(Y) \subsetneqq Y$, a contradiction.
We deduce that all irreducible components of $X$ are extremal.
\par
Let us now show that every non-isolated irreducible component of $X$ is finite.
Let $C \in \CC$ be a non-isolated irreducible component. 
We know that   $C$ is extremal.
Suppose that $C$ is maximal (the proof in the case when $C$ is minimal is similar).
Let then $C_0\in \CC$ be an irreducible component adjacent to $C$, so that $C_0 \prec C$.
Set $K_0 \coloneqq K_{C_0} \in \KK$ and $K_1 \coloneqq K_{C} \in \KK$, and observe that $K_0 \prec K_1$ with $K_0$ and $K_1$ adjacent.
Set now $Y \coloneqq X_{C_0,C}$. 
Thus, $Y$ is the nonempty set consisting of all configurations $x \in X$ whose chain of path components is $\Gamma(x) = \{K_0,K_1\}$.
Consider the endomorphism $\tau \coloneqq  \tau_C \colon X \to X$ defined in Section~\ref{sec:construction-tau-C}.
As $\tau$ is injective and the subshift $X$ is surjunctive, $\tau$ is surjective.
It follows from Remark~\ref{r:same-chain} that $\tau^{-1}(Y) \subset Y$.
As $\tau$ is surjective, we deduce that $\tau(Y) = Y$. Therefore $\tau^k(Y) = Y$ for all $k \in \N$.
\par
Let $y \in Y$. Denote by $b_0 < a_1$ (here $b_0 = b_0(y)$ and $a_1 = a_1(y)$) 
the integers satisfying \eqref{e:normal-form-irred} with respect to  the configuration $y$  and set $v \coloneqq y_{a_1} \in K_1$.
Let then $\pi_v = c'_0  \cdots c'_\ell$ be the preferred cycle based at $v$ of length $\ell$ as in Step 1.
Recall that it is entirely contained in $K_1$.
For every $k \in \N$ there exists $x = x(k) \in Y$ such that $\tau^k(x) = y$.
By applying \eqref{e:tau-x} we get $y_i = \tau^k(x)_i = c'_j$ for all $a_1 \leq i \leq a_1 + k \ell$ with $i \equiv a_1 + j \mod \ell$.
As $k \in \N$ was arbitrary, we deduce that in fact
\begin{equation}
\label{e:x-dans-Y-periodique-a-droite}
\mbox{$y_i = c'_j$ for all $i \geq a_1$ with $i \equiv a_1 + j \mod \ell$.}
\end{equation}
Let $z' = z'(y) \in C$ denote the periodic configuration defined by setting $z'_i \coloneqq c'_j$ for all $i \in \Z$ with
$i \equiv a_1 + j \mod \ell$. Keeping in mind \eqref{e:x-dans-Y-periodique-a-droite}, we have that $\omega(y) = \omega(z')$.
We thus deduce that
\begin{equation}
\label{omega-limit-y}
\vert \omega(y) \vert = \vert \omega(z') \vert = \vert \OO(z') \vert = \ell \mbox{ for all } y \in Y.
\end{equation}
Suppose $C$ is infinite. Then by Proposition~\ref{p:finite-component-2}, $C$ contains infinitely many periodic points
and we can find $z \in \Per(C)$ whose minimal period $p$ satisfies $p > \ell$.
Consider the configuration $y^z \in Y$ defined as follows.
Since $K_1$ is a path-component, there exist an integer $k \geq 0$ and a  path $v_0  \cdots  v_k$ within $K_1$
such that $v_0 = x_{a_1}$ and $v_k = z_0$. We then set, for every $i \in \Z$,
\begin{equation}
\label{e:z-dans-Y}
y^z_i \coloneqq \begin{cases} x_i & \mbox{ if } i < a_1\\
v_{i - a_1} & \mbox{ if } a_1 \leq i \leq a_1+k \\
z_{i - k - a_1} & \mbox{ if } a_1 + k \leq i.
\end{cases}
\end{equation}
It is clear that $\Gamma(y^z) = \{K_0,K_1\}$, so that $y^z \in Y$.
As above, the $\omega$-limit set $\omega(y^z)$ of $y^z$ equals $\omega(z) = \OO(z)$.
This yields $\vert \omega(y^z) \vert = \vert \omega(z) \vert = \vert \OO(z) \vert = p > \ell$, contradicting \eqref{omega-limit-y}.
We deduce that $C$ is finite.
Thus, that every non-isolated component in $X$ is finite.
This completes the proof of (a) $\implies$ (c).
\par
Conversely, suppose that $X$ satisfies (c). Let us show that $X$ is surjunctive.
Let $\tau \colon X \to X$ be an injective endomorphism.
Note that, for every integer $m \geq 1$, the map $\tau^m$ is an injective endomorphism as well and
that $\tau$ is surjective if $\tau^m$ is surjective. 
On the other hand, it follows from  Proposition~\ref{p:inj-tau-implies-inj-rho} that the map $C \mapsto \tau(C)$, $C \in \CC$,
is a permutation of $\CC$. As a consequence, up to possibly replacing $\tau$ by $\tau^m$ for $m \coloneqq |\CC| !$,
we may assume that $\tau(C) = C$ for every $C \in \CC$.
\par
Let $y \in X$. Suppose first that $y \in \NW(X)$, that is, $y \in C$ for some $C \in \CC$.
As $\tau(C) = C$, there exists $x \in C$ such that $\tau(x) = y$.
\par
Suppose now that $y \in X \setminus \NW(X)$.
Then the irreducible components $C \coloneqq C_\alpha(y)$ and $C' \coloneqq C_\omega(y)$ satisfy $C \prec C'$ by
Proposition~\ref{p:irred-comp}. Thus $C$ and $C'$ are non-isolated. Therefore, $C$ and $C'$  are adjacent and both finite by our assumptions.
Consider   the nonempty shift-invariant subset $X_{C,C'} \subset X$ consisting of all $x \in X$ such that $C_\alpha(x) = C$
and $C_\omega(x) = C'$.
Observe that $\tau$ induces, by restriction, an injective map from $X_{C,C'}$ into itself
(cf.~Proposition~\ref{p:ca-sft-irred-compo}.(ii)).
As $\tau$ commutes with the shift, it induces a map $\eta \colon \OO_{C,C'} \to \OO_{C,C'}$.
We claim that $\eta$ is injective.
Indeed, if $z,z' \in X_{C,C'}$ have respective shift-orbits $O = \OO(z)$ and $O' = \OO(z')$ satisfying $\eta(O) = \eta(O')$,
then there exists $n \in \Z$ such that $\tau(z) = \sigma^n(\tau(z'))$.
This implies $\tau(z) = \tau(\sigma^n(z'))$ and hence $z = \sigma^n(z')$, since $\tau$ is injective.
We deduce that $O = O'$. This proves our claim.
As $\OO_{C,C'}$ is finite by Lemma~\ref{l:W-C-C'}, $\eta$ is surjective. Thus, there exists $x \in X_{C,C'}$ such that $\eta(\OO(x)) = \OO(y)$.
We deduce that there exists $n \in \N$ such that $y = \tau(\sigma^n(x))$.
\par
Thus $y \in \tau(X)$, showing that $\tau$ is surjective.
We deduce that $X$ is surjunctive.
This shows  (c) $\implies$ (a). 
\end{proof}

\begin{proof}[Proof of (c) $\implies$ (b) in Theorem~\ref{t:char-surj-sft}]
Suppose (c).
Let $\tau \colon X \to X$ be a surjective endomorphism and let us show that $\tau$ is
pre-injective.
Note that for every integer $m \geq 1$, the map $\tau^m$ is a surjective endomorphism as well and that $\tau$ is pre-injective if  $\tau^m$ is pre-injective (cf.~\cite[Corollary~2.6]{csc-goe-smale}).
It follows from  Proposition~\ref{p:surj-tau-implies-surj-rho} that the map
$C \mapsto \tau(C)$, $C \in \CC$,  is a permutation of $\CC$.
As a consequence, up to possibly replacing $\tau$ by $\tau^m$ for $m \coloneqq |\CC| !$,
we may assume that $\tau(C) = C$ for every $C \in \CC$.
\par
Let $x,y \in X$ be homoclinic configurations such that $\tau(x) = \tau(y)$.
By using Proposition~\ref{p:limit-sets-homoclinic} and Proposition~\ref{p:limi-set-vertex-ss},
we deduce that $C_\alpha(x) = C_\alpha(y)$ and $C_\omega(x) = C_\omega(y)$.
Let us set $C \coloneqq C_\alpha(x) = C_\alpha(y)$ and $C' \coloneqq C_\omega(x) = C_\omega(y)$.
\par
If $C=C'$, then $x,y \in C$ (cf.~Proposition~\ref{p:limi-set-vertex-ss}).
As $C$ is an irreducible subshift of finite type and $\tau$ induces, by restriction,
a surjective endomorphism of $C$, the Garden of Eden theorem for
irreducible subshifts of finite type (cf.~\cite[Corollary~2.19]{fiorenzi-sofic}) ensures that $x=y$.
\par
Otherwise, the irreducible components $C$ and $C'$ satisfy $C \prec C'$ by Proposition~\ref{p:irred-comp}.
Thus $C$ and $C'$ are non-isolated, and therefore  adjacent and both finite by (c).
Consider the shift-invariant subset $X_{C,C'} \subset X$   consisting of all $z \in X$ such that $C_\alpha(z) = C$ and $C_\omega(z) = C'$.
Observe that $X_{C,C'} \neq \varnothing$ since $x,y \in X_{C,C'}$.
It follows from Proposition~\ref{p:ca-sft-irred-compo}.(ii) that $\tau^{-1}(X_{C,C'}) = X_{C,C'}$.
As $\tau$ is surjective, we deduce that $\tau(X_{C,C'}) = X_{C,C'}$.
Consider   the set $\OO_{C,C'}$ of shift-orbits in $X_{C,C'}$.
Since $\tau$ commutes with the shift, it induces a surjective map
$\eta \colon \OO_{C,C'} \to \OO_{C,C'}$.
As $\OO_{C,C'}$ is finite by Lemma~\ref{l:W-C-C'}, $\eta$ is injective.
Since $\tau(x) = \tau(y)$, we have $\eta(\OO(x)) = \OO(\tau(x)) = \OO(\tau(y)) = \eta(\OO(y))$
so that $\OO(x) = \OO(y)$. We deduce that there exists $n \in \Z$ such that $x = \sigma^n(y)$.
This implies $\tau(y) = \tau(x) = \tau(\sigma^n(y)) = \sigma^n(\tau(y))$ and hence $n = 0$,
since the configurations in $X_{C,C'}$ are wandering (cf.~Proposition~\ref{p:irred-comp}.(vii)) and therefore non-periodic.
Consequently, $x=y$. This shows that $\tau$ is pre-injective. 
\end{proof}

\begin{lemma}
\label{l:finite-compo-not-moore}
Let $A$ be a finite set and let $X \subset A^{\Z}$ be a subshift of finite type.
Suppose that one of the following conditions is satisfied:
\begin{enumerate}[\rm (C1)]
\item
 $X$ admits a minimal irreducible component $C$ which is finite
and two irreducible components $C'$ and $C''$  such that $C \prec C' \prec C''$; 
\item
 $X$ admits a maximal irreducible component $C$ which is finite
and two irreducible components $C'$ and $C''$  such that $C'' \prec C' \prec C$.
\end{enumerate} 
Then $X$  is not Moore.
\end{lemma}

\begin{proof}
We can assume that $X$ is the vertex subshift associated with a finite directed graph $G = (V,E)$.
We give the proof of the lemma assuming (C1) (the proof when (C2) is satisfied is similar).
\par
Let $K$ (resp.~$ C'$, resp.~$ C''$) denote the path-component of $G$ associated with $C$ (resp.~$ C'$, resp.~$ C''$).
By our hypotheses,
there exists a configuration $y \in X$ in which appears 
a path in $G$  of the form
\[
v_0 t_1\cdots t_q v_0'\cdots v_r' t_1'\cdots t_s' v_0''
\]
with $v_0 \in K$, $v_0',\dots ,v_r' \in K'$, $v_0'' \in K''$, $t_1,\dots,t_q,t_1',\dots,t_s' \in V \setminus (K \cup K' \cup K'')$,  and $q,r,s \geq 0$. 
After suppressing cycles in $K'$ that possibly appear in $y$, we can assume that  the vertices $v_0',\dots,v_r'$ are all distinct.
\par
Choose an integer  $n \geq 1$ such that there exist  a cycle
$v_0v_1\cdots v_{n-1}v_0$ of length $n$ in $K$ and a cycle  $v_0'u_1'\cdots u_{n-1}'v_0'$ of length $n$ in $K'$ (cf.~Step~1 in Section~\ref{sec:construction-tau-C}).
\par
Consider the paths $\pi_0$ and $\pi_1$ in $G$  of length $\ell \coloneqq n + q + 1$  going from $v_0$ to $v_0'$ respectively defined by
\[
\pi_0 \coloneqq v_0 t_1 \cdots t_q v_0' u_1' \cdots u_{n - 1}' v_0'  \text{ and } 
\pi_1 \coloneqq v_0 v_1 \cdots v_{n-1} v_0 t_1 \cdots t_q v_0'.
\]
Define a map $\tau \colon X \to X$ as follows.
Let $x \in X$. 
If the word $\pi_0$ appears in $x$, i.e., there exists $i \in \Z$ such that
$x_{[i,i+ \ell ]} = \pi_0$
(note that such an $i $ is then unique), we define $\tau(x)$ as being the configuration in $X$ obtained from $x$ by replacing 
$\pi_0$   by  $\pi_1$ on $[i,i+ \ell]$.
Otherwise, we set $\tau(x) \coloneqq x$.
\par
The map  $\tau \colon X \to X$ is  an endomorphism of $X$
since it clearly satisfies Condition~(b) in Proposition~\ref{p:char-ca} for $m \coloneqq \ell$.
\par
We claim that $\tau$ is surjective but not pre-injective.
\par
To prove surjectivity, let $z \in X$.
 If the word $\pi_0$ does not appear in $z$, then $z = \tau(z)$.
 Otherwise, consider the unique  $i \in \Z$ such that $z_{[i,i+ \ell]} = \pi_0$.
Since $C$ is minimal and  finite, we then have  $z_{[i - n,i+\ell -n]} = \pi_1$ by Corollary~\ref{c:int-x-finite-irred-compo}.
Consider  the configuration $z' \in X$ obtained from $z$ by replacing $\pi_1$ by $\pi_0$ on $[i - n,i+\ell -n]$.
 We then have $z = \tau(z')$.
This shows that $\tau$ is surjective.
\par
To prove that $\tau$ is not pre-injective, consider the configuration $y' \in X$ obtained from $y$ by replacing the path $\pi_1$ on $[i - n,i+ \ell -n]$ by the path $\pi_0$.
Then $\tau(y') = y = \tau(y)$.
As $y$ and $y'$ are disctinct homoclinic configurations, we deduce that $\tau$ is not pre-injective.
\par
Since $\tau$ is a surjective endomorphism of $X$ which is not pre-injective, we conclude that $X$ is not Moore.
  \end{proof}

\begin{lemma}
\label{l:infinite-compo-not-moore-1}
Let $A$ be a finite set and let $X \subset A^{\Z}$ be a subshift of finite type.
Suppose that $X$ contains two irreducible components $C$ and $C'$ with $C \prec C'$
such that one of these components is infinite and the other is finite and extremal.
Then $X$ is not Moore. 
\end{lemma}

\begin{proof}
We can assume that $X$ is the vertex subshift associated with a finite directed graph $G = (V,E)$.
We give the proof assuming that $C$ is infinite and $C'$ is finite and maximal (the proof when $C'$ is infinite and $C$ is finite and minimal is similar).
\par
Let $K$ (resp.~$K'$) denote the path-component of $G$ associated with $C$ (resp.~$C'$).
Since $C \prec C'$, there exists a configuration $\overline{x} \in X$ whose initial and terminal path-components are  
$K$ and $K'$, respectively.
Then, up to shifting $\overline{x}$ if necessary, there exist $v_0 \in K$, $u_0\in K'$, $t_1,\dots,t_q \in V \setminus (K \cup K')$,
where $q \geq 0$, such that $\overline{x}_{[0,q+1]} = v_0 t_1\cdots t_q u_0$.
\par
As $C$ is infinite, it follows from Lemma~\ref{l:2-cycles-de-long-n} that there exists an integer $n \geq 1$
and two distinct cycles $v_0v_1 \cdots v_{n-1}v_0$ and $v_0v'_1 \cdots v'_{n-1}v_0$ in $K$ based at $v_0$ of length $n$.
In fact, arguing as in Step 1 in Section~\ref{sec:construction-tau-C}, up to replacing $n$ by a suitably chosen multiple,
we can also find a cycle $u_0u_1 \cdots u_{n-1}u_0$ in $K'$ based at $u_0$ of the same length $n$.
\par
Consider the paths $\pi_0$, $\pi'_0$, and $\pi_1$ in $G$, with length $\ell \coloneqq n + q + 1$ and going from $v_0$ to $u_0$,
respectively defined by
\[
\begin{split}
\pi_0 & \coloneqq v_0v_1 \cdots v_{n-1}v_0t_1\cdots t_q u_0\\
\pi'_0 & \coloneqq v_0v'_1 \cdots v'_{n-1}v_0t_1\cdots t_q u_0\\
\pi_1 & \coloneqq v_0t_1\cdots t_q u_0u_1 \cdots u_{n-1}u_0.
\end{split}
\]
Define a map $\tau \colon X \to X$ as follows.
Let $x \in X$. If there exists $i \in \Z$ such that
$x_{[i,i+ \ell]}$ equals either $\pi_0$ or $\pi'_0$
(note that such an $i$ is then unique), we define $\tau(x)$ as being the configuration in $X$ obtained from $x$ by replacing
$\pi_0$ or $\pi'_0$ by $\pi_1$ on $[i,i+ \ell]$.
If there is no such $i$, we set $\tau(x) \coloneqq x$.
\par
The map  $\tau \colon X \to X$ is an endomorphism of $X$
since it clearly satisfies Condition~(b) in Proposition~\ref{p:char-ca} with $m \coloneqq \ell$.
\par
We claim that $\tau$ is surjective but not pre-injective.
Let $z \in X$. If neither $\pi_0$ nor $\pi'_0$ appears in $z$, then $z = \tau(z)$ is in the image of $\tau$.
Otherwise, consider the unique $i \in \Z$ such that $z_{[i,i+ \ell]} = \pi_0$ (resp.~$z_{[i,i+ \ell]} = \pi'_0$).
Since $C'$ is finite and extremal, we have $z_{[i + n,i+\ell  + n]} = \pi_1$ by Corollary~\ref{c:int-x-finite-irred-compo}.
Consider now the configuration $w \in X$ obtained from $z$ by replacing $\pi_1$ by $\pi_0$ on
$[i + n,i+\ell +n]$. We then have $z = \tau(w)$. This shows that $\tau$ is surjective.
\par
In order to show that $\tau$ is not pre-injective, we observe that the configuration $\overline{x} \in X$ satisfies
that $\overline{x}_{[0,\ell]} = \pi_1$. This follows from the fact that $C'$ is finite and maximal,
by using again Corollary~\ref{c:int-x-finite-irred-compo}.
Consider the configurations $y,y' \in X$ obtained from $\overline{x}$ by replacing $\pi_1$ by $\pi_0$ and $\pi'_0$ on $[0,\ell]$, respectively. We have $y \not= y'$ since $\pi_0 \not= \pi'_0$.
Since $y$ and $y'$ are homoclinic and satisfy $\tau(y) = \overline{x} = \tau(y')$,
we deduce that $\tau$ is not pre-injective.
\par
Since $\tau$ is an endomorphism of $X$ which is surjective but not pre-injective, we conclude that $X$ is not Moore.
\end{proof}

\begin{lemma}
\label{l:infinite-compo-not-moore-2}
Let $A$ be a finite set and let $X \subset A^{\Z}$ be a subshift of finite type.
Suppose that $X$ contains an infinite irreducible component which is extremal but not isolated.
Then $X$ is not Moore.
\end{lemma}
\begin{proof}
Let $C$ denote an infinite irreducible component of $X$ which is extremal but not isolated.
Suppose that $C$ is minimal (the case when $C$ is maximal is similar).
Since $C$ is not isolated, we can find an irreducible component $C'$ of $X$ such that $C \prec C'$.
Consider the endomorphism $\tau = \tau^C$ defined in Section~\ref{sec:construction-tauC}.
In Step 5 in that section, when proving surjectivity of $\tau$, given a configuration $x \in X_{C,C'}$ we
defined a configuration $z \in X$ such that $\tau(z) = x$.
Keeping the same notation therein, we define another configuration $z'$ by setting, for all $i \in \Z$,
\[
z'_i \coloneqq \begin{cases} x_i & \mbox{ if } i \leq b\\
v'_j & \mbox{ if } i = b + j \bmod n, \ b \leq i \leq b + \ell\\
x_{i-\ell} & \mbox{ if } i \geq b + \ell,
\end{cases}
\]
where now $\pi' = v'_0v'_1 \cdots v'_n$ is a cycle of length $n$ based at $v'_0 = x_b$ and distinct from $\pi$
(recall that $C$ is infinite so that Lemma~\ref{l:2-cycles-de-long-n} applies).
It is clear that $z' \in X$, it is distinct from $z$ (since $\pi' \neq \pi$), it is homoclinic to $z$
($z$ and $z'$ agree on $\Z \setminus [b, b+\ell]$), $T(z') = (1,b) = T(z)$ and $\tau(z') = x = \tau(z)$.
We deduce that $\tau$ is not pre-injective.
\par
Since $\tau$ is an endomorphism of $X$ which is surjective but not pre-injective, we conclude that $X$ is not Moore.
\end{proof}

\begin{proof}[Proof of (b) $\implies$ (c) in Theorem~\ref{t:char-surj-sft}]
Suppose that $X$ contains infinite irreducible components which are not isolated.
If there exists one of these that is extremal, we deduce from Lemma~\ref{l:infinite-compo-not-moore-2} that $X$ is not Moore.
Otherwise, we can find two irreducible components $C, C'$ with $C$ infinite and $C'$ finite and minimal, such that $C' \prec C$.
It then follows from Lemma~\ref{l:infinite-compo-not-moore-1} that $X$ is not Moore either.
Suppose now that all infinite irreducible components of $X$ are isolated and that there is an irreducible component
which is not extremal. It then follows from Lemma~\ref{l:finite-compo-not-moore} that $X$ is not Moore.
This shows (b) $\implies$ (c).
\end{proof}

\begin{proof}[Proof of (c) $\iff$ (d) in Theorem~\ref{t:char-surj-sft}]
Let $P$ denote the perfect kernel of $X$.
Let $\CC$ denote the poset of irreducible components of $X$.
Denote by $\CC_\infty$ (resp.\ $\CC_{\infty,{\rm i}}$, resp.\ $\CC_{\rm f}$, resp.\ $\CC_{\rm f,i}$)
the set of irreducible components of $X$ that are infinite (resp.\ infinite and isolated, resp.\ finite, resp.\ finite and isolated).
\par
Suppose (c).
Thus, $\CC_\infty = \CC_{\infty, {\rm i}}$ and
$X \setminus (\bigcup_{C \in \CC_\infty} C) = \bigcup_{C_1, C_2 \in \CC_{\rm f} } X_{C_1,C_2}$.
We then deduce from Proposition~\ref{p:isolated-configuration}  that
the derived set of $X$  is the non-wandering subshift of finite type $X' = \bigcup_{C \in \CC_\infty \cup (\CC_{\rm f} \setminus \CC_{\rm f,i})} C$.
It follows that
$X^{(n)} = \bigcup_{C \in \CC_\infty} C$ for every $n \geq 2$.
Therefore, we have $P = \bigcup_{C \in \CC_\infty} C$ and $\CBrank(X) \leq 2$.
This shows that (c) implies (d).
\par
Conversely, suppose (d).
Let $C$ be an infinite irreducible component of $X$.
It follows from Proposition~\ref{p:isolated-configuration} that $C' = C$. Consequently, we have $C = C' \subset X'$ and,
recursively, $C \subset X^{(n)}$ for all $n \in \N$. We deduce that $C \subset P$.
Since, $P = \bigcup_{C \in \CC_{\infty,{\rm i}}} C$, we deduce that $C$ is isolated.
\par
It follows from Proposition~\ref{p:isolated-component} that the subset $Y \coloneqq X \setminus P$ is closed.
We thus have $X^{(n)} = P \sqcup Y^{(n)}$ for all $n \in \N$.
In particular, $P \sqcup Y^{(\CBrank(X))} = P^{(\CBrank(X))} \sqcup Y^{(\CBrank(X))} = X^{(\CBrank(X))} = P$,
which yields $Y^{(\CBrank(X))} = \varnothing$.
Since $\CBrank(X) \leq 2$, we thus have $Y'' = \varnothing$. As $Y$ is countable and compact, this implies that $Y'$ is finite.
As $Y'$ is $\sigma$-invariant, every configuration $y$ in $Y'$ is periodic, so that $C_\alpha(y) = C_\omega(y)$.
On the other hand, it follows from Proposition~\ref{p:isolated-configuration} that the set $Y \setminus Y'$ of isolated configurations in
$Y$ consists of all $y \in Y$ such that $C_\alpha(y)$ and $C_\omega(y)$ are both finite, with $C_\alpha(y)$ minimal and $C_\omega(y)$ maximal.
Therefore $Y$ is the disjoint union of the set of finite extremal components of $X$, and
of the subset consisting of all $y \in X$ such that $C_\alpha(y) \prec C_\omega(y)$ with $C_\alpha(y)$ and $C_\omega(y)$ both finite,
$C_\alpha(x)$ minimal, and $C_\omega(x)$ maximal.
Since isolated (finite or infinite) components are extremal, we deduce that every irreducible component of $X$ is extremal.
This shows that (d) implies (c).
\end{proof}

\begin{proof}[Proof of (d) $\iff$ (e) in Theorem~\ref{t:char-surj-sft}]
We keep the notation from the proof of (c) $\iff$ (d).
\par
Suppose (d).
It follows from Proposition~\ref{p:nw-vertex-ss}.(iii) that the perfect kernel $P = \bigcup_{C \in \CC_{\infty}}C$ of $X$
is non-wandering. Moreover, being a finite disjoint union of subshifts of finite type (Proposition~\ref{p:irred-comp}.(iv)),
$P$ is itself a subshift of finite type. Set $Y \coloneqq X \setminus P$. It follows from the Cantor-Bendixson decomposition theorem
that $Y$ is countable. As it is closed (by Proposition~\ref{p:isolated-component}) and shift-invariant, $Y$ is a subshift.
In fact, keeping in mind that all infinite irreducible components are isolated, we have that
$Y$ is the vertex subshift associated with $V \setminus \bigcup_{\substack{K \in \KK \\ C_K \in \CC_{\infty}}} K$,
so it is a subshift of finite type as well.
As  $\CBrank(X) \leq 2$,
we have $P = X'' =  (X')' = ((P \sqcup Y)')' = (P \sqcup Y')' = P \sqcup Y''$.
We deduce that $Y'' = \varnothing$, so that $\CBrank(Y) \leq 2$.
This shows that (d) implies (e).
\par
Conversely, assume (e).
Then $X = Z \sqcup Y$,   
where $Z \subset \NW(X)$ and $Y \coloneqq X \setminus Z$ are  subshifts of finite type with $Y$ countable and $\CBrank(Y) \leq 2$, respectively. 
Since $Y$ is countable with $\CBrank(Y) \leq 2$, we have $Y'' = \varnothing$.
On the other hand, for $C \in \CC$, we have $C' = C$ if $C \in \CC_{\infty}$ and $C' = \varnothing$, otherwise.
Consequently, $Z' = \bigcup_{\substack{C \in \CC \\ C \subset Z}} C' = \bigcup_{\substack{C \in \CC_{\infty}\\ C \subset Z}} C$
and therefore $Z'' = Z'$. We deduce that $X'' = Z'' \cup Y'' = Z''$, so that $\CBrank(X) \leq 2$, and $P = X'' = Z'' = Z'$.
We claim that $P = \bigcup_{C \in \CC_{\infty, {\rm i}}} C$. Indeed, suppose that there exists $C \in \CC_{\infty}$ which
is not isolated. Then there exists $C' \in \CC$ such that either $C' \prec C$ or $C \prec C'$.
Accordingly, either $X_{C',C}$ or $X_{C,C'}$ is contained in $X \setminus \NW(X) \subset Y$.
Since $C$ is infinite, the corresponding set is uncountable, contradicting the countability of $Y$.
Thus every irreducible component $C \in \CC_\infty$ is
isolated. Moreover, every irreducible component $C \in \CC_{\infty}$ is uncountable and since $Y$ is countable, one has $C \in Z$.
We deduce that $P = \bigcup_{\substack{C \in \CC_{\infty}\\ C \subset Z}} C = \bigcup_{\CC_{\infty, {\rm i}}} C$ and the claim follows.
\end{proof}

\subsection{Proof of Theorem~\ref{t:finite-extremal-components}}
\label{sec:finite-extremal-components}
We shall prove the equivalence of the three conditions stated in Theorem~\ref{t:finite-extremal-components} by successively showing
(a) $\iff$ (b) and  (b) $\iff$ (c).

\begin{proof}[Proof of (a) $\iff$ (b) in Theorem~\ref{t:finite-extremal-components}]
Suppose that  $X$ is injunctive.
By Proposition~\ref{p:not-injunctive-if-inf-irred}, every irreducible component of $X$ is finite.
As $X$ is Moore since it is injunctive, we then deduce from  Lemma~\ref{l:finite-compo-not-moore}  that every irreducible component of $X$ is  extremal.
This shows (a) $\implies$ (b).
\par
Conversely, suppose that $X$ satisfies (b), i.e., every irreducible component of $X$ is finite and extremal.
Let $\tau \colon X \to X$ be a surjective endomorphism of $X$.
We want to show that $\tau$ is injective.
Let $\CC$ denote the set of irreducible components of $X$.
  Up to replacing $\tau$ by $\tau^m$ for $m \coloneqq |\CC| !$, it follows from Proposition~\ref{p:surj-tau-implies-surj-rho} that we may assume that $\tau(C) = C$ for every $C \in \CC$.
  Let $x,y \in X$ such that $\tau(x) = \tau(y)$.
  If $x \in C$ for some $C \in \CC$, then $\tau(y) = \tau(x) \in C$ and hence $y \in C$.
  As $C$ is finite and $\tau$ is surjective, the restriction of $\tau$ to $C$ is injective, so that $x = y$.
Otherwise,  setting $C \coloneqq C_\alpha(x), C' \coloneqq C_\omega(x) \in \CC$, 
we have $C \prec C'$ and $x \in X_{C,C'}$.
  Observe that $\tau$ induces, by restriction, a surjective map from $X_{C,C'}$ onto itself
(cf.~Proposition~\ref{p:ca-sft-irred-compo}.(ii)).
As $\tau$ commutes with the shift, it  induces a map $\eta \colon \OO_{C,C'} \to \OO_{C,C'}$.
We claim that the map $\eta$ is also surjective.
Indeed, if  $O \in \OO_{C,C'}$ then $O = \OO(z)$ for some $z \in X_{C,C'}$.
If $z' \in X_{C,C'}$ is such that $z = \tau(z')$,  then $O = \eta(O')$ for $O' \coloneqq \OO(z')$.
This proves our claim.
As $\OO_{C,C'}$ is finite by Lemma~\ref{l:W-C-C'}, the map $\eta$ must be injective.
As $\tau(y) = \tau(x) \in X_{C,C'}$, we have $y \in X_{C,C'}$ and $\eta(\OO(y)) = \eta(\OO(x))$.  
Thus, we have $\OO(y) = \OO(x)$. 
We deduce that there exists $n \in \Z$ such that $y = \sigma^n(x)$.
This implies that $\tau(x) = \tau(y) = \tau(\sigma^n(x)) = \sigma^n(\tau(x))$. 
As $\tau(x) \in X_{C,C'}$ is wandering and therefore not periodic, we get $n = 0$, so that  $y = x$.
Therefore $\tau$ is injective. 
This shows that $X$ is injunctive and completes the proof of (a) $\iff$ (b).
\end{proof}

\begin{proof}[Proof of (b) $\iff$ (c) in Theorem~\ref{t:finite-extremal-components}]
Suppose (b).
Denoting by $\CC'$ the set of irreducible components of $X$ that are not isolated, 
we then deduce from Proposition~\ref{p:isolated-configuration} that
 $X$ has  derivatives
$X' = \bigcup_{C \in \CC'} C$ and $X'' = \varnothing$.
Therefore, $X$ is countable and $\CBrank(X) \leq 2$.
This shows that (b) implies (c).
\par
Conversely, suppose (c).
As $X$ is countable, we have $X^{(\CBrank(X))} = \varnothing$.
Since $\CBrank(X) \leq 2$,  we deduce that  $X'' = \varnothing$.
By compactness, this implies that $X'$ is  finite.
As $X'$ is $\sigma$-invariant,  every configuration in $X'$ is periodic.
On the other hand, it follows from Proposition~\ref{p:isolated-configuration} that  the set $X \setminus X'$ of isolated configurations in $X$ 
consists of all $x \in X$ such that $C_\alpha(x)$ and $C_\omega(x)$ are both finite, with $C_\alpha(x)$ minimal and $C_\omega(x)$ maximal.
Therefore, $X \setminus X'$ is the disjoint union of the set of finite extremal components of $X$ and
of the subset  $Y \subset X$ consisting of all $x \in X$ such that $C_\alpha(x) \prec C_\omega(x)$ with $C_\alpha(x)$ and $C_\omega(x)$ both finite, $C_\alpha(x)$ minimal, and $C_\omega(x)$ maximal.
We deduce that every irreducible component of $X$ is finite and extremal.
This shows that (c) implies (b).   
\end{proof}

\section{Examples}
\label{sec:examples}

It follows from our results 
that if $A$ is a finite set and $X \subset A^{\Z}$ is a subshift of finite type then $X$ is of one of the following types:
\begin{enumerate}[Type 1:]
\item
$X$ is surjunctive,  injunctive, Moore, and  Myhill; 
\item
$X$ is surjunctive, injunctive, and Moore, but not Myhill; 
\item
$X$ is surjunctive, Moore, and Myhill, but not injunctive;
\item
$X$ is surjunctive and Moore, but neither injunctive nor Myhill;
\item
$X$ is neither surjunctive, nor injunctive, nor Moore, nor Myhill.
\end{enumerate}
The examples of subshifts of finite type below show that each of these five types can occur.

\begin{example}[Type 1]
\label{ex:type-1}
Let $A \coloneqq \{0\}$ and let $X \coloneqq A^{\Z} = \{0^{\Z}\}$ denote the full shift over $A$.
Observe that $X$ is a subshift of finite type admitting $\varnothing$ as a defining set of forbidden words. 
The identity map $\Id_X$ is the unique  endomorphism of $X$.
As $\Id_X$ is bijective,
$X$ is of Type~1.
Note that $X$ is irreducible.
The perfect kernel of $X$ is the empty set and we have $\CBrank(X) = 1$. 
\end{example}

\begin{example}[Type 1]
\label{ex:type-1b}
Let $A \coloneqq \{0,1\}$ and consider the subshift $X \coloneqq \{(01)^{\Z},(10)^{\Z}\} \subset A^{\Z}$ 
consisting of the two periodic configurations of minimal period $2$.
Observe that $X$ is a subshift of finite type admitting $\{00,11\}$ as a defining set of forbidden words.
On the other hand, the endomorphisms of $X$ are its two permutations.
As every endomorphism of $X$ is bijective,
 $X$ is surjunctive, injunctive, Moore, and Myhill.
Note that $X$ is irreducible. 
The perfect kernel of $X$ is the empty set and we have $\CBrank(X) = 1$. 
 \end{example} 

\begin{example}[Type 2]
\label{ex:type-2}
Let $A \coloneqq \{0,1\}$ and consider the subshift  $X \coloneqq \{0^{\Z},1^{\Z}\} \subset A^{\Z}$ consisting of the two constant configurations.
Observe that $X$ is a subshift of finite type admitting $\{01,10\}$ as a defining set of forbidden words.
As $X$ is finite, it is surjunctive, injunctive, and it  has the Moore property.
On the other hand, the endomorphism  $\tau \colon X \to X$ defined by $\tau(0^{\Z}) = \tau(1^{\Z}) = 0^{\Z}$ is pre-injective but not surjective.
Consequently, $X$ does not have the Myhill property.
Note that $X$ is non-wandering. 
Its irreducible components are $\{0^{\Z}\}$ and $\{1^{\Z}\}$,
and they  are both isolated.
The perfect kernel of $X$ is the empty set and we have $\CBrank(X) = 1$. 
 \end{example} 

\begin{remark}
The subshifts described in the two previous examples have homeomorphic underlying topological spaces.
As one of these subshifts is Myhill while the other is not, this shows that the Myhill property is not a topological invariant for subshifts of finite type.
\end{remark} 
\begin{example}[Type 3]
\label{ex:type-3}
Let $A \coloneqq \{0,1\}$  and let $X \coloneqq  A^{\Z}$ denote the full shift over $A$.
Observe that $X$ is a subshift of finite type admitting $\varnothing$ as a defining set of forbidden words. 
By the original Moore-Myhill Garden of Eden theorem~\cite{moore}, \cite{myhill},
$X$  is both Moore and Myhill.
As $X$ is Myhill, it is surjunctive.
However, $X$ is not injunctive.
This follows from Theorem~\ref{t:finite-extremal-components} since  $X$ is infinite and irreducible.
Actually, an explicit surjective but not injective endomorphism of $X$ is provided by
the map $\tau \colon X \to X$ given by $\tau(x)_i \coloneqq x_i + x_{i + 1} \mod 2$ for all $x \in X$ and $i \in \Z$.
Consequently, $X$ is of Type~3.
Note that $X$ is irreducible.
The perfect kernel of $X$ is $X$ itself and we have $\CBrank(X) = 0$. 
  \end{example}

\begin{example}[Type 4]
\label{ex:type-4}
Let $A \coloneqq \{0,1,2\}$ and consider the subshift $X \coloneqq \{0,1\}^{\Z} \cup \{2^{\Z}\}$.
 Observe that $X$ is a subshift of finite type admitting $\{02,20,12,21\}$ as a defining set of forbidden words.
 It has two irreducible components, namely $\{0,1\}^{\Z}$ and $\{2^{\Z}\}$.
 These two irreducible components are isolated.
We thus deduce from Theorem~\ref{t:char-surj-sft} that $X$ is surjunctive and has the Moore property.
On the other hand, it follows from Theorem~\ref{t:finite-extremal-components} that $X$ is not injunctive.
Actually, an explicit surjective but not injective endomorphism of $X$ is provided by
the map $\tau \colon X \to X$ given by $\tau(x)_i \coloneqq x_i + x_{i + 1} \mod 2$ for all $x \in \{0,1\}^{\Z}$ and $i \in \Z$, and $\tau(2^{\Z}) = 2^{\Z}$.
Moreover, $X$ is not Myhill since the map $\tau \colon X \to X$ given by
$\tau(x) \coloneqq x$ for all $x \in \{0,1\}^{\Z}$ and $\tau(2^{\Z}) \coloneqq 0^{\Z}$ is a pre-injective endomorphism of $X$ which is not surjective.
Consequently, $X$ is of Type~4.
The perfect kernel of $X$ is $\{0,1\}^{\Z}$ and we have $\CBrank(X) = 1$.
\end{example}

\begin{example}[Type 5]
\label{ex:type-5}
Let $A \coloneqq \{0,1,2\}$ and consider the subshift of finite type $X \subset A^{\Z}$ admitting
$\{10,20,21\}$ as a defining set of forbidden words.
 Observe that the word $12$ can appear at most once in a given configuration of $X$.
The map $\tau \colon X \to X$ which replaces the word $12$  by the word $11$ is an injective endomorphism that is not surjective.
Therefore, $X$ is not surjunctive. In particular, it does not have the Myhill property.
On the other hand, the map $\tau' \colon X \to X$ which replaces the word $112$ by the word $122$ is a surjective endomorphism which is not pre-injective
(the configurations $\dots00.1222\dots$ and $\dots00.1122\dots$ are homoclinic and have the same image under $\tau'$).  
Thus, $X$  does not have the Moore property. It follows that  $X$ is not injunctive either.
Consequently, $X$ is of Type~5.
Note that $X$ fails to be non-wandering.
It has three irreducible components, namely, $\{0^{\Z}\}$, $\{1^{\Z}\}$, and $\{2^{\Z}\}$, and they satisfy  
$\{0^{\Z}\} \prec  \{1^{\Z}\} \prec \{2^{\Z}\}$.
The perfect kernel of $X$ is the empty set  and we have $\CBrank(X) = 3$.
\end{example}

Here is another example of a subshift of finite type of Type 5.
Contrary to the previous example, this one is uncountable.

\begin{example}[Type 5]
\label{ex:type-5-bis}
Let $A \coloneqq \{0,1,2\}$ and consider the subshift of finite type $X \subset A^{\Z}$ admitting
$\{02,20,21\}$ as a defining set of forbidden words.
 Observe that the word $12$ can appear at most once in a given configuration $x \in X$.
The map $\tau \colon X \to X$ which replaces the word $12$  by the word $11$ is an injective endomorphism that is not surjective
(the word $012$ appears in $X$ but not in $\tau(X)$).
Therefore, $X$ is not surjunctive. In particular, it does not have the Myhill property.
On the other hand, the map $\tau' \colon X \to X$ which replaces the word $112$ by the word $122$ is a surjective endomorphism which is not pre-injective
(the configurations $\dots00.1222\dots$ and $\dots00.1122\dots$ are homoclinic and have the same image under $\tau'$).  
Thus, $X$  does not have the Moore property. It follows that  $X$ is not injunctive either.
Consequently, $X$ is of Type~5.
Note that $X$ fails to be non-wandering.
It has two irreducible components, namely, $\{0,1\}^{\Z}$ and  $\{2^{\Z}\}$, and they satisfy  
$\{0,1\}^{\Z}  \prec \{2^{\Z}\}$.
The perfect kernel of $X$ is $X$ itself  and we have $\CBrank(X) = 0$.
\end{example}

\begin{remark}
The subshift  described in the previous example and the full shift $\{0,1\}^{\Z}$ of Example~\ref{ex:type-3} have  underlying topological spaces 
that are homeomorphic (they are both non-empty perfect compact metrizable spaces).
We deduce that surjunctivity (resp.~Moore) is not a topological invariant for subshifts of finite type.
\end{remark}

We describe now an example showing that injunctivity is not a topological invariant for sofic subshifts.

\begin{example}
\label{ex:sofic-injunctive}
Let $A \coloneqq \{0,1,2,3,4,5\}$.
For each $k \in \{0,1,2\}$, consider the subshift $X^{(k)} \subset A^{\Z}$ defined by  
\[
X^{(k)} \coloneqq \{x \in \{2k,2k+1\}^{\Z}: \vert \{i \in \Z: x_i = 2k+1\}\vert \leq 2\}.
\]
Then the subshift $Z \coloneqq X^{(0)} \sqcup X^{(1)} \sqcup X^{(2)}$ is sofic (Figure \ref{fig:sofic-Z} 
describes a finite labeled graph presenting $Z$). Observe that $Z'' = \{0^{\Z}, 2^{\Z}, 4^{\Z}\}$.
One easily checks that $Z$ is injunctive (a surjective endomorphism of $Z$ is also injective since it 
must permute the three constant configurations).
\par
\begin{figure}[htbp]
  \centering
\centering
  \begin{tikzpicture}[
    node distance = 12mm and 12mm,
state/.append style = {circle, draw=black, fill=black,
                     minimum size=0.8em},
every edge/.style = {draw, -Stealth, semithick}
                      ]
  \begin{scope}[nodes=state]
\node (n1) {} ;
\node (n2)  [right=of n1]   {};
\node (n3)  [right=of n2]   {};
\node (n4) [right=of n3]{} ;
\node (n5)  [right=of n4]   {};
\node (n6)  [right=of n5]   {};
\node (n7)  [right=of n6]   {};
\node (n8)  [right=of n7]   {};
\node (n9)  [right=of n8]   {};
  \end{scope}
\path   (n1) edge node[above] {$1$} (n2)
     (n2) edge  node[above] {$1$} (n3)
  (n1) edge [loop above, looseness=20] node {$0$} (n1)
     (n2) edge [loop above,looseness=20] node {$0$} (n2)
     (n3) edge [loop above,looseness=20] node {$0$} (n3) ;
     \path   (n4) edge node[above] {$3$} (n5)
     (n5) edge  node[above] {$3$} (n6)
     (n4) edge [loop above,looseness=20] node {$2$} (n4)
     (n5) edge [loop above,looseness=20] node {$2$} (n5)
     (n6) edge [loop above,looseness=20] node {$2$} (n6) ;
\path   (n7) edge  node[above] {$5$} (n8)
     (n8) edge node[above] {$5$} (n9)
     (n7) edge [loop above,looseness=20] node {$4$} (n7)
     (n8) edge [loop above,looseness=20]  node {$4$} (n8)
     (n9) edge [loop above,looseness=20] node {$4$} (n9) ;
  \end{tikzpicture}
 \caption{A finite labeled graph presenting the sofic subshift $Z$.}
  \label{fig:sofic-Z}
\end{figure}
\par
On the other hand, the underlying topological space of $Z$ is homeomorphic to the underlying topological 
space of the subshift of finite type $X$ of Example~\ref{ex:type-5}.
Indeed, they are both countable with Mazurkiewicz-Sierpi\'nski characteristic $(2,3)$.
As $Z$ is surjunctive while $X$ is not, this shows that injunctivity is not a topological invariant for sofic subshifts.
\end{example}

\section{Open questions}

It is shown in \cite{csc-goe-smale} that non-wandering Smale systems are surjunctive.  
The characterization of surjunctivity remains open for general Smale systems. 

\begin{question}
Is every  surjunctive Smale system  topologically conjugate to the disjoint union of a non-wandering Smale system and a countable subshift of finite type of Cantor-Bendixson rank at most $2$?
\end{question}

We note that a positive  answer to this question  would imply the well-known  conjecture that every Anosov diffeomorphism on a compact manifold is non-wandering (see, e.g.,~\cite{brin-nonwandering}). 
Indeed, every self-embedding of a compact manifold is surjective (by Brouwer's invariance of domain), 
and a non-empty open subset of a manifold cannot be countable.

Similarly, one may ask for the characterization of the Moore, Myhill and injunctivity properties of general Smale systems. 
It is reasonable to conjecture that general Smale systems satisfy similar properties as
subshifts of finite type, that the Myhill property still implies the Moore property,
and that the Moore property is equivalent to surjunctivity.

For subshifts of finite type, the Myhill property remains without a clean characterization.

\begin{problem}
Give a characterization of the Myhill property for $\Z$-subshifts of finite type.
\end{problem}

As shown by Fiorenzi, the even subshift is Myhill (and therefore surjunctive) but not Moore \cite{fiorenzi-sofic}, which already shows that sofic subshifts behave quite differently from subshifts of finite type.

\begin{question}
Which sofic subshifts are surjunctive (resp.~injunctive, resp.~Moore, resp.~Myhill)?
\end{question}

It is possible to construct compact dynamical $\Z$-systems that are injunctive but not surjunctive. 
For subshifts, we do not know any such examples.

\begin{question}
Does there exist a  subshift with finite alphabet over $\Z$ which is injunctive but not surjunctive?
\end{question}

\bibliographystyle{siam}

\begin{thebibliography}{10}

\bibitem{auslander-endomorphisms}
{\sc J.~Auslander}, {\em Endomorphisms of minimal sets}, Duke Math. J., 30
  (1963), pp.~605--614.

\bibitem{ax-elementary}
{\sc J.~Ax}, {\em The elementary theory of finite fields}, Ann. of Math. (2),
  88 (1968), pp.~239--271.

\bibitem{boyle-lind-rudolph}
{\sc M.~Boyle, D.~Lind, and D.~Rudolph}, {\em The automorphism group of a shift
  of finite type}, Trans. Amer. Math. Soc., 306 (1988), pp.~71--114.

\bibitem{brin-stuck}
{\sc M.~Brin and G.~Stuck}, {\em Introduction to dynamical systems}, Cambridge
  University Press, Cambridge, 2002.

\bibitem{brin-nonwandering}
{\sc M.~I. Brin}, {\em Nonwandering points of {A}nosov diffeomorphisms}, in
  Dynamical systems, {V}ol. {I}---{W}arsaw, vol.~No. 49 of Ast\'erisque, Soc.
  Math. France, Paris, 1977, pp.~11--18.

\bibitem{csc-cat}
{\sc T.~Ceccherini-Silberstein and M.~Coornaert}, {\em Surjunctivity and
  reversibility of cellular automata over concrete categories}, in Trends in
  harmonic analysis, vol.~3 of Springer INdAM Ser., Springer, Milan, 2013,
  pp.~91--133.

\bibitem{csc-cag2}
\leavevmode\vrule height 2pt depth -1.6pt width 23pt, {\em Cellular automata
  and groups}, Springer Monographs in Mathematics, Springer, Cham, second~ed.,
  [2023] \copyright 2023.

\bibitem{csc-ecag}
\leavevmode\vrule height 2pt depth -1.6pt width 23pt, {\em Exercises in
  cellular automata and groups}, Springer Monographs in Mathematics, Springer,
  Cham, [2023] \copyright 2023.
\newblock With a foreword by Rostislav I. Grigorchuk.

\bibitem{csc-goe-smale}
\leavevmode\vrule height 2pt depth -1.6pt width 23pt, {\em A garden of {E}den
  theorem for {S}male spaces}, Mosc. Math. J., 25 (2025), pp.~479--494.

\bibitem{downarowicz}
{\sc T.~Downarowicz}, {\em The royal couple conceals their mutual relationship:
  a noncoalescent {T}oeplitz flow}, Israel J. Math., 97 (1997), pp.~239--251.

\bibitem{fiorenzi-sofic}
{\sc F.~Fiorenzi}, {\em The {G}arden of {E}den theorem for sofic shifts}, Pure
  Math. Appl., 11 (2000), pp.~471--484.

\bibitem{gottschalk}
{\sc W.~Gottschalk}, {\em Some general dynamical notions}, in Recent advances
  in topological dynamics ({P}roc. {C}onf. {T}opological {D}ynamics, {Y}ale
  {U}niv., {N}ew {H}aven, {C}onn., 1972; in honor of {G}ustav {A}rnold
  {H}edlund), Springer, Berlin, 1973, pp.~120--125. Lecture Notes in Math.,
  Vol. 318.

\bibitem{gromov-esav}
{\sc M.~Gromov}, {\em Endomorphisms of symbolic algebraic varieties}, J. Eur.
  Math. Soc. (JEMS), 1 (1999), pp.~109--197.

\bibitem{hedlund}
{\sc G.~A. Hedlund}, {\em Endomorphisms and automorphisms of the shift
  dynamical system}, Math. Systems Theory, 3 (1969), pp.~320--375.

\bibitem{kechris}
{\sc A.~S. Kechris}, {\em Classical descriptive set theory}, vol.~156 of
  Graduate Texts in Mathematics, Springer-Verlag, New York, 1995.

\bibitem{kim-roush}
{\sc K.~H. Kim and F.~W. Roush}, {\em On the automorphism groups of subshifts},
  Pure Math. Appl. Ser. B, 1 (1990), pp.~203--230.

\bibitem{kitchens}
{\sc B.~P. Kitchens}, {\em Symbolic dynamics}, Universitext, Springer-Verlag,
  Berlin, 1998.
\newblock One-sided, two-sided and countable state Markov shifts.

\bibitem{lind-marcus-second}
{\sc D.~Lind and B.~Marcus}, {\em An introduction to symbolic dynamics and
  coding}, Cambridge Mathematical Library, Cambridge University Press,
  Cambridge, second~ed., 2021.

\bibitem{mazurkiewicz-sierpinski}
{\sc S.~Mazurkiewicz and W.~Sierpi{\'n}ski}, {\em A contribution to the
  topology of countable sets.}, Fundam. Math., 1 (1920), pp.~17--27.

\bibitem{moore}
{\sc E.~F. Moore}, {\em Machine models of self-reproduction}, vol.~14 of Proc.
  Symp. Appl. Math., American Mathematical Society, Providence, 1963,
  pp.~17--34.

\bibitem{myhill}
{\sc J.~Myhill}, {\em The converse of {M}oore's {G}arden-of-{E}den theorem},
  Proc. Amer. Math. Soc., 14 (1963), pp.~685--686.

\bibitem{salo-torma_2015}
{\sc V.~Salo and I.~T\"orm\"a}, {\em Category theory of symbolic dynamics},
  Theoretical Computer Science, 567 (2015), pp.~21--45.

\bibitem{walters-intro-book}
{\sc P.~Walters}, {\em An introduction to ergodic theory}, vol.~79 of Graduate
  Texts in Mathematics, Springer-Verlag, New York-Berlin, 1982.

\bibitem{weiss-sgds}
{\sc B.~Weiss}, {\em Sofic groups and dynamical systems}, Sankhy\=a Ser. A, 62
  (2000), pp.~350--359.
\newblock Ergodic theory and harmonic analysis (Mumbai, 1999).

\end{thebibliography}

\end{document}